\documentclass{amsart}
\usepackage{graphicx}
\usepackage{amsmath}%in Winedt
\usepackage{amssymb}
\usepackage{amsthm}
\usepackage{esint}
\usepackage{color}

\newtheorem{theorem}{Theorem}[section]
\newtheorem{lem}[theorem]{Lemma}

\theoremstyle{Corollary}
\newtheorem{cor}[theorem]{Corollary}
\newtheorem{prop}[theorem]{Proposition}
\let\a\alpha
\let\b\beta
\let\d\delta

\let\f\varphi
\let\g\gamma

\let\l\lambda

\let\n\nu

\let\th\theta
\let\u\varepsilon

\def\CC{\mathbb C}
\def\RR{\mathbb R}

\def\HH{\mathbb H}

\def\NN{\mathbb N}

\def\hh{~}
\def\ii{\sqrt{-1}}

\def\VV{dV_{\th_0}}

\def\ww{2+\frac{2}{n}}

\def\tt{\int_M}
\def\ss{\sum_{k=1}^m}

\numberwithin{equation}{section}

\begin{document}

\title{Convergence of the CR Yamabe Flow}

\author{Pak Tung Ho}
\address{Department of Mathematics, Sogang University, Seoul
121-742, Korea}

\email{ptho@sogang.ac.kr, paktungho@yahoo.com.hk}

\author{Weimin Sheng}
\address{School of Mathematical Sciences, Zhejiang University, Hangzhou 310027, China.}
\email{weimins@zju.edu.cn}

\author{Kunbo Wang}
\address{School of Mathematical Sciences, Zhejiang University, Hangzhou 310027, China.}
\email{21235005@zju.edu.cn}

\thanks{
The first author was supported by the National Research Foundation of Korea (NRF) grant funded
by the Korea government (MEST) (No.201631023.01).
The second and  the third authors were supported in part by NSF in China Nos. 11131007, 11571304 and by Zhejiang Provincial NSF in China, No. LY14A010019.}

%    General info

\date{February 13, 2017}


\begin{abstract}
We consider the CR Yamabe flow  on a  compact
strictly pseudoconvex CR manifold $M$ of real dimension $2n+1$.
We prove convergence of the CR Yamabe flow
when $n=1$ or $M$ is spherical.
\end{abstract}

\maketitle

\section{Introduction}

Suppose $(M,g_0)$ is a compact $n$-dimensional manifold without boundary where $n\geq 3$.
As a generalization of Uniformization theorem, the Yamabe problem \cite{Yamabe} is
to find a metric $g$ conformal to $g_0$ such that its scalar curvature $R_g$ is constant. This problem was solved by Yamabe, Trudinger, Aubin and Schoen
 \cite{Yamabe, Trudinger, Aubin0, Schoen}. See the survey article \cite{Lee&Parker} by Lee and Parker
for more details.

A different approach has been introduced to solve the Yamabe problem. Hamitlon \cite{Hamilton} introduced the Yamabe flow, which is defined by
$$\frac{\partial}{\partial t}g(t)=-(R_{g(t)}-r_{g(t)})g(t),$$
where $r_{g(t)}$ is the average of the scalar curvature $R_{g(t)}$ of $g(t)$.
The Yamabe flow was considered by Chow \cite{Chow},
Ye \cite{Ye},
Schwetlick and Struwe \cite{Schwetlick&Struwe}. Finally Brendle
\cite{Brendle4,Brendle5} showed that the Yamabe flow exists for all
time and converges to a metric of constant scalar curvature by using the positive mass theorem.

The Yamabe problem can also be formulated in the context of CR manifold.
Suppose that $(M,\theta_0)$ is a compact strongly pseudoconvex CR manifold of real dimension $2n+1$ with a given contact form $\theta_0$. The CR Yamabe problem is to find a contact form $\theta$ conformal to $\theta_0$ such that
its Webster scalar curvature $R_\theta$ is constant. This was introduced by Jerison and Lee in \cite{Jerison&Lee3}, and was solved by Jerison and Lee
for the case when $n\geq 2$ and $M$ is not spherical in \cite{Jerison&Lee1,Jerison&Lee2,Jerison&Lee3}.
We say that $M$ is spherical if and only if $M$ is locally CR equivalent to the CR sphere $S^{2n+1}$.
The remaining case, namely, when $n=1$ or  $M$ is spherical,  was solved by Gamara and Yacoub in \cite{Gamara2,Gamara1} by using critical point at infinity.
See also the recent work by Cheng-Chiu-Yang \cite{Cheng&Chiu&Yang} and Cheng-Malchiodi-Yang \cite{Cheng&Malchiodi&Yang} for these cases.

As an analogue to Yamabe flow, one can consider the CR Yamabe flow defined by
$$\frac{\partial}{\partial t}\theta(t)=-(R_{\theta(t)}-r_{\theta(t)})\theta(t),$$
where $r_{\theta(t)}$ is the average of the Webster scalar curvature $R_{\theta(t)}$ of $\theta(t)$.
Cheng and Chang \cite{Chang&Cheng} proved the short time existence of the CR Yamabe flow.
Zhang \cite{Zhang} proved the long time existence and convergence
of the CR Yamabe flow for the case
$Y(M,\theta_0)<0$  (see also \cite{Ho2} for the proof when $n=1$).
For the case $Y(M,\theta_0)>0$, Chang-Chiu-Wu \cite{Chang&Chiu&Wu} proved the convergence of the CR Yamabe flow when $M$ is spherical and $n=1$ and $\theta_0$ is torsion-free.
For general $n$,  the long time existence was proved by the first author in \cite{Ho} for the case when $Y(M,\theta_0)>0$.
In \cite{Ho}, the first author also proved
the convergence of the CR Yamabe flow when $M$ is
the CR sphere (see \cite{Ho1} for
an alternative proof).

In this paper, we prove the following convergence result of the CR Yamabe flow
when $n=1$ or $M$ is spherical.

\begin{theorem}\label{main}
Let $(M,\theta_0)$ be a compact strongly pseudoconvex CR manifold of real dimension $2n+1$ and $Y(M,\theta_0)>0$.
Suppose that $n=1$ or $M$ is spherical
such that $M$ is not CR equivalent to the CR sphere $S^{2n+1}$.
When $n = 1$, we also assume
that the CR Panetiz operator of $M$ is nonnegative.
When $n=2$, we also assume that the minimum exponent of the integrability of
the Green function satisfies
$s(M)<1$. Then
the Yamabe flow exists for all time and converges to a contact form with constant
Webster scalar curvature.
\end{theorem}

We refer the readers to (3.6) in \cite{Cheng&Chiu&Yang} for the precise definition of $s(M)$.
When $n=1$, the definition of the CR Panetiz operator  was included in section \ref{section2}, but
we also refer the readers to \cite{Cheng&Malchiodi&Yang} for more properties of
the CR Panetiz operator.
Similar to the Yamabe flow, we need to use the CR positive mass theorem to prove the convergence
of the CR Yamabe flow. The CR positive mass theorem when $M$ is spherical was obtained
by Cheng-Chiu-Yang in \cite{Cheng&Chiu&Yang} for $n\geq 2$, with further assumption
that $s(M)<1$ for $n=2$. On the other hand,
the CR positive mass theorem for the case when $n = 1$ was obtained by
Cheng-Malchiodi-Yang in \cite{Cheng&Malchiodi&Yang}, under the assumption that the CR Panetiz
operator of $M$ is nonnegative.
Therefore, we have included the assumptions in Theorem \ref{main} so that we
can apply the CR positive mass theorem.
%%%%%%%%%%%%%%%%%%%%%%%%%%%%%%

Our proof of Theorem \ref{main} basically follows \cite{Brendle4}. We sketch the proof here. Using the concentration-compactness result (see Theorem \ref{Theorem4.1}), which is the CR version of the results of Struwe \cite{Struwe} and Bahri-Coron \cite{Bahri&Coron}, we can show that the solution of the CR Yamabe flow either converges, or else concentrates in finite number of bubbles. Concisely, if we write $\theta(t)=u(t)^{\frac{2}{n}}\theta_0$ for some positive function $u(t)$, then the CR Yamabe flow
can be written as follows: Let $\{t_\nu: \nu\in\mathbb{N}\}$ be a sequence of times such that $t_\nu\to\infty$ as $\nu\to\infty$ and $u_\nu=u(t_\nu)$. Then there exists an integer $m\geq 0$, a collection of positive numbers $\{\u_{i,\n}:1\leq i \leq m, \, \n\in\NN\}$, and a collection of points $\{p_{i,\n}:1\leq i\leq m, \hh\n\in\NN\}\subset M$, such that
$$u_{\n}-\ss\left(\frac{n(2n+2)}{r_{\infty}}\right)^{\frac{n}{2}}\cdot \left[\frac{\u_{i,\n}^2}{(t^2+(\u_{i,\n}^2+|z|^2)^2)}\right]^{\frac{n}{2}}\rightarrow u_{\infty}.$$
Here we denote $(z,t)$  the CR normal coordinate at  $p_{i,\n}$, and
$r_\infty=\lim_{t\to\infty}r_{\theta(t)}$.
Then we divide the proof into two cases, namely, $u_{\infty}\equiv0$ and $u_{\infty}>0$.
Following the estimates in \cite{Brendle4}, we can rule out the formation of bubbles and show that $u(t)$ is uniformly bounded from above and below. By using the estimate of Bramanti and Brandolini \cite{Bramanti}, we can obtain the uniform bounds for the higher-order derivatives of $u(t)$, which implies the convergence of the CR Yamabe flow.

This paper is organized as follows. In section \ref{section2}, we recall some basic concepts in CR geometry. In section \ref{section3}, we recall some basic propertities of the CR Yamabe flow. In section \ref{section4}, by assuming Proposition \ref{Prop3.3}, we give the proof of Theorem \ref{main}. We give the blow-up analysis in section \ref{section5A}, and then prove Proposition \ref{Prop3.3} in section \ref{section7}. In Appendix \ref{Appendix}, we give some basic estimates regarding the test functions. In Appendix \ref{AppendixB}, we give the proof of the concentration-compactness theorem, i.e. the proof of Theorem \ref{Theorem4.1}.

%%%%%%%%%%%%%%%%%%%%%%%%%%%%%%%%%%%%%
\textit{Acknowledgements.} The first author would like to thank Prof. Paul Yang, who encouraged him to work on the general convergence of the CR Yamabe flow, Prof. Simon Brendle who answered many of his questions, and Prof. Sai-Kee Yeung for helpful discussions.
The second and the third authors would also like to thank Prof. Paul Yang, who brought them into the field of CR geometry since his short course at Zhejiang University in June of 2014.
Part of the work was done when the first author visited Princeton University and Zhejiang University,
and he is grateful for their kind hospitality.

%%%%%%%%%%%%%%%%%%%%%%%%%%%%%%%%
\section{Preliminaries and Notations}\label{section2}

In this section, we include some basic concepts in the CR geometry. Most of them can be found in \cite{Dragomir} or \cite{Lee}.
Let $M$ be an orientable smooth manifold of real dimension $2n+1$. A CR structure on $M$ is given by a complex $n$-dimensional
subbundle $T^{1,0}$ of the complexified tangent bundle ${\mathbb C}TM$ of $M$, satisfying $T^{1,0}\cap T^{0,1}=\{0\}$, where
$T^{0,1}=\overline{T^{1,0}}$. We assume the CR structure is integrable, that is, $T^{1,0}$ satisfies the formal Frobenius condition
$[T^{1,0}, T^{1,0}]\subset T^{1,0}$. We set $G=Re(T^{1,0}\oplus T^{0,1})$, so that $G$ is a real $2n$-dimensional subbundle of $TM$.
Then $G$ carries a natural complex structure map: $J: G\rightarrow G$ given by $J(V+\overline{V})=\ii(V-\overline{V})$ for $V\in T^{1,0}$.

Let $E\subset T^{\ast}M$ denote the real line bundle $G^{\bot}$. Because we assume $M$ is orientable and the complex structure $J$ induces an
orientation on $G$, $E$ has a global nonvanishing section. A choice of such a 1-form $\th$ is called a pseudohermitian structure on $M$. Associated with each such $\theta$ is the real symmetric bilinear form $\mathcal{L}_\theta$ on $G$:
\begin{equation*}
\mathcal{L}_\theta(V,W)=d\theta(V,JW),~~V,W\in G
\end{equation*}
called the Levi-form of $\theta$. $\mathcal{L}_\theta$ extends by complex linearity to $\mathbb{C}G$, and induces a Hermitian form on $T^{1,0}$, which we write
\begin{equation*}
\mathcal{L}_\theta(V,\overline{W})=-\ii d\theta(V,\overline{W}),~~V,W\in T^{1,0}
\end{equation*}
If $\theta$ is replaced by $\tilde\theta=f\theta$, $\mathcal{L}_\theta$ changes conformally by $\mathcal{L}_{\tilde\theta}=f\mathcal{L}_\theta$. We will assume that $M$ is strictly pseudoconvex, that is, $\mathcal{L}_\theta$ is positive definite for a suitable $\theta$. In this case, $\theta$ defines a contact structure on $M$, and we call $\theta$ a contact form. Then we define the volume form on $M$ as $dV_{\theta}=\theta\wedge d\theta^n$.
The volume of $M$  is denoted by $\mbox{Vol}(M,\theta)$, i.e.
$\mbox{Vol}(M,\theta)=\displaystyle\int_MdV_{\theta}$.

We can choose a unique $T$, which is called the characteristic direction, such that $\theta(T)=1$, $d\theta(T, \cdot)=0$, and $TM=G\oplus \mathbb{R}T$. Then we can define a coframe $\{\theta, \theta^1, \theta^2, \cdots, \theta^n\}$ that satisfies $\theta^\alpha(T)=0$, which is called admissible coframe. Its dual frame $\{T, Z_1, Z_2, \cdots, Z_n\}$ is called admissible frame. In the coframe, we have $d\theta=\ii h_{\alpha\bar\beta}\theta^\alpha\wedge\theta^{\bar\beta}$, where $h_{\alpha\bar\beta}$ is a Hermitian matrix.

The sub-Laplacian operator $\Delta_b$ is defined by
$$\int_M (\Delta_b u)fdV_{\th}=-\int_M\langle du,df\rangle_{\th}dV_{\th},$$
for all smooth function $f$. Here $\langle\,,\,\rangle_{\th}$ is the inner product induced by $\mathcal{L}_{\th}$. Tanaka \cite{T} and Webster \cite{W} showed that there is a natural connection on the bundle $T^{1,0}$ adapted to a pseudohermitian structure, which is called the Tanaka-Webster connection. To define this connection, we choose an admissible coframe $\{\th^{\a}\}$ and dual frame $\{Z_{\a}\}$ for $T^{1,0}$. Then there are uniquely determined 1-forms $\omega_{\a\overline{\b}}$, $\tau_{\a}$ on $M$, satisfying
\begin{eqnarray*}
d\theta^\alpha &=& \theta^\beta\wedge{\omega_\alpha}^\beta+\theta\wedge\tau^\alpha,\\
dh_{\alpha\bar\beta} &=& \omega_{\alpha\overline{\beta}}+\omega_{\overline{\beta}\alpha},\\
\tau_\alpha\wedge\theta^\alpha &=& 0.
\end{eqnarray*}
From the third equation, we can find $A_{\alpha\gamma}$ such that
$\tau_\alpha=A_{\alpha\gamma}\theta^\gamma$
and $A_{\alpha\gamma}=A_{\gamma\alpha}$. Here $A_{\alpha\gamma}$ is called the pseudohermitian torsion.
With this connection, the covariant differentiation is defined by
$$
D Z_\alpha={\omega_\alpha}^\beta\otimes Z_\beta,~~~~DZ_{\bar\alpha}={\omega_{\bar\alpha}}^{\bar\beta}\otimes Z_{\bar\beta},~~~~DT=0.
$$
$\{{\omega_{\alpha}}^\beta\}$ are called connection 1-forms.
For a smooth function $f$ on $M$, we write $f_\alpha=Z_\alpha f,~~f_{\bar\alpha}=Z_{\bar\alpha} f,~~f_0=Tf$, so that
$df=f_\alpha \theta_\alpha+f_{\bar\alpha} \theta_{\bar\alpha}+f_0 \theta$. The second covariant differential $D^2 f$ is the 2-tensor with components
\begin{equation*}
\begin{split}
f_{\alpha\beta} &=\overline{\overline{f}_{\bar\alpha\bar\beta}}=Z_\beta Z_\alpha f-{\omega_\alpha}^\gamma(Z_\beta) Z_\gamma f, ~~f_{\alpha\bar\beta} =\overline{\overline{f}_{\bar\alpha\beta}}=Z_{\bar\beta} Z_\alpha f-{\omega_\alpha}^\gamma(Z_{\bar\beta}) Z_\gamma f,\\
f_{0\alpha} &=\overline{\overline{f}_{0\bar\alpha}}=Z_\alpha Tf,~~f_{\alpha0}=\overline{\overline{f}_{\bar\alpha 0}}=TZ_\alpha f-{\omega_\alpha}^\gamma(T) Z_\gamma f,~~f_{00}=T^2 f.
\end{split}
\end{equation*}
$h_{\a\bar{\b}}$ and $h^{\a\bar{\b}}$ are used to lower and raise the indices. We have
$$
d{\omega_\beta}^\alpha-{\omega_\beta}^\gamma\wedge{\omega_\gamma}^\alpha=\frac{1}{2}R_{\beta~~\rho\sigma}^{~~\alpha}\theta^\rho\wedge\theta^{\sigma}+
\frac{1}{2}R_{\beta~~\bar\rho\bar\sigma}^{~~\alpha}\theta^{\bar\rho}\wedge\theta^{\bar\sigma}
+R_{\beta~~\rho\bar\sigma}^{~~\alpha}\theta^\rho\wedge\theta^{\bar\sigma}+R_{\beta~~\rho 0}^{~~\alpha}\theta^\rho\wedge\theta-R_{\beta~~\bar\sigma 0}^{~~\alpha}\theta^{\bar\sigma}\wedge\theta.
$$
We call $R_{\beta\bar\alpha\rho\bar\sigma}$ the pseudohermitian curvature. Contractions of the pseudohermitian curvature yield the pseudohermitian Ricci curvature $R_{\rho\bar\sigma}=R_{\alpha~~\rho\bar\sigma}^{~~\alpha}$, or $R_{\rho\bar\sigma}=h^{\alpha\bar\beta}R_{\alpha\bar\beta\rho\bar\sigma}$, and the pseudohermitian scalar curvature $R=h^{\rho\bar\sigma}R_{\rho\bar\sigma}$.

The sub-Laplacian operator in this connection can be expressed by
\begin{equation*}
\Delta_b u={u_\alpha}^\alpha+{u_{\bar\alpha}}^{\bar\alpha}
\end{equation*}
If we define $\widetilde{\th}=u^{\frac{2}{n}}\th$, then we have
$$\widetilde{\Delta}_b f=u^{-(1+\frac{2}{n})}(u\Delta_b f+2\langle du,df\rangle_{\th}),$$
where $\widetilde{\Delta}_b$ is the sub-Laplacian operator with respect to the contact form $\widetilde{\th}$ (see (2.4) in  \cite{Ho} for example). If we set $\widetilde{u}=r^{-1}u,$
then we have the following CR transformation law
$$(-(\ww)\widetilde{\Delta}_b +\widetilde{R})\tilde{u}=r^{-1-\frac{2}{n}}(-(\ww)\Delta_b +R)u.$$
In particular, if $r=u$, then we get the CR Yamabe equation
\begin{equation}\label{CRYE}
-(2+\frac{2}{n})\Delta_b u+Ru=\widetilde{R}u^{1+\frac{2}{n}},
\end{equation}

The Heisenberg group $\mathbb H ^n$ is a
Lie group whose underlying manifold is $\mathbb C^n\times \RR$ with coordinates $(z,t)=(z_1,z_2,\cdots,z_n,t)$ and $z=x+\ii y$. The group law of $\HH^n$ is given by
$$(z,t)(z',t')=(z+z',t+t'+2 \, {\rm{Im}}(z\cdot \bar{z'})).$$
The norm on $\mathbb H ^n$ is given by $|(z,t)|=(|z|^4+t^2)^{\frac{1}{4}}$ and $\d_{\l}:\HH^n\mapsto\HH^n$ is the dilation on $\HH^n$ given by
$$\d_{\l}(z,t)=(\l z,\l^2 t),$$
and $\tau_{\xi}:\HH^n\mapsto\HH^n$ is the translation on $\HH^n$ given by
$$\tau_{\xi}(x,y,t)=(x+x',y+y', t+t'+2(xy'-x'y)),\hh\hh \xi=(x', y', t')\in \HH^n.$$
The vector fields $Z_j=\frac{\partial}{\partial z_j} +\ii \overline{z_j}\frac{\partial}{\partial t}$, $j=1,2,...,n$, are invariant with respect to the group multiplication on the left. And $T^{1,0}=\mbox{span}\{Z_1,...,Z_n\}$ gives a left-invariant CR structure on  $\mathbb H ^n$. The contact form of $\mathbb H ^n$ is
$$\th_{\mathbb{H}^n}=dt+\sqrt{-1}\sum_{j=1}^n (z_jd\overline{z}_j-\overline{z}_jdz_j).$$

If $\{W_1,\cdots,W_n\}$ is a frame for $T^{1,0}$ over some open set $U\subset M$ which is orthonormal with respect to the given pseudohermitian structure on $M$, we call $\{W_1,\cdots,W_n\}$ a pseudohermitian frame. And $\{W_1,\cdots,W_n,\overline{W}_1,\cdots, \overline{W}_n, T\}$ forms a local frame for $\mathbb C TM$. Now let $U$ be a relatively compact open subset of a normal coordinate neighborhood, with contact form $\th$ and pseudo-hermitian frame
$\{W_1,\cdots, W_n\}.$ Let $X_j={\rm{Re}} W_j$ and $X_{j+n}={\rm{Im}} W_j$. Denote $X^{\a}=X_{\a_1}\cdots X_{\a_k}$, where $\a=(\a_1,\cdots,\a_k)$, and denote $l(\a)=k$.
Define the norm
$$\|f\|_{S_k^p(U)}=\sup_{l(\a)\leq k}\|X^{\a}f\|_{L^p(U)}.$$
The Folland-Stein space $S_k^p(U)$ is defined as the completion of $C_0^{\infty}$ with respect to the norm $\|\cdot\|_{S_k^p(U)}$.
(c.f. \cite{Folland&Stein}).
We have the following Folland-Stein embedding theorem, which is the CR version of Sobolev embedding theorem.

\begin{prop}{\rm{(\cite{Jerison&Lee3})}}\label{Theorem 2.1}
For $\frac{1}{s}=\frac{1}{r}-\frac{k}{2n+2}$, where $1<r<s<\infty$. Then we have
$$S_k^r(M)\subset L^s(M).$$
\end{prop}

%%%%%%%%%%%%%%%%%%%%%%%%%%%%%%%%

When $n=1$, we define the CR Paneitz operator $P$ by
$$P\f=4(\f_{\bar{1}\hh 1}^{\hh \bar{1}}+\ii A_{11}\f^1)^1.$$
We say $P$ is non-negative if
$$\tt \f P \f \VV\geq 0$$
for all real smooth functions $\f$.
The positivity of the CR Paneitz operator
and the conformal sub-Laplacian
guarantees that a compact three-dimensional CR manifold can be embedded into $\CC^{n}$ for some integer $n$ (see \cite{CCY}).
At the end of this section, we introduce the concept of Carnot-Carath\'{e}dory distance $d(\cdot,\cdot)$ between any two points $p, q\in M$.
A piecewise smooth curve $\g:[0,1]\rightarrow M$ is said to be a Legendrian curve if $\g'(t)\in G$ whenever $\g'(t)$ exists. And
$$l(\g)=\int_0^1 h(\g'(t),\g'(t))^{\frac{1}{2}}dt,$$
where $h(X,Y)=d\th(X,JY)$. We denote $C_{p,q}$ the set of all Legendrian curves which join $p$ and $q$. Then the Carnot-Carath\'{e}dory distance is defined as
$$d(p,q)=\inf\{l(\g):\g\in C_{p,q}\}.$$

%%%%%%%%%%%%%%%%%%%%%%%%%%%%%%%%%%
\section{The CR Yamabe flow}\label{section3}

In this section, we recall the definition and some basic facts
of the CR Yamabe flow.
Throughout this paper, we assume that $(M,\theta_0)$ is a compact
strictly pseudoconvex CR manifold of real dimension $2n+1$ with a
contact form $\theta_0$.
Hereafter, we denote $R_{\theta_0}$ the
Webster scalar curvature with respect to the contact form
$\theta_0$. The CR Yamabe constant of $\theta_0$ is defined as
\begin{equation}\label{12}
Y(M,\theta_0)=\inf_{u\in C^\infty(M), u>0}\frac{\int_M\big((2+\frac{2}{n})|\nabla_{\theta_0}u|^2_{\theta_0}+R_{\theta_0}u^2\big)dV_{\theta_0}}
{\big(\int_Mu^{2+\frac{2}{n}}dV_{\theta_0}\big)^{\frac{n}{n+1}}}.
\end{equation}
The CR Yamabe flow is defined by
\begin{equation}\label{0}
\frac{\partial}{\partial
t}\theta(t)=-(R_{\theta(t)}-r_{\theta(t)})\theta(t)\mbox{ for }t\geq 0,\hspace{2mm}\theta(t)|_{t=0}=\theta_0,
\end{equation}
where
$R_{\theta(t)}$ is the Webster scalar curvature with
respect to the contact form
$\theta(t)$, and
$r_{\theta(t)}$ is the average value of the Webster scalar
curvature:
\begin{equation}\label{6}
r_{\theta(t)}=\frac{\int_MR_{\theta(t)}dV_{\theta(t)}}{\int_MdV_{\theta(t)}}.
\end{equation}
If we write $\theta(t)=u(t)^{\frac{2}{n}}\theta_0$ for some function $u(t)$, then the CR Yamabe flow in (\ref{0})
can be written as
\begin{equation}\label{0.5}
\frac{\partial}{\partial
t}u(t)=-\frac{n}{2}(R_{\theta(t)}-r_{\theta(t)})u(t)\mbox{ for }t\geq 0,\hspace{2mm} u(t)|_{t=0}=1.
\end{equation}
Since $\theta(t)=u(t)^{\frac{2}{n}}\theta_0$, it follows from (\ref{CRYE}) that we have the CR Yamabe equation:
\begin{equation}\label{9}
-(2+\frac{2}{n})\Delta_{\theta_0}u(t)+R_{\theta_0}u(t)=R_{\theta(t)}u(t)^{1+\frac{2}{n}}.
\end{equation}
By (\ref{0.5}) and (\ref{9}), the CR Yamabe flow in (\ref{0})
is equivalent to
\begin{equation}\label{10}
\frac{\partial}{\partial t}(u(t)^{\frac{n+2}{n}})=\frac{n+2}{2}\left((2+\frac{2}{n})\Delta_{\theta_0}u(t)-R_{\theta_0}u(t)+r_{\theta(t)}u(t)^{1+\frac{2}{n}}\right),
\end{equation}
which is a weakly parabolic partial differential equation. The short time existence of the CR Yamabe flow was proved
by Chang and Cheng in \cite{Chang&Cheng}.

As we have mentioned above,
the long time existence and convergence
of the CR Yamabe flow
were proved in \cite{Zhang}
when $Y(M,\theta_0)<0$. In this paper, we consider the case when the CR Yamabe constant is positive, i.e. $Y(M,\theta_0)>0$.
In this case, the long time existence has already been proved by the first author in \cite{Ho}. Therefore, to prove Theorem \ref{main},
 we will focus on proving the convergence part.
Since $Y(M,\theta_0)>0$, by choosing another contact form in the same conformal class if necessary, we may assume that $\theta_0$
has positive Webster scalar curvature, i.e.
\begin{equation}\label{11}
R_{\theta_0}>0.
\end{equation}
Without loss of generality, we can choose the initial contact form $\theta_0$ such that
\begin{equation}\label{2}
\mbox{Vol}(M,\theta_0)=\int_MdV_{\theta_0}=1.
\end{equation}
Since the CR Yamabe flow preserves volume (see Proposition 3.1 in \cite{Ho}), it follows from (\ref{2}) that
\begin{equation}\label{3}
\mbox{Vol}(M,\theta(t))=\int_MdV_{\theta(t)}=\int_Mu(t)^{2+\frac{2}{n}}dV_{\theta_0}=1\hspace{2mm}\mbox{ for all }t\geq 0.
\end{equation}
By Proposition 3.3 in \cite{Ho}, the function $t\mapsto r_{\theta(t)}$ is non-increasing. Indeed, we have (see (3.5) in \cite{Ho})
\begin{equation}\label{8}
\frac{d}{dt}r_{\theta(t)}=-n\int_M(R_{\theta(t)}-r_{\theta(t)})^2dV_{\theta(t)}.
\end{equation}
It follows from  (\ref{12}), (\ref{6}), and (\ref{9}) that $r_{\theta(t)}\geq Y(M,\theta_0)$. Hence, the following limit
exists and satisfies:
\begin{equation}\label{5}
r_\infty=\lim_{t\to\infty}r_{\theta(t)}\geq Y(M,\theta_0)>0.
\end{equation}
It was proved in \cite{Ho} that (see Proposition 4.1 and Corollary 4.1 in \cite{Ho})
\begin{equation}\label{4A}
\lim_{t\to\infty}\int_M|R_{\theta(t)}-r_{\theta(t)}|^pdV_{\theta(t)}=0.
\end{equation}
for all $1<p<n+2$, and
\begin{equation}\label{4}
\lim_{t\to\infty}\int_M|R_{\theta(t)}-r_{\infty}|^pdV_{\theta(t)}=0.
\end{equation}
for all $1<p<n+2$.

%%%%%%%%%%%%%%%%%%%%%%%%%%%%%%%%%%

\section{Proof of the main result assuming Proposition \ref{Prop3.3}}\label{section4}

The proof of Theorem \ref{main} will be based on the following proposition.

\begin{prop}\label{Prop3.3}
Let $\{t_\nu: \nu\in\mathbb{N}\}$ be a sequence of times such that $t_\nu\to\infty$ as $\nu\to\infty$. Then we can
find a real number $0<\gamma<1$ and a constant $C$ such that,
after passing to a subsequence, we have
$$r_{\theta(t_\nu)}-r_\infty\leq C\left(\int_Mu(t_\nu)^{2+\frac{2}{n}}|R_{\theta(t_\nu)}-r_\infty|^{\frac{2n+2}{n+2}}dV_{\theta_0}\right)^{\frac{n+2}{2n+2}(1+\gamma)}$$
for all integer $\nu$ in that subsequence. Note that $\gamma$ and $C$ may depend on the sequence
$\{t_\nu: \nu\in\mathbb{N}\}$.
\end{prop}

By assuming Proposition \ref{Prop3.3}, we are going to prove Theorem \ref{main}.
The following result is an immediate consequence of Proposition \ref{Prop3.3}.

\begin{prop}\label{Prop3.4}
There exists $0<\gamma<1$ and $t_0>0$ such that
$$r_{\theta(t)}-r_\infty\leq C\left(\int_Mu(t)^{2+\frac{2}{n}}|R_{\theta(t)}-r_\infty|^{\frac{2n+2}{n+2}}dV_{\theta_0}\right)^{\frac{n+2}{2n+2}(1+\gamma)}$$
for all $t\geq t_0$.
\end{prop}
\begin{proof}
Suppose this is not true. Then there exists a sequence of times $\{t_\nu: \nu\in\mathbb{N}\}$
and $\{C_\nu:\nu\in\mathbb{N}\}$
such that $t_\nu\geq\nu$, $C_\nu\to\infty$ as $\nu\to\infty$ and
$$r_{\theta(t_\nu)}-r_\infty\geq C_\nu \left(\int_Mu(t_\nu)^{2+\frac{2}{n}}|R_{\theta(t_\nu)}-r_\infty|^{\frac{2n+2}{n+2}}dV_{\theta_0}\right)^{\frac{n+2}{2n+2}(1+\frac{1}{\nu})}$$
for all $\nu\in\mathbb{N}$. We now apply Proposition \ref{Prop3.3} to this sequence $\{t_\nu: \nu\in\mathbb{N}\}$. Hence,
there exists an infinite subset $I\subset\mathbb{N}$, a real number $0<\gamma<1$ and a positive constant $C$ such that
$$r_{\theta(t_\nu)}-r_\infty\leq C\left(\int_Mu(t_\nu)^{2+\frac{2}{n}}|R_{\theta(t_\nu)}-r_\infty|^{\frac{2n+2}{n+2}}dV_{\theta_0}\right)^{\frac{n+2}{2n+2}(1+\gamma)}$$
for all  $\nu\in I$. Thus we conclude that
$$C_\nu\leq C\left(\int_Mu(t_\nu)^{2+\frac{2}{n}}|R_{\theta(t_\nu)}-r_\infty|^{\frac{2n+2}{n+2}}dV_{\theta_0}\right)^{\frac{n+2}{2n+2}(\gamma-\frac{1}{\nu})}
$$ for all $\nu\in I$,
which is a contradiction
in view of
(\ref{4}) and the fact that $C_\nu\to\infty$ as $\nu\to\infty$.
This proves the assertion.
\end{proof}

\begin{prop}\label{Prop3.5}
We have
$$\int_0^\infty\left(\int_Mu(t)^{2+\frac{2}{n}}(R_{\theta(t)}-r_{\theta(t)})^2dV_{\theta_0}\right)^{\frac{1}{2}}dt\leq C.$$
\end{prop}
\begin{proof}
It follows from (\ref{3}) and Proposition \ref{Prop3.4} that
\begin{equation*}
\begin{split}
r_{\theta(t)}-r_\infty&\leq C\left(\int_Mu(t)^{2+\frac{2}{n}}|R_{\theta(t)}-r_\infty|^{\frac{2n+2}{n+2}}dV_{\theta_0}\right)^{\frac{n+2}{2n+2}(1+\gamma)}\\
&\leq C\left(\int_Mu(t)^{2+\frac{2}{n}}|R_{\theta(t)}-r_{\theta(t)}|^{\frac{2n+2}{n+2}}dV_{\theta_0}\right)^{\frac{n+2}{2n+2}(1+\gamma)}
+C(r_{\theta(t)}-r_\infty)^{1+\gamma},
\end{split}
\end{equation*}
hence,
\begin{equation}\label{3.1}
r_{\theta(t)}-r_\infty\leq C\left(\int_Mu(t)^{2+\frac{2}{n}}|R_{\theta(t)}-r_{\theta(t)}|^{\frac{2n+2}{n+2}}dV_{\theta_0}\right)^{\frac{n+2}{2n+2}(1+\gamma)}
\end{equation}
if $t$ is sufficiently large in view of (\ref{5}). Therefore, we obtain
\begin{equation*}
\frac{d}{dt}(r_{\theta(t)}-r_\infty)
\leq -n\left(\int_Mu(t)^{2+\frac{2}{n}}|R_{\theta(t)}-r_{\theta(t)}|^{\frac{2n+2}{n+2}}dV_{\theta_0}\right)^{\frac{n+2}{n+1}}\leq -C(r_{\theta(t)}-r_\infty)^{\frac{2}{1+\gamma}}
\end{equation*}
where the first equality follows from (\ref{8}), and the first inequality follows from (\ref{3}) and H\"{o}lder's inequality,
and the last inequality follows from (\ref{3.1}). This implies that
$$\frac{d}{dt}(r_{\theta(t)}-r_\infty)^{-\frac{1-\gamma}{1+\gamma}}\geq C$$
where $C$ is a positive constant independent of $t$. From this, it follows that
\begin{equation}\label{3.2}
r_{\theta(t)}-r_\infty\leq Ct^{-\frac{1+\gamma}{1-\gamma}}
\end{equation}
if $t$ is sufficiently large. Therefore, we have
\begin{equation*}
\begin{split}
&\int_T^{2T}\left(\int_Mu(t)^{2+\frac{2}{n}}(R_{\theta(t)}-r_{\theta(t)})^2dV_{\theta_0}\right)^{\frac{1}{2}}dt\\
&\leq \left(T\int_T^{2T}\int_Mu(t)^{2+\frac{2}{n}}(R_{\theta(t)}-r_{\theta(t)})^2dV_{\theta_0}dt\right)^{\frac{1}{2}}\leq \left(\frac{T}{n}\big(r_{\theta(T)}-r_\infty\big)\right)^{\frac{1}{2}}\leq CT^{-\frac{\gamma}{1-\gamma}}
\end{split}
\end{equation*}
where the first inequality follows from H\"{o}lder's inequality, the last inequality follows from (\ref{8}) and (\ref{3.2}).
Since $0<\gamma<1$, we conclude that
\begin{equation*}
\begin{split}
&\int_0^{\infty}\left(\int_Mu(t)^{2+\frac{2}{n}}(R_{\theta(t)}-r_{\theta(t)})^2dV_{\theta_0}\right)^{\frac{1}{2}}dt\\
&=\int_0^{1}\left(\int_Mu(t)^{2+\frac{2}{n}}(R_{\theta(t)}-r_{\theta(t)})^2dV_{\theta_0}\right)^{\frac{1}{2}}dt\\
&\hspace{4mm}
+\sum_{k=0}^\infty\int_{2^k}^{2^{k+1}}\left(\int_Mu(t)^{2+\frac{2}{n}}(R_{\theta(t)}-r_{\theta(t)})^2dV_{\theta_0}\right)^{\frac{1}{2}}dt
\leq C\sum_{k=0}^\infty2^{-\frac{k\gamma}{1-\gamma}}\leq C.
\end{split}
\end{equation*}
This proves the assertion.
\end{proof}

\begin{prop}\label{Prop3.6}
Given any $\eta_0>0$, we can find a real number $r>0$ such that
$$\int_{B_{r}(x)}u(t)^{2+\frac{2}{n}}dV_{\theta_0}\leq\eta_0$$
for all $x\in M$ and $t\geq 0$.
Here
$B_r(x)=\{y\in M:d(x,y)<r\}$
and $d$ is the Carnot-Carath\'{e}odory distance on $M$
with respect to the contact form $\theta_0$.
\end{prop}
\begin{proof}
Given any $\eta_0>0$, it follows from Proposition \ref{Prop3.5} that there exists
a real number $T>0$ such that
$$\int_T^\infty\left(\int_Mu(t)^{2+\frac{2}{n}}(R_{\theta(t)}-r_{\theta(t)})^2dV_{\theta_0}\right)^{\frac{1}{2}}dt\leq \frac{\eta_0}{2(n+1)}.$$
On the other hand, by (\ref{3}), we can choose a real number $r>0$ such that
$$\int_{B_r(x)}u(t)^{2+\frac{2}{n}}dV_{\theta_0}\leq\frac{\eta_0}{2}$$
for all $x\in M$ and $0\leq t\leq T$. Combining these with
(\ref{0.5}), we have
\begin{equation*}
\begin{split}
\int_{B_r(x)}u(t)^{2+\frac{2}{n}}dV_{\theta_0}&\leq
\int_{B_r(x)}u(T)^{2+\frac{2}{n}}dV_{\theta_0}\\
&\hspace{4mm}+
(n+1)\int_T^\infty\left(\int_Mu(t)^{2+\frac{2}{n}}(R_{\theta(t)}-r_{\theta(t)})^2dV_{\theta_0}\right)^{\frac{1}{2}}dt\\
&\leq \frac{\eta_0}{2}+(n+1)\cdot\frac{\eta_0}{2(n+1)}=\eta_0
\end{split}
\end{equation*}
for all all $x\in M$ and $t\geq T$. This proves the assertion.
\end{proof}

\begin{prop}\label{Prop3.7}
The function $u(t)$ satisfies
\begin{equation}\label{3.5}
\sup_Mu(t)\leq C
\end{equation}
and
\begin{equation}\label{3.6}
\inf_Mu(t)\geq c
\end{equation}
for all $t\geq 0$. Here, $C$ and $c$ are positive constants independent of $t$.
\end{prop}
\begin{proof}
We follow the proof of Proposition 5.5 in \cite{Ho}.
It follows from (\ref{3}) and (\ref{4}) that for $n+1<p<n+2$
\begin{equation}\label{3.3}
\int_{M}|R_{\theta(t)}|^{p}dV_{\theta(t)}\leq C
\end{equation}
for some constant $C$ independent of $t$. Using (\ref{3.3}), Proposition \ref{Prop3.6}, and H\"{o}lder's inequality, we obtain that
for $p>q>n+1$
\begin{equation}\label{3.4}
\int_{B_r(x)}|R_{\theta(t)}|^{q}dV_{\theta(t)}\leq \left(\int_{B_r(x)}dV_{\theta(t)}\right)^{\frac{p-q}{p}}\left(\int_{M}|R_{\theta(t)}|^{p}dV_{\theta(t)}\right)^{\frac{q}{p}}
\leq C\eta_0.
\end{equation}
Since $u(t)$ is smooth by the long time existence of the CR Yamabe flow,
we can choose $\eta_0$ sufficiently small in (\ref{3.4}) so that we can apply Proposition A.2 in \cite{Ho} to conclude that
$u(t)$ is uniformly bounded from above. This proves (\ref{3.5}).

Now, if we define
$$P=R_{\theta_0}+\sigma\left(\sup_{t\geq0}\sup_M u(t)\right)^{\frac{2}{n}}$$
where $\sigma$ is given by
$$\sigma=\max\Big\{\sup(1-R_{\theta_0}),1\Big\}.$$
It was proved in \cite{Ho} that (see Proposition 3.4 in \cite{Ho})
$$R_{\theta(t)}+\sigma\geq 1,$$
which implies that
\begin{equation*}
\begin{split}
-(2+\frac{2}{n})\Delta_{\theta_0}u(t)+Pu(t)
&\geq -(2+\frac{2}{n})\Delta_{\theta_0}u(t)+R_{\theta_0}u(t)+\sigma u(t)^{1+\frac{2}{n}}\\
&=(R_{\theta(t)}+\sigma) u(t)^{1+\frac{2}{n}}\geq 0.
\end{split}
\end{equation*}
Hence, we can apply Proposition A.1 in \cite{Ho} to conclude that
$$C\Big(\inf_Mu(t) \Big)\Big(\sup_M u(t)\Big)^{1+\frac{2}{n}}\geq \int_M u(t)^{2+\frac{2}{n}}dV_{\theta_0}$$
for some positive constant $C$. Hence, (\ref{3.6}) follows from (\ref{3}) and  (\ref{3.5}). This
proves the assertion.
\end{proof}

\begin{prop}\label{Prop3.8}
Let $0<\alpha<\frac{2}{n+2}$. There exists a constant $C$ such that
$$|u(x_2,t_2)-u(x_1,t_1)|\leq C\big((t_1-t_2)^\frac{\alpha}{2}+d(x_1,x_2)^\alpha\big)$$
for all  $x_1, x_2\in M$ and all $t_1, t_2\geq 0$ satisfying
$0<t_1-t_2<1$. Here $d$ is the Carnot-Carath\'{e}odory distance on $M$
with respect to the contact form $\theta_0$.
\end{prop}
\begin{proof}
Choose $\alpha=2-\displaystyle\frac{2n+2}{p}$ with $n+1<p<n+2$.
Then we have
\begin{equation*}
\begin{split}
&\int_{M}\Big|-(2+\frac{2}{n})\Delta_{\theta_0}u(t)+R_{\theta_0}
u(t)\Big|^pdV_{\theta_0}\\
&\leq C\int_{M}|R_{\theta(t)}|^pdV_{\theta(t)}\leq C\Big(\int_{M}|R_{\theta(t)}-r_{\infty}|^pdV_{\theta(t)}+\int_{M}r_{\infty}^pdV_{\theta(t)}\Big)\leq C
\end{split}
\end{equation*}
where the first inequality follows from Proposition \ref{Prop3.7}, and the last inequality follows from
(\ref{3}) and (\ref{4}). This implies that
$$|u(x_2,t)-u(x_1,t)|\leq Cd(x_1,x_2)^\alpha$$
for all $x_1, x_2\in M$ and all $t\geq 0$. On the other hand,
\begin{equation}\label{3.7}
\begin{split}
\int_{M}\Big|\frac{\partial}{\partial t}u(t)\Big|^pdV_{\theta_0}
&=\big(\frac{n}{2}\big)^p\int_{M}|(R_{\theta(t)}-r_{\theta(t)})u(t)|^pdV_{\theta_0}\\
&\leq C\int_{M}|R_{\theta(t)}-r_{\theta(t)}|^pdV_{\theta(t)}\leq C,
\end{split}
\end{equation}
where the first equality follows from (\ref{0.5}), and the second inequality follows from Proposition \ref{Prop3.7},
and the last inequality follows from (\ref{4A}). Using (\ref{3.7}), we get
\begin{equation*}
\begin{split}
&|u(x,t_1)-u(x,t_2)|\leq
C(t_1-t_2)^{-(n+1)}\int_{B_{\sqrt{t_1-t_2}}(x)}|u(x,t_1)-u(x,t_2)|dV_{\theta_0}\\
&\leq
C(t_1-t_2)^{-(n+1)}\int_{B_{\sqrt{t_1-t_2}}(x)}|u(t_1)-u(t_2)|dV_{\theta_0}+C(t_1-t_2)^{\frac{\alpha}{2}}\\
&\leq C(t_1-t_2)^{-n}\sup_{t_2\leq t\leq
t_1}\int_{B_{\sqrt{t_1-t_2}}(x)}\left|\frac{\partial}{\partial
t}u(t)\right|dV_{\theta_0}+C(t_1-t_2)^{\frac{\alpha}{2}}\\
&\leq C(t_1-t_2)^{\frac{\alpha}{2}}\sup_{t_2\leq t\leq
t_1}\left(\int_{M}\left|\frac{\partial}{\partial
t}u(t)\right|^pdV_{\theta_0}\right)^{\frac{1}{p}}+C(t_1-t_2)^{\frac{\alpha}{2}}\\
&\leq C(t_1-t_2)^{\frac{\alpha}{2}}
\end{split}
\end{equation*}
for all $x\in M$ and all $t_1,t_2\geq 0$ satisfying $0<t_1-t_2<1$.
This proves the assertion.
\end{proof}

In view of Proposition \ref{Prop3.8}, it is easy to see that all
derivatives of $u(x, t)$ are uniformly bounded on
$[0,\infty)$. Indeed, we can apply Theorem 1.1 in \cite{Bramanti}, which
says: let $X_1,X_2,\dots,X_q$ be a system of real smooth vector
fields satisfying H\"{o}rmander's condition in a bounded domain
$\Omega$ of $\mathbb{R}^n$. Let $A=\{a_{ij}(x,t)\}^q_{i,j=1}$ be a
symmetric, uniformly positive-definite matrix of real functions
defined in a domain $U\subset\Omega\times\mathbb{R}$. For operator of
the form
$$H=\partial_t-\sum^q_{i,j=1}a_{ij}(x,t)X_iX_j-\sum^q_{i=1}b_i(x,t)X_i-c(x,t)$$
we have a priori estimate of Schauder type in parabolic
H\"{o}rmander H\"{o}lder spaces $C^{k,\beta}_P(U)$. Namely, for
$a_{ij},b_i,c\in C^{k,\beta}_P(U)$ and $U'\Subset U$, we have
\begin{eqnarray}\label{3.8}
\| u\|_{C^{k+2,\beta}_P(U')}\le C\{\| Hu\|_{C^{k,\beta}_P(U)}+\|
u\|_{L^\infty(U)}\}.
\end{eqnarray}
Here, (see P.193-194 in \cite{Bramanti})
\begin{equation*}
\begin{split}
C^{k,\beta}_P(U)&=\{u:U\rightarrow\mathbb{R}:
\|u\|_{C^{k,\beta}_P(U)}<\infty\},\\
\|u\|_{C^{k,\beta}_P(U)}&=\sum_{|I|+2h\leq
k}\left\|\partial_t^hX^Iu\right\|_{C^\beta_P(U)},\\
\|u\|_{C^{\beta}_P(U)}&=|u|_{C^\beta_P(U)}+\|u\|_{L^\infty(U)},\\
|u|_{C^{\beta}_P(U)}&=\sup\left\{\frac{|u(t,x)-u(s,y)|}{d_P((x,t),(y,s))^{\beta}}:(x,t),(y,s)\in
U, (t,x)\neq(s,y)\right\},
\end{split}
\end{equation*}
where $d_P$ is the parabolic Carnot-Carath\'{e}odory distance (see
P. 189 in \cite{Bramanti}) which is given by
$$d_P((x_1,t_1),(t_2,x_2))=\sqrt{d(x_1,x_2)^2+|t_1-t_2|}.$$
Here $d$ is the Carnot-Carath\'{e}odory distance in $\Omega$.
Moreover, for any multiindex $I=(i_1,i_2,...,i_s)$, with $1\leq
i_j\leq q$, $X^Iu=X_{i_1}X_{i_2}\cdots X_{i_s}u$.

It follows from Proposition \ref{Prop3.8} that $u(x,t)\in
C_P^{0,\lambda}([0,\infty)\times M)$. Therefore, with the
estimate (\ref{3.8}), Proposition \ref{Prop3.7} and Proposition \ref{Prop3.8}, we can
now apply the standard regularity theory for the weakly parabolic
equation in (\ref{10}) to show that all higher order derivatives of $u(t)$ are
uniformly bounded on $[0,\infty)$. Therefore, $u(t)$ converges to a smooth function $u_\infty$,
which is positive in view of Proposition \ref{Prop3.7}, or equivalently,
$\theta(t)$ converges to a contact form $\theta_\infty=u_\infty^{\frac{2}{n}}\theta_0$
as $t\to\infty$. On the other hand, the contact form $\theta_\infty$ has constant
Webster scalar curvature thanks to (\ref{4}). This proves Theorem \ref{main}.

%%%%%%%%%%%%%%%%%%%%%%%%%%%%%%%%%%%%%

\section{Blow-up analysis}\label{section5A}

The remaining part of this paper will be concerned with the proof of Proposition \ref{Prop3.3}.
Let $\{t_\nu: \nu\in\mathbb{N}\}$ be a sequence of times such that $t_\nu\to\infty$ as $\nu\to\infty$.
For abbreviation, we write $u_\nu=u(t_\nu)$ and $\theta_\nu=\theta(t_\nu)=u(t_\nu)^{\frac{2}{n}}\theta_0=u_\nu^{\frac{2}{n}}\theta_0$,
it follows from
(\ref{3}) that
\begin{equation}\label{4.1}
\int_Mu_\nu^{2+\frac{2}{n}}dV_{\theta_0}=1\hspace{2mm}\mbox{ for all }\nu\in\mathbb{N}.
\end{equation}
On the other hand, it follows from (\ref{9}) and (\ref{4}) with $p=\frac{2n+2}{n+2}$ that
\begin{equation}\label{4.2}
\int_M\Big|-(2+\frac{2}{n})\Delta_{\theta_0}u_\nu+R_{\theta_0}u_\nu-r_\infty u_\nu^{1+\frac{2}{n}}\Big|^{\frac{2n+2}{n+2}}dV_{\theta_0}
\to 0\hspace{2mm}\mbox{ as }\nu\to \infty.
\end{equation}

At this point, we may apply the following concentration-compactness result,
which is the CR version of the results of Struwe \cite{Struwe} and Bahri and Coron \cite{Bahri&Coron}.
See also \cite{Citti} and \cite{Citti&Uguzzoni}.

\begin{theorem}\label{Theorem4.1}
Under the assumptions of Theorem \ref {main}, suppose that $\{u_\nu\}$ be a sequence of positive functions satisfying $(\ref{4.1})$ and $(\ref{4.2})$. After passing to a subsequence if necessary, we can find an integer $m$,
a smooth nonnegative function $u_\infty$ and a sequence of $m$-tuples $(x^*_{k,\nu}, \varepsilon^*_{k,\nu})_{1\leq k\leq m}$
with the following properties:\\
(i) The function $u_\infty$ satisfies
\begin{equation}\label{4.3}
(2+\frac{2}{n})\Delta_{\theta_0}u_\infty-R_{\theta_0}u_\infty+r_\infty u_\infty^{1+\frac{2}{n}}=0.
\end{equation}
(ii) For all $i\neq j$, we have
\begin{equation}\label{4.4}
\frac{\varepsilon^*_{i,\nu}}{\varepsilon^*_{j,\nu}}+\frac{\varepsilon^*_{j,\nu}}{\varepsilon^*_{i,\nu}}
+\frac{d(x^*_{i,\nu},x^*_{j,\nu})^2}{\varepsilon^*_{i,\nu}\varepsilon^*_{j,\nu}}\rightarrow\infty\hspace{2
mm}\mbox{ as }\hspace{2 mm}\nu\rightarrow\infty.
\end{equation}
(iii) We have
\begin{equation}\label{4.5}
\Big\|u_\nu-u_\infty-\sum_{k=1}^m\overline{u}_{(x^*_{k,\nu},\varepsilon^*_{k,\nu})}\Big\|_{S^2_1(M)}\rightarrow
0\hspace{2 mm}\mbox{ as }\hspace{2 mm}\nu\rightarrow\infty.
\end{equation}
Here $\overline{u}_{(x^*_{k,\nu},\varepsilon^*_{k,\nu})}$ are the standard test functions constructed
in $(\ref{testfcn})$ in Appendix \ref{Appendix}, and $d$ is Carnot-Carath\'{e}odory distance on $M$
with respect to the contact form $\theta_0$.
\end{theorem}

The proof of
Theorem \ref{Theorem4.1} will be included in
Appendix \ref{AppendixB}.

\begin{prop}\label{Prop4.2}
If $u_\infty$ vanishes at one point in $M$, then $u_\infty$ vanishes everywhere.
\end{prop}
\begin{proof}
It follows from (\ref{4.3}) that
$$-(2+\frac{2}{n})\Delta_{\theta_0}u_\infty+R_{\theta_0}u_\infty=r_\infty u_\infty^{1+\frac{2}{n}}\geq 0.$$
Since $R_{\theta_0}\geq 0$ by (\ref{11}), we can apply Proposition A.1 in \cite{Ho} to conclude that
\begin{equation}\label{4.9}
C\Big(\inf_Mu_\infty \Big)\Big(\sup_M u_\infty\Big)^{1+\frac{2}{n}}\geq \int_M u_\infty^{2+\frac{2}{n}}dV_{\theta_0}
\end{equation}
for some positive constant $C$. Proposition \ref{Prop4.2} follows from (\ref{4.9}).
\end{proof}

The cases $u_\infty\equiv 0$ and $u_\infty>0$ will be discussed separately.
The case $u_\infty\equiv 0$ will be studied in section \ref{section5}.
The case $u_\infty>0$ will be studied in section \ref{section6}.

We define two functionals $E(u)$ and $F(u)$ by
\begin{equation}\label{4.6}
E(u)=\frac{\int_M\big((2+\frac{2}{n})|\nabla_{\theta_0}u|^2_{\theta_0}+R_{\theta_0}u^2\big)dV_{\theta_0}}
{\big(\int_Mu^{2+\frac{2}{n}}dV_{\theta_0}\big)^{\frac{n}{n+1}}}
\end{equation}
and
\begin{equation}\label{4.7}
F(u)=\frac{\int_M\big((2+\frac{2}{n})|\nabla_{\theta_0}u|^2_{\theta_0}+R_{\theta_0}u^2\big)dV_{\theta_0}}{\int_Mu^{2+\frac{2}{n}}dV_{\theta_0}}.
\end{equation}
Then we have
\begin{equation}\label{4.8}
\begin{split}
1&=\lim_{\nu\to\infty}\int_Mu_\nu^{2+\frac{2}{n}}dV_{\theta_0}\\
&=\lim_{\nu\to\infty}\left(\int_Mu_\infty^{2+\frac{2}{n}}dV_{\theta_0}
+\sum_{k=1}^m\int_M\overline{u}_{(x^*_{k,\nu},\varepsilon^*_{k,\nu})}^{2+\frac{2}{n}}dV_{\theta_0}\right)\\
&=\left(\frac{E(u_\infty)}{r_\infty}\right)^{n+1}+m\left(\frac{Y(S^{2n+1})}{r_\infty}\right)^{n+1}
\end{split}
\end{equation}
where we have used (\ref{4.1}), (\ref{4.3}), (\ref{4.5}), and (\ref{B.14}).

\subsection{The case $u_\infty\equiv 0$}\label{section5}

Throughout this subsection, we assume that $u_\infty\equiv 0$. For every $\nu\in\mathbb{N}$, we denote by $\mathcal{A}_\nu$ the set of
all $m$-tuples $(x_k,\varepsilon_k,\alpha_k)_{1\leq k\leq m}\in (M\times\mathbb{R}_+\times\mathbb{R}_+)^m$
such that
\begin{equation}\label{5.1}
d(x_k,x_{k,\nu}^*)\leq\varepsilon_{k,\nu}^*,\hspace{4mm}\frac{1}{2}\leq\frac{\varepsilon_{k}}{\varepsilon_{k,\nu}^*}\leq 2,
\hspace{4mm}  \frac{1}{2}\leq\alpha_k\leq 2
\end{equation}
for all $1\leq k\leq m$. Moreover, we can find
$(x_{k,\nu},\varepsilon_{k,\nu},\alpha_{k,\nu})_{1\leq k\leq m}\in \mathcal{A}_\nu$
such that
\begin{equation}\label{5.2}
\begin{split}
&\int_M\left((2+\frac{2}{n})\Big|\nabla_{\theta_0}\Big(u_\nu-\sum_{k=1}^m\alpha_{k,\nu}\overline{u}_{(x_{k,\nu},\varepsilon_{k,\nu})}\Big)\Big|_{\theta_0}^2
+R_{\theta_0}\Big(u_\nu-\sum_{k=1}^m\alpha_{k,\nu}\overline{u}_{(x_{k,\nu},\varepsilon_{k,\nu})}\Big)^2\right)dV_{\theta_0}\\
&\leq
\int_M\left((2+\frac{2}{n})\Big|\nabla_{\theta_0}\Big(u_\nu-\sum_{k=1}^m\alpha_{k}\overline{u}_{(x_{k},\varepsilon_{k})}\Big)\Big|_{\theta_0}^2
+R_{\theta_0}\Big(u_\nu-\sum_{k=1}^m\alpha_{k}\overline{u}_{(x_{k},\varepsilon_{k})}\Big)^2\right)dV_{\theta_0}
\end{split}
\end{equation}
for all $(x_k,\varepsilon_k,\alpha_k)_{1\leq k\leq m}\in \mathcal{A}_\nu$.

\begin{prop}\label{Prop5.1}
(i) For all $i\neq j$, we have
$$\frac{\varepsilon_{i,\nu}^2}{\varepsilon_{j,\nu}^2}+\frac{\varepsilon_{j,\nu}^2}{\varepsilon_{i,\nu}^2}
+\frac{d(x_{i,\nu},x_{j,\nu})^4}{\varepsilon_{i,\nu}^2\varepsilon_{j,\nu}^2}\rightarrow\infty\hspace{2
mm}\mbox{ as }\hspace{2 mm}\nu\rightarrow\infty,$$
(ii) We have
$$\Big\|u_\nu-\sum_{k=1}^m\alpha_{k,\nu}\overline{u}_{(x_{k,\nu},\varepsilon_{k,\nu})}\Big\|_{S^2_1(M)}\rightarrow
0\hspace{2 mm}\mbox{ as }\hspace{2 mm}\nu\rightarrow\infty.$$
\end{prop}
\begin{proof}
(i) In view of (\ref{5.1}), we have
\begin{equation*}
\begin{split}
&32\frac{\varepsilon_{i,\nu}}{\varepsilon_{j,\nu}}+32\frac{\varepsilon_{j,\nu}}{\varepsilon_{i,\nu}}
+8\frac{d(x_{i,\nu},x_{j,\nu})^2}{\varepsilon_{i,\nu}\varepsilon_{j,\nu}}\\
&\geq
8\frac{\varepsilon_{i,\nu}^*}{\varepsilon_{j,\nu}^*}+8\frac{\varepsilon_{j,\nu}^*}{\varepsilon_{i,\nu}^*}
+2\frac{d(x_{i,\nu},x_{j,\nu})^2}{\varepsilon_{i,\nu}^*\varepsilon_{j,\nu}^*}\\
&\geq
4\frac{\varepsilon_{i,\nu}^*}{\varepsilon_{j,\nu}^*}+4\frac{\varepsilon_{j,\nu}^*}{\varepsilon_{i,\nu}^*}
+\frac{(d(x_{i,\nu},x_{j,\nu})+\varepsilon_{i,\nu}^*+\varepsilon_{j,\nu}^*)^2}{\varepsilon_{i,\nu}^*\varepsilon_{j,\nu}^*}\\
&\geq
4\frac{\varepsilon_{i,\nu}^*}{\varepsilon_{j,\nu}^*}+4\frac{\varepsilon_{j,\nu}^*}{\varepsilon_{i,\nu}^*}
+\frac{d(x_{i,\nu}^*,x_{j,\nu}^*)^2}{\varepsilon_{i,\nu}^*\varepsilon_{j,\nu}^*},
\end{split}
\end{equation*}
and the last expression tends to infinity as $\nu\rightarrow\infty$
by Theorem \ref{Theorem4.1}. Thus we have
$$\frac{\varepsilon_{i,\nu}}{\varepsilon_{j,\nu}}+\frac{\varepsilon_{j,\nu}}{\varepsilon_{i,\nu}}
+\frac{d(x_{i,\nu},x_{j,\nu})^2}{\varepsilon_{i,\nu}\varepsilon_{j,\nu}}\rightarrow\infty\hspace{2
mm}\mbox{ as }\hspace{2 mm}\nu\rightarrow\infty.$$
Now Proposition \ref{Prop5.1}(i) follows from this and Cauchy-Schwarz inequality.\\
(ii) By definition of
$(x_{k,\nu},\varepsilon_{k,\nu},\alpha_{k,\nu})_{1\leq k\leq m}$
in (\ref{5.2}), we have
\begin{equation*}
\begin{split}
&\int_M\left((2+\frac{2}{n})\Big|\nabla_{\theta_0}\Big(u_\nu-\sum_{k=1}^m\alpha_{k,\nu}\overline{u}_{(x_{k,\nu},\varepsilon_{k,\nu})}\Big)\Big|_{\theta_0}^2
+R_{\theta_0}\Big(u_\nu-\sum_{k=1}^m\alpha_{k,\nu}\overline{u}_{x_{(k,\nu},\varepsilon_{k,\nu})}\Big)^2\right)dV_{\theta_0}\\
&\leq
\int_M\left((2+\frac{2}{n})\Big|\nabla_{\theta_0}\Big(u_\nu-\sum_{k=1}^m\overline{u}_{(x_{k}^*,\varepsilon_{k}^*)}\Big)\Big|_{\theta_0}^2
+R_{\theta_0}\Big(u_\nu-\sum_{k=1}^m\overline{u}_{(x_{k}^*,\varepsilon_{k}^*)}\Big)^2\right)dV_{\theta_0}.
\end{split}
\end{equation*}
By Theorem \ref{Theorem4.1}, the expression on the right-hand side
tends to 0 as $\nu\rightarrow\infty$. This proves the assertion.
\end{proof}

\begin{prop}\label{Prop5.2}
We have
\begin{equation*}
d(x_{k,\nu},x_{k,\nu}^*)\leq o(1)\varepsilon_{k,\nu}^*,\hspace{4mm}\frac{\varepsilon_{k,\nu}}{\varepsilon_{k,\nu}^*}=1+o(1),
\hspace{4mm}  \alpha_{k,\nu}=1+o(1)
\end{equation*}
for all $1\leq k\leq m$. In particular, $(x_{k,\nu},\varepsilon_{k,\nu},\alpha_{k,\nu})_{1\leq k\leq m}$
is an interior point of $\mathcal{A}_\nu$ if $\nu$ is sufficiently large.
\end{prop}
\begin{proof}
Observe that
\begin{equation*}
\begin{split}
&\Big\|\sum_{k=1}^m\alpha_{k,\nu}\overline{u}_{(x_{k,\nu},\varepsilon_{k,\nu})}
-\sum_{k=1}^m\overline{u}_{(x_{k}^*,\varepsilon_{k}^*)}\Big\|_{S_1^2(M)}\\
&\leq
\Big\|u_\nu-\sum_{k=1}^m\overline{u}_{(x_{k}^*,\varepsilon_{k}^*)}\Big\|_{S_1^2(M)}
+\Big\|u_\nu-\sum_{k=1}^m\alpha_{k,\nu}\overline{u}_{(x_{k,\nu},\varepsilon_{k,\nu})}\Big\|_{S_1^2(M)}=o(1)
\end{split}
\end{equation*}
by Theorem \ref{Theorem4.1} and Proposition \ref{5.1}. From this, the assertion follows.
\end{proof}

Now we decompose the function $u_\nu$ as
$$u_\nu=v_\nu+w_\nu$$
where
\begin{equation}\label{5.3}
v_\nu=\sum_{k=1}^m\alpha_{k,\nu}\overline{u}_{(x_{k,\nu},\varepsilon_{k,\nu})}
\end{equation}
and
\begin{equation}\label{5.4}
w_\nu=u_\nu-\sum_{k=1}^m\alpha_{k,\nu}\overline{u}_{(x_{k,\nu},\varepsilon_{k,\nu})}.
\end{equation}
By Proposition \ref{Prop5.1}, the function $w_\nu$ satisfies
\begin{equation}\label{5.5}
\int_M\left((2+\frac{2}{n})|\nabla_{\theta_0}w_\nu|^2_{\theta_0}+R_{\theta_0}w_\nu^2\right)dV_{\theta_0}=o(1).
\end{equation}

\begin{prop}\label{Prop5.4}
If $\nu$ is sufficiently large, then
$$\frac{n+2}{n}r_\infty\int_M\sum_{k=1}^m\overline{u}_{(x_{k,\nu},\varepsilon_{k,\nu})}^{\frac{2}{n}}w_\nu^2\,dV_{\theta_0}\leq (1-c)
\int_M\left((2+\frac{2}{n})|\nabla_{\theta_0}w_\nu|^2_{\theta_0}+R_{\theta_0}w_\nu^2\right)dV_{\theta_0}$$
for some positive constant $c$ which is independent of $\nu$.
\end{prop}
\begin{proof}
Suppose this is not true. Upon rescaling, we obtain a sequence
of functions $\{\widetilde{w}_\nu:\nu\in\mathbb{N}\}$ such that
\begin{equation}\label{5.10}
\int_M\left((2+\frac{2}{n})|\nabla_{\theta_0}\widetilde{w}_\nu|^2_{\theta_0}+R_{\theta_0}\widetilde{w}_\nu^2\right)dV_{\theta_0}=1
\end{equation}
and
\begin{equation}\label{5.11}
\lim_{\nu\to\infty}\frac{n+2}{n}r_\infty\int_M\sum_{k=1}^m\overline{u}_{(x_{k,\nu},\varepsilon_{k,\nu})}^{\frac{2}{n}}\widetilde{w}_\nu^2\,dV_{\theta_0}\geq 1.
\end{equation}
Note that
\begin{equation}\label{5.12}
\int_M|\widetilde{w}_\nu|^{2+\frac{2}{n}}dV_{\theta_0}\leq Y(M,\theta_0)^{-\frac{n+1}{n}}
\end{equation}
by (\ref{5.10}). In view of Proposition \ref{Prop5.1}, we can find a sequence $\{N_\nu:\nu\in\mathbb{N}\}$ such that
$N_\nu\to\infty$, $N_\nu\varepsilon_{j,\nu}\to 0$ for all $1\leq j\leq m$, and
\begin{equation}\label{5.13}
\frac{1}{N_\nu}\frac{\varepsilon_{j,\nu}+d(x_{i,\nu},x_{j,\nu})}{\varepsilon_{i,\nu}}\to\infty
\end{equation}
for all $i<j$. Let
\begin{equation}\label{5.14}
\Omega_{j,\nu}=B_{N_\nu\varepsilon_{j,\nu}}\setminus\bigcup_{i=1}^{j-1}B_{N_\nu\varepsilon_{i,\nu}}(x_{i,\nu})
\end{equation}
for every $1\leq j\leq m$. In view of (\ref{5.10}) and (\ref{5.11}), we can find an integer $1\leq j\leq m$ such that
\begin{equation}\label{5.15}
\lim_{\nu\to\infty}\int_M\overline{u}_{(x_{j,\nu},\varepsilon_{j,\nu})}^{\frac{2}{n}}\widetilde{w}_\nu^2\,dV_{\theta_0}>0
\end{equation}
and
\begin{equation}\label{5.16}
\lim_{\nu\to\infty}\int_{\Omega_{j,\nu}}\left((2+\frac{2}{n})|\nabla_{\theta_0}\widetilde{w}_\nu|^2_{\theta_0}+R_{\theta_0}\widetilde{w}_\nu^2\right)dV_{\theta_0}
\leq\lim_{\nu\to\infty}\frac{n+2}{n}r_\infty\int_M\overline{u}_{(x_{j,\nu},\varepsilon_{j,\nu})}^{\frac{2}{n}}\widetilde{w}_\nu^2\,dV_{\theta_0}.
\end{equation}
We now define a sequence of functions $\widehat{w}_\nu:T_{x_{j,\nu}}M\to\mathbb{R}$ by
$$\widehat{w}_\nu(z,t)=\varepsilon_{j,\nu}^n\widetilde{w}_\nu(\exp_{x_{j,\nu}}(\varepsilon_{j,\nu}z,\varepsilon_{j,\nu}^2t))$$
for $(z,t)\in T_{x_{j,\nu}}M$. The sequence $\{\widehat{w}_\nu:\nu\in\mathbb{N}\}$ satisfies
\begin{equation*}
\lim_{\nu\to\infty}\int_{\{(z,t)\in\mathbb{H}^n: |(z,t)|\leq N_\nu\}}(2+\frac{2}{n})|\nabla_{\theta_{\mathbb{H}^n}}\widehat{w}_\nu(z,t)|^2_{\theta_{\mathbb{H}^n}}dV_{\theta_{\mathbb{H}^n}}\leq 1
\end{equation*}
and
\begin{equation*}
\lim_{\nu\to\infty}\int_{\{(z,t)\in\mathbb{H}^n: |(z,t)|\leq N_\nu\}}|\widehat{w}_\nu(z,t)|^{2+\frac{2}{n}}dV_{\theta_{\mathbb{H}^n}}\leq Y(M,\theta_0)^{-\frac{n+1}{n}}
\end{equation*}
in view of (\ref{5.10}) and (\ref{5.12}). Hence, we can take the weak limit to obtain a function
$\widehat{w}:\mathbb{H}^n\to\mathbb{R}$ such that
\begin{equation}\label{5.17}
\int_{\mathbb{H}^n}\frac{1}{t^2+(1+|z|^2)^2}\widehat{w}(z,t)^2dV_{\theta_{\mathbb{H}^n}}>0
\end{equation}
and
\begin{equation}\label{5.18}
\int_{\mathbb{H}^n}|\nabla_{\theta_{\mathbb{H}^n}}\widehat{w}(z,t)|^2_{\theta_{\mathbb{H}^n}}dV_{\theta_{\mathbb{H}^n}}
\leq(n+2)n\int_{\mathbb{H}^n}\frac{1}{t^2+(1+|z|^2)^2}\widehat{w}(z,t)^2dV_{\theta_{\mathbb{H}^n}}
\end{equation}
by (\ref{5.15}), (\ref{5.16}), and the definition of $\overline{u}_{(x_{j,\nu},\varepsilon_{j,\nu})}$.
By definition of $(x_{k,\nu},\varepsilon_{k,\nu},\alpha_{k,\nu})_{1\leq k\leq m}$ in (\ref{5.2}), we have
\begin{equation*}
\begin{split}
&\frac{d}{d\alpha_k}
\int_M\left((2+\frac{2}{n})\Big|\nabla_{\theta_0}\Big(u_\nu-\sum_{k=1}^m\alpha_{k}\overline{u}_{(x_{k,\nu},\varepsilon_{k,\nu})}\Big)\Big|_{\theta_0}^2
\right.
\\
&\left.\left.\hspace{2cm}+R_{\theta_0}\Big(u_\nu-\sum_{k=1}^m\alpha_{k}\overline{u}_{(x_{k,\nu},\varepsilon_{k,\nu})}\Big)^2\right)dV_{\theta_0}
\right|_{\alpha_k=\alpha_{k,\nu}}=0,
\end{split}
\end{equation*}
which implies
\begin{equation*}
\begin{split}
0&=\int_M\left((2+\frac{2}{n})\Delta_{\theta_0}\overline{u}_{(x_{k,\nu},\varepsilon_{k,\nu})}
-R_{\theta_0}\overline{u}_{(x_{k,\nu},\varepsilon_{k,\nu})}\right)w_\nu dV_{\theta_0}.
\end{split}
\end{equation*}
Using the estimate
$$\Big\|(2+\frac{2}{n})\Delta_{\theta_0}\overline{u}_{(x_{k,\nu},\varepsilon_{k,\nu})}
-R_{\theta_0}\overline{u}_{(x_{k,\nu},\varepsilon_{k,\nu})}
+r_\infty\overline{u}_{(x_{k,\nu},\varepsilon_{k,\nu})}^{1+\frac{2}{n}}\Big\|_{L^{\frac{2n+2}{n+2}}(M)}=o(1)$$
and H\"{o}lder's inequality,
we conclude that
\begin{equation*}
\begin{split}
&r_\infty\left|\int_M\overline{u}_{(x_{k,\nu},\varepsilon_{k,\nu})}^{1+\frac{2}{n}}w_\nu dV_{\theta_0}\right|\\
&\leq \Big\|(2+\frac{2}{n})\Delta_{\theta_0}\overline{u}_{(x_{k,\nu},\varepsilon_{k,\nu})}
-R_{\theta_0}\overline{u}_{(x_{k,\nu},\varepsilon_{k,\nu})}
+r_\infty\overline{u}_{(x_{k,\nu},\varepsilon_{k,\nu})}^{1+\frac{2}{n}}\Big\|_{L^{\frac{2n+2}{n+2}}(M)}\\
&\hspace{4mm}
\cdot
\left(\int_M|w_\nu|^{2+\frac{2}{n}} dV_{\theta_0}\right)^{\frac{n}{2n+2}}\\
&= o(1)
\left(\int_M|w_\nu|^{2+\frac{2}{n}} dV_{\theta_0}\right)^{\frac{n}{2n+2}}
\end{split}
\end{equation*}
for all $1\leq k\leq m$. Since $r_\infty>0$,  we have
\begin{equation*}
\left|\int_M\overline{u}_{(x_{k,\nu},\varepsilon_{k,\nu})}^{1+\frac{2}{n}}w_\nu dV_{\theta_0}\right|\leq o(1)
\left(\int_M|w_\nu|^{2+\frac{2}{n}} dV_{\theta_0}\right)^{\frac{n}{2n+2}}
\end{equation*}
for all $1\leq k\leq m$. Taking the weak limit yields
\begin{equation}\label{5.19}
\int_{\mathbb{H}^n}\frac{1}{(t^2+(1+|z|^2)^2)^{\frac{n+2}{2}}}\widehat{w}(z,t)dV_{\theta_{\mathbb{H}^n}}=0
\end{equation}
by the definition of $\overline{u}_{(x_{j,\nu},\varepsilon_{j,\nu})}$, $\widehat{w}_\nu$,  and $\widehat{w}$.
By definition of $(x_{k,\nu},\varepsilon_{k,\nu},\alpha_{k,\nu})_{1\leq k\leq m}$ in (\ref{5.2}), we also have
\begin{equation*}
\begin{split}
&\frac{d}{d\varepsilon_{k}}
\int_M\left((2+\frac{2}{n})\Big|\nabla_{\theta_0}\Big(u_\nu-\sum_{k=1}^m\alpha_{k}\overline{u}_{(x_{k,\nu},\varepsilon_{k})}\Big)\Big|_{\theta_0}^2
\right.
\\
&\left.\left.\hspace{2cm}+R_{\theta_0}\Big(u_\nu-\sum_{k=1}^m\alpha_{k}\overline{u}_{(x_{k,\nu},\varepsilon_{k})}\Big)^2\right)dV_{\theta_0}
\right|_{\varepsilon_k=\varepsilon_{k,\nu}}=0,
\end{split}
\end{equation*}
which implies that
\begin{equation*}
\begin{split}
0&=\int_M\left((2+\frac{2}{n})\Delta_{\theta_0}\Big(\frac{\overline{u}_{(x_{k,\nu},\varepsilon_{k,\nu})}}{\partial\varepsilon_{k,\nu}}\Big)
-R_{\theta_0}\Big(\frac{\overline{u}_{(x_{k,\nu},\varepsilon_{k,\nu})}}{\partial\varepsilon_{k,\nu}}\Big)\right)w_\nu dV_{\theta_0}.
\end{split}
\end{equation*}
Using the estimate
$$\Big\|(2+\frac{2}{n})\Delta_{\theta_0}\Big(\frac{\overline{u}_{(x_{k,\nu},\varepsilon_{k,\nu})}}{\partial\varepsilon_{k,\nu}}\Big)
-R_{\theta_0}\Big(\frac{\overline{u}_{(x_{k,\nu},\varepsilon_{k,\nu})}}{\partial\varepsilon_{k,\nu}}\Big)
+r_\infty\Big(\frac{\overline{u}_{(x_{k,\nu},\varepsilon_{k,\nu})}}{\partial\varepsilon_{k,\nu}}\Big)^{1+\frac{2}{n}}\Big\|_{L^{\frac{2n+2}{n+2}}(M)}
=o(\varepsilon_{k,\nu}^{-1})$$
and H\"{o}lder's inequality,
we conclude that
\begin{equation*}
\begin{split}
&(1+\frac{2}{n})r_\infty\left|\int_M\overline{u}_{(x_{k,\nu},\varepsilon_{k,\nu})}^{\frac{2}{n}}
\Big(\frac{\overline{u}_{(x_{k,\nu},\varepsilon_{k,\nu})}}{\partial\varepsilon_{k,\nu}}\Big)w_\nu dV_{\theta_0}\right|\\
&\leq \Big\|(2+\frac{2}{n})\Delta_{\theta_0}\Big(\frac{\overline{u}_{(x_{k,\nu},\varepsilon_{k,\nu})}}{\partial\varepsilon_{k,\nu}}\Big)
-R_{\theta_0}\Big(\frac{\overline{u}_{(x_{k,\nu},\varepsilon_{k,\nu})}}{\partial\varepsilon_{k,\nu}}\Big)
+r_\infty\Big(\frac{\overline{u}_{(x_{k,\nu},\varepsilon_{k,\nu})}}{\partial\varepsilon_{k,\nu}}\Big)^{1+\frac{2}{n}}\Big\|_{L^{\frac{2n+2}{n+2}}(M)}\\
&\hspace{4mm}
\cdot
\left(\int_M|w_\nu|^{2+\frac{2}{n}} dV_{\theta_0}\right)^{\frac{n}{2n+2}}\\
&= o(\varepsilon_{k,\nu}^{-1})
\left(\int_M|w_\nu|^{2+\frac{2}{n}} dV_{\theta_0}\right)^{\frac{n}{2n+2}}
\end{split}
\end{equation*}
for all $1\leq k\leq m$. Since $r_\infty>0$,  we have
\begin{equation*}
\left|
\int_M\overline{u}_{(x_{k,\nu},\varepsilon_{k,\nu})}^{\frac{2}{n}}
\Big(\frac{\overline{u}_{(x_{k,\nu},\varepsilon_{k,\nu})}}{\partial\varepsilon_{k,\nu}}\Big)w_\nu dV_{\theta_0}\right|\leq o(\varepsilon_{k,\nu}^{-1})
\left(\int_M|w_\nu|^{2+\frac{2}{n}} dV_{\theta_0}\right)^{\frac{n}{2n+2}}
\end{equation*}
for all $1\leq k\leq m$. Taking the weak limit yields
\begin{equation}\label{5.20}
\int_{\mathbb{H}^n}\frac{1-|z|^4-t^2}{(t^2+(1+|z|^2)^2)^{\frac{n+4}{2}}}\widehat{w}(z,t)dV_{\theta_{\mathbb{H}^n}}=0
\end{equation}
by the definition of $\overline{u}_{(x_{j,\nu},\varepsilon_{j,\nu})}$, $\widehat{w}_\nu$,  and $\widehat{w}$.
Similarly, we have
\begin{equation*}
\begin{split}
\left|
\int_M\overline{u}_{(x_{k,\nu},\varepsilon_{k,\nu})}^{\frac{2}{n}}
T(\overline{u}_{(x_{k,\nu},\varepsilon_{k,\nu})})w_\nu dV_{\theta_0}\right|&\leq o(\varepsilon_{k,\nu}^{-2})
\left(\int_M|w_\nu|^{2+\frac{2}{n}} dV_{\theta_0}\right)^{\frac{n}{2n+2}},\\
\left|
\int_M\overline{u}_{(x_{k,\nu},\varepsilon_{k,\nu})}^{\frac{2}{n}}
Z_l(\overline{u}_{(x_{k,\nu},\varepsilon_{k,\nu})})w_\nu dV_{\theta_0}\right|&\leq o(\varepsilon_{k,\nu}^{-1})
\left(\int_M|w_\nu|^{2+\frac{2}{n}} dV_{\theta_0}\right)^{\frac{n}{2n+2}},\\
\left|
\int_M\overline{u}_{(x_{k,\nu},\varepsilon_{k,\nu})}^{\frac{2}{n}}
\overline{Z}_l(\overline{u}_{(x_{k,\nu},\varepsilon_{k,\nu})})w_\nu dV_{\theta_0}\right|&\leq o(\varepsilon_{k,\nu}^{-1})
\left(\int_M|w_\nu|^{2+\frac{2}{n}} dV_{\theta_0}\right)^{\frac{n}{2n+2}}
\end{split}
\end{equation*}
for all $1\leq k,l\leq n$, where $T=\displaystyle\frac{\partial}{\partial t}$,
$Z_l=\displaystyle\frac{\partial}{\partial z_l}+\sqrt{-1}\overline{z}_l\frac{\partial}{\partial t}$,
and $\overline{Z}_l=\displaystyle\frac{\partial}{\partial \overline{z}_l}-\sqrt{-1}z_l\frac{\partial}{\partial t}$ are tangent vector fields in $\mathbb{H}^n$.
Taking the weak limit yields
\begin{equation}\label{5.21}
\begin{split}
\int_{\mathbb{H}^n}\frac{t}{(t^2+(1+|z|^2)^2)^{\frac{n+4}{2}}}\,\widehat{w}(z,t)dV_{\theta_{\mathbb{H}^n}}&=0,\\
\int_{\mathbb{H}^n}\frac{(1+|z|^2)\overline{z}_l+\sqrt{-1}\overline{z}_lt}{(t^2+(1+|z|^2)^2)^{\frac{n+4}{2}}}\,\widehat{w}(z,t)dV_{\theta_{\mathbb{H}^n}}&=0\mbox{ for }1\leq k\leq n,\\
\int_{\mathbb{H}^n}\frac{(1+|z|^2)z_l-\sqrt{-1}z_lt}{(t^2+(1+|z|^2)^2)^{\frac{n+4}{2}}}\,\widehat{w}(z,t)dV_{\theta_{\mathbb{H}^n}}&=0\mbox{ for }1\leq k\leq n.
\end{split}
\end{equation}

Now we define $\widetilde{w}(z,t)=\widehat{w}(\overline{z},t)$ for $(z,t)\in\mathbb{H}^n\subset\mathbb{C}^n\times\mathbb{R}$. Then
it follows from (\ref{5.17})-(\ref{5.21})
that
\begin{eqnarray}\label{5.26}
&&\int_{\mathbb{H}^n}\frac{1}{t^2+(1+|z|^2)^2}\,\widetilde{w}(z,t)^2dV_{\theta_{\mathbb{H}^n}}>0,\\
\label{5.27}
&&\int_{\mathbb{H}^n}|\nabla_{\theta_{\mathbb{H}^n}}\widetilde{w}(z,t)|^2_{\theta_{\mathbb{H}^n}}dV_{\theta_{\mathbb{H}^n}}
\leq(n+2)n\int_{\mathbb{H}^n}\frac{1}{t^2+(1+|z|^2)^2}\,\widetilde{w}(z,t)^2dV_{\theta_{\mathbb{H}^n}},
\end{eqnarray}
and
\begin{equation}\label{5.25}
\begin{split}
&\int_{\mathbb{H}^n}\frac{1}{(t^2+(1+|z|^2)^2)^{\frac{n+2}{2}}}\,\widetilde{w}(z,t)dV_{\theta_{\mathbb{H}^n}}=0,\\
&\int_{\mathbb{H}^n}\frac{1-|z|^4-t^2}{(t^2+(1+|z|^2)^2)^{\frac{n+4}{2}}}\,\widetilde{w}(z,t)dV_{\theta_{\mathbb{H}^n}}=0,\\
&\int_{\mathbb{H}^n}\frac{t}{(t^2+(1+|z|^2)^2)^{\frac{n+4}{2}}}\,\widetilde{w}(z,t)dV_{\theta_{\mathbb{H}^n}}=0,\\
&\int_{\mathbb{H}^n}\frac{(1+|z|^2)\overline{z}_l-\sqrt{-1}\overline{z}_lt}{(t^2+(1+|z|^2)^2)^{\frac{n+4}{2}}}\,\widetilde{w}(z,t)dV_{\theta_{\mathbb{H}^n}}=0\mbox{ for }1\leq k\leq n,\\
&\int_{\mathbb{H}^n}\frac{(1+|z|^2)z_l+\sqrt{-1}z_lt}{(t^2+(1+|z|^2)^2)^{\frac{n+4}{2}}}\,\widetilde{w}(z,t)dV_{\theta_{\mathbb{H}^n}}=0\mbox{ for }1\leq k\leq n.
\end{split}
\end{equation}
We claim that any $\widetilde{w}$ satisfying (\ref{5.27}) and (\ref{5.25}) must be zero identically.
Indeed, if we define $\widetilde{u}(x)=|1+x_{n+1}|^{-n}\widetilde{w}\circ F(x)$ where
$F:S^{2n+1}\rightarrow \mathbb{H}^n$ is given by
$$F(x)=\left(\frac{x_1}{1+x_{n+1}},\cdots,\frac{x_n}{1+x_{n+1}},
Re\big(\sqrt{-1}\frac{1-x_{n+1}}{1+x_{n+1}}\big)\right),$$
then we can rewrite (\ref{5.25}) as (see p.177 in \cite{Jerison&Lee3})
\begin{equation}\label{5.22}
\begin{split}
&\int_{S^{2n+1}}\widetilde{u}\,dV_{\theta_{S^{2n+1}}}=0,\\
&\int_{S^{2n+1}}x_k\widetilde{u}\,dV_{\theta_{S^{2n+1}}}=\int_{S^{2n+1}}\overline{x}_k\widetilde{u}\,dV_{\theta_{S^{2n+1}}}=0\mbox{ for }1\leq k\leq n+1,
\end{split}
\end{equation}
where $x=(x_1,\cdots,x_{n+1})\in S^{2n+1}\subset\mathbb{C}^{n+1}$, because
$F^{-1}:\mathbb{H}^n\rightarrow S^{2n+1}$ is given by
\begin{equation*}
F^{-1}(z,t)=\left(\frac{2 z}{1+|z|^2-\sqrt{-1}t},
\frac{1-|z|^2+\sqrt{-1}t}{1+|z|^2-\sqrt{-1}t}\right)
\end{equation*} where
$(z,t)\in\mathbb{H}^n\subset\mathbb{C}^n\times\mathbb{R}$. The eigenfunctions and eigenvalues of
$\Delta_{\theta_{S^{2n+1}}}$ are well-known----see \cite{Folland}. Namely, $1$ is
an eigenfunction with respect to the eigenvalue $\lambda_0=0$, and
$x_k, \overline{x}_k$, where $1\leq k\leq n+1$ are
eigenfunctions with respect to the eigenvalue $\lambda_1=n/2$.
Using (\ref{5.22}), we see that
\begin{equation}\label{5.23}
\int_{S^{2n+1}}|\nabla_{\theta_{S^{2n+1}}}\widetilde{u}|_{\theta_{S^{2n+1}}}^2dV_{\theta_{S^{2n+1}}}\geq \lambda_2\int_{S^{2n+1}}\widetilde{u}^2dV_{\theta_{S^{2n+1}}},
\end{equation}
where $\lambda_2>n/2$ is the second eigenvalue of $\lambda_{\theta_{S^{2n+1}}}$. On the other hand,
(\ref{5.27}) can be written as (see p.177 in \cite{Jerison&Lee3})
\begin{equation*}
\int_{S^{2n+1}}\left(|\nabla_{\theta_{S^{2n+1}}}\widetilde{u}|_{\theta_{S^{2n+1}}}^2+\frac{n^2}{4}\widetilde{u}^2\right)dV_{\theta_{S^{2n+1}}}\leq
\frac{(n+2)n}{4}\int_{S^{2n+1}}\widetilde{u}^2dV_{\theta_{S^{2n+1}}}.
\end{equation*}
Combining this with (\ref{5.23}), we conclude that $\widetilde{u}$ is zero identically, which implies that
$\widetilde{w}$ is zero identically. But this contradicts (\ref{5.26}). This proves the assertion.
\end{proof}

\begin{cor}\label{Cor5.5}
If $\nu$ is sufficiently large, then
$$\frac{n+2}{n}r_\infty\int_Mv_\nu^{\frac{2}{n}}w_\nu^2\,dV_{\theta_0}\leq (1-c)
\int_M\left((2+\frac{2}{n})|\nabla_{\theta_0}w_\nu|^2_{\theta_0}+R_{\theta_0}w_\nu^2\right)dV_{\theta_0}$$
for some positive constant $c$ which is independent of $\nu$.
\end{cor}
\begin{proof}
Note that
$$\int_M\Big|v_\nu^{\frac{2}{n}}-\sum_{k=1}^m\overline{u}_{(x_{k,\nu},\varepsilon_{k,\nu})}^{\frac{2}{n}}\Big|dV_{\theta_0}=o(1)$$
by (\ref{5.3}) and Proposition \ref{Prop5.2}. Thus, Corollary \ref{Cor5.5}
follows from Proposition \ref{Prop5.4}.
\end{proof}

\begin{prop}\label{Prop5.6}
If $\nu$ is sufficiently large, the energy of $v_\nu$ satisfies the estimate
$$E(v_\nu)\leq\Big(\sum_{k=1}^mE(\overline{u}_{(x_{k,\nu},\varepsilon_{k,\nu})})^{n+1}\Big)^{\frac{1}{n+1}}.$$
\end{prop}
\begin{proof}
By H\"{o}lder's inequality, we have
\begin{equation}\label{5.9}
\begin{split}
&\left(\sum_{k=1}^mE(\overline{u}_{(x_{k,\nu},\varepsilon_{k,\nu})})^{n+1}\right)^{\frac{1}{n+1}}
\left(\int_Mv_\nu^{2+\frac{2}{n}}dV_{\theta_0}\right)^{\frac{n}{n+1}}\\
&\geq \int_M\left(\sum_{k=1}^mF(\overline{u}_{(x_{k,\nu},\varepsilon_{k,\nu})})^{n+1}
\,\overline{u}_{(x_{k,\nu},\varepsilon_{k,\nu})}^{2+\frac{2}{n}}\right)^{\frac{1}{n+1}}v_{\nu}^2\,dV_{\theta_0}\\
&\geq \int_M\sum_{k=1}^m\alpha_{k,\nu}^2F(\overline{u}_{(x_{k,\nu},\varepsilon_{k,\nu})})
\overline{u}_{(x_{k,\nu},\varepsilon_{k,\nu})}^{2+\frac{2}{n}}dV_{\theta_0}\\
&\hspace{4mm}+
2\int_M\sum_{1\leq i<j\leq m}\alpha_{i,\nu}\alpha_{j,\nu}
 \Big(F(\overline{u}_{(x_{i,\nu},\varepsilon_{i,\nu})})^{n+1}
\,\overline{u}_{(x_{i,\nu},\varepsilon_{i,\nu})}^{2+\frac{2}{n}}\\
&\hspace{12mm}
+F(\overline{u}_{(x_{j,\nu},\varepsilon_{j,\nu})})^{n+1}
\,\overline{u}_{(x_{j,\nu},\varepsilon_{j,\nu})}^{2+\frac{2}{n}}\Big)^{\frac{1}{n+1}}
\overline{u}_{(x_{i,\nu},\varepsilon_{i,\nu})}\overline{u}_{(x_{j,\nu},\varepsilon_{j,\nu})}dV_{\theta_0}.
\end{split}
\end{equation}
Consider a pair $i<j$. We can find positive constants $c$ and $C$ independent of $\nu$ such that
\begin{equation*}
\begin{split}
&\overline{u}_{(x_{i,\nu},\varepsilon_{i,\nu})}(x)^{1+\frac{2}{n}}\overline{u}_{(x_{j,\nu},\varepsilon_{j,\nu})}(x)
\geq c\left(\frac{\varepsilon_{i,\nu}^2\varepsilon_{j,\nu}^2}{\varepsilon_{j,\nu}^4+d(x_{i,\nu},x_{j,\nu})^4}\right)^{\frac{n}{2}}\varepsilon_{i,\nu}^{-2n-2}\mbox{ and }\\
&\overline{u}_{(x_{i,\nu},\varepsilon_{i,\nu})}(x)\overline{u}_{(x_{j,\nu},\varepsilon_{j,\nu})}(x)^{1+\frac{2}{n}}
\leq C\left(\frac{\varepsilon_{i,\nu}^2\varepsilon_{j,\nu}^2}{\varepsilon_{j,\nu}^4+d(x_{i,\nu},x_{j,\nu})^4}\right)^{\frac{n+2}{2}}\varepsilon_{i,\nu}^{-2n-2}
\end{split}
\end{equation*}
if $d(x_{i,\nu},x)\leq\varepsilon_{i,\nu}$ and $\nu$ is sufficiently large. From this, it follows that
\begin{equation*}
\begin{split}
&\Big(F(\overline{u}_{(x_{i,\nu},\varepsilon_{i,\nu})})^{n+1}
\,\overline{u}_{(x_{i,\nu},\varepsilon_{i,\nu})}^{2+\frac{2}{n}}
+F(\overline{u}_{(x_{j,\nu},\varepsilon_{j,\nu})})^{n+1}
\,\overline{u}_{(x_{j,\nu},\varepsilon_{j,\nu})}^{2+\frac{2}{n}}\Big)^{\frac{1}{n+1}}
\overline{u}_{(x_{i,\nu},\varepsilon_{i,\nu})}\overline{u}_{(x_{j,\nu},\varepsilon_{j,\nu})}\\
&\geq C \left(\frac{\varepsilon_{i,\nu}^2\varepsilon_{j,\nu}^2}{\varepsilon_{j,\nu}^4+d(x_{i,\nu},x_{j,\nu})^4}\right)^{\frac{n}{2}}\varepsilon_{i,\nu}^{-2n-2} 1_{\{d(x_{i,\nu},x)\leq\varepsilon_{i,\nu}\}}
+F(\overline{u}_{(x_{j,\nu},\varepsilon_{j,\nu})})\overline{u}_{(x_{i,\nu},\varepsilon_{i,\nu})}\overline{u}_{(x_{j,\nu},\varepsilon_{j,\nu})}^{1+\frac{2}{n}}
\end{split}
\end{equation*}
if $\nu$ is sufficiently large.
Integrating it over $M$, we obtain
\begin{equation}\label{save}
\begin{split}
&\int_M\Big(F(\overline{u}_{(x_{i,\nu},\varepsilon_{i,\nu})})^{n+1}
\,\overline{u}_{(x_{i,\nu},\varepsilon_{i,\nu})}^{2+\frac{2}{n}}
+F(\overline{u}_{(x_{j,\nu},\varepsilon_{j,\nu})})^{n+1}
\,\overline{u}_{(x_{j,\nu},\varepsilon_{j,\nu})}^{2+\frac{2}{n}}\Big)^{\frac{1}{n+1}}
\overline{u}_{(x_{i,\nu},\varepsilon_{i,\nu})}\overline{u}_{(x_{j,\nu},\varepsilon_{j,\nu})}dV_{\theta_0}\\
&\geq C \left(\frac{\varepsilon_{i,\nu}^2\varepsilon_{j,\nu}^2}{\varepsilon_{j,\nu}^4+d(x_{i,\nu},x_{j,\nu})^4}\right)^{\frac{n}{2}}
+ \int_M
F(\overline{u}_{(x_{j,\nu},\varepsilon_{j,\nu})})\overline{u}_{(x_{i,\nu},\varepsilon_{i,\nu})}
\overline{u}_{(x_{j,\nu},\varepsilon_{j,\nu})}^{1+\frac{2}{n}}dV_{\theta_0}.
\end{split}
\end{equation}
if $\nu$ is sufficiently large. Combining this with (\ref{5.9}),
we get
\begin{equation}\label{5.10.0}
\begin{split}
&\left(\sum_{k=1}^mE(\overline{u}_{(x_{k,\nu},\varepsilon_{k,\nu})})^{n+1}\right)^{\frac{1}{n+1}}
\left(\int_Mv_\nu^{2+\frac{2}{n}}dV_{\theta_0}\right)^{\frac{n}{n+1}}\\
&\geq \int_M\sum_{k=1}^m\alpha_{k,\nu}^2F(\overline{u}_{(x_{k,\nu},\varepsilon_{k,\nu})})
\overline{u}_{(x_{k,\nu},\varepsilon_{k,\nu})}^{2+\frac{2}{n}}dV_{\theta_0}\\
&\hspace{4mm}+
2\int_M\sum_{1\leq i<j\leq m}\alpha_{i,\nu}\alpha_{j,\nu}
F(\overline{u}_{(x_{j,\nu},\varepsilon_{j,\nu})})\overline{u}_{(x_{i,\nu},\varepsilon_{i,\nu})}
\overline{u}_{(x_{j,\nu},\varepsilon_{j,\nu})}^{1+\frac{2}{n}}dV_{\theta_0}\\
&\hspace{4mm}+C\sum_{1\leq i<j\leq m}
\left(\frac{\varepsilon_{i,\nu}^2\varepsilon_{j,\nu}^2}{\varepsilon_{j,\nu}^4+d(x_{i,\nu},x_{j,\nu})^4}\right)^{\frac{n}{2}}.
\end{split}
\end{equation}
By the definition of $v_\nu$ in (\ref{5.3}), we have
\begin{equation}\label{5.8}
\begin{split}
&E(v_\nu)\left(\int_Mv_\nu^{2+\frac{2}{n}}dV_{\theta_0}\right)^{\frac{n}{n+1}}\\
&=\int_M\sum_{k=1}^m\alpha_{k,\nu}^2F(\overline{u}_{(x_{k,\nu},\varepsilon_{k,\nu})})
\overline{u}_{(x_{k,\nu},\varepsilon_{k,\nu})}^{2+\frac{2}{n}}dV_{\theta_0}\\
&\hspace{4mm}-
2\int_M\sum_{1\leq i<j\leq m}\alpha_{i,\nu}\alpha_{j,\nu}
 \overline{u}_{(x_{i,\nu},\varepsilon_{i,\nu})}\Big((2+\frac{2}{n})\Delta_{\theta_0}
 \overline{u}_{(x_{j,\nu},\varepsilon_{j,\nu})}
-R_{\theta_0}\overline{u}_{(x_{j,\nu},\varepsilon_{j,\nu})}\Big)dV_{\theta_0}.
\end{split}
\end{equation}
Substituting  (\ref{5.10.0}) into (\ref{5.8}), we obtain
\begin{equation*}
\begin{split}
&E(v_\nu)\left(\int_Mv_\nu^{2+\frac{2}{n}}dV_{\theta_0}\right)^{\frac{n}{n+1}}\\
&\leq \left(\sum_{k=1}^mE(\overline{u}_{(x_{k,\nu},\varepsilon_{k,\nu})})^{n+1}\right)^{\frac{1}{n+1}}
\left(\int_Mv_\nu^{2+\frac{2}{n}}dV_{\theta_0}\right)^{\frac{n}{n+1}}\\
&\hspace{4mm}-
2\int_M\sum_{1\leq i<j\leq m}\alpha_{i,\nu}\alpha_{j,\nu}
 \overline{u}_{(x_{i,\nu},\varepsilon_{i,\nu})}\\
&\hspace{1.5cm}
\cdot\left((2+\frac{2}{n})\Delta_{\theta_0}
 \overline{u}_{(x_{j,\nu},\varepsilon_{j,\nu})}-R_{\theta_0}\overline{u}_{(x_{j,\nu},\varepsilon_{j,\nu})}
+F(\overline{u}_{(x_{j,\nu},\varepsilon_{j,\nu})})
\overline{u}_{(x_{j,\nu},\varepsilon_{j,\nu})}^{1+\frac{2}{n}}\right)dV_{\theta_0}\\
&\hspace{4mm}-C\sum_{1\leq i<j\leq m}
\left(\frac{\varepsilon_{i,\nu}^2\varepsilon_{j,\nu}^2}{\varepsilon_{j,\nu}^4+d(x_{i,\nu},x_{j,\nu})^4}\right)^{\frac{n}{2}}.
\end{split}
\end{equation*}
Since $F(\overline{u}_{(x_{j,\nu},\varepsilon_{j,\nu})})=r_\infty+o(1)$ by Lemma \ref{LemmaB.6}, it follows from Lemma \ref{LemmaB.4} and
\ref{LemmaB.5} that
\begin{equation*}
\begin{split}
&\int_M
 \overline{u}_{(x_{i,\nu},\varepsilon_{i,\nu})}\Big|(2+\frac{2}{n})\Delta_{\theta_0}
 \overline{u}_{(x_{j,\nu},\varepsilon_{j,\nu})}\\
&\hspace{2cm}
-R_{\theta_0}\overline{u}_{(x_{j,\nu},\varepsilon_{j,\nu})}
+F(\overline{u}_{(x_{j,\nu},\varepsilon_{j,\nu})})
\overline{u}_{(x_{j,\nu},\varepsilon_{j,\nu})}^{1+\frac{2}{n}}\Big|\,dV_{\theta_0}\\
&\leq \int_M
 \overline{u}_{(x_{i,\nu},\varepsilon_{i,\nu})}\Big|(2+\frac{2}{n})\Delta_{\theta_0}
 \overline{u}_{(x_{j,\nu},\varepsilon_{j,\nu})}
-R_{\theta_0}\overline{u}_{(x_{j,\nu},\varepsilon_{j,\nu})}
+r_\infty
\overline{u}_{(x_{j,\nu},\varepsilon_{j,\nu})}^{1+\frac{2}{n}}\Big|\,dV_{\theta_0}\\
&\hspace{4mm}+|F(\overline{u}_{(x_{j,\nu},\varepsilon_{j,\nu})})-r_\infty|
\int_M\overline{u}_{(x_{i,\nu},\varepsilon_{i,\nu})}\overline{u}_{(x_{j,\nu},\varepsilon_{j,\nu})}^{1+\frac{2}{n}}dV_{\theta_0}\\
&\leq C(\delta^4+\delta^{2n}+\frac{\varepsilon_{j,\nu}^2}{\delta^2})\left(\frac{\varepsilon_{i,\nu}^2\varepsilon_{j,\nu}^2}{\varepsilon_{j,\nu}^4+d(x_{i,\nu},x_{j,\nu})^4}\right)^{\frac{n}{2}}
+
o(1)\left(\frac{\varepsilon_{i,\nu}^2\varepsilon_{j,\nu}^2}{\varepsilon_{j,\nu}^4+d(x_{i,\nu},x_{j,\nu})^4}\right)^{\frac{n}{2}}
\end{split}
\end{equation*}
for $i<j$. Hence, if we choose $\delta$ sufficiently small, the assertion follows.
\end{proof}

\begin{cor}\label{Cor5.7}
The energy of $v_\nu$ satisfies the estimate
$$E(v_\nu)\leq\Big(m Y(S^{2n+1})^{n+1}\Big)^{\frac{1}{n+1}}$$
if $\nu$ is sufficiently large.
\end{cor}
\begin{proof}
Using Proposition \ref{PropB.3}, we obtain
$$E(\overline{u}_{(x_{k,\nu},\varepsilon_{k,\nu})})\leq Y(S^{2n+1})$$
for all $1\leq k\leq m$. Now Corollary \ref{Cor5.7} follows from Proposition \ref{Prop5.6}.
\end{proof}

%%%%%%%%%%%%%%%%%%%%%%%%%%%%

\subsection{The case $u_\infty>0$}\label{section6}

Next we discuss the case $u_\infty>0$.

\begin{prop}\label{Prop6.1}
There exists a sequence of smooth functions $\{\psi_a: a\in\mathbb{N}\}$ and a sequence of
real numbers $\{\lambda_a: a\in\mathbb{N}\}$ with the following properties:\\
(i) For every $a\in\mathbb{N}$, the function $\psi_a$ satisfies
\begin{equation}\label{6.1}
(2+\frac{2}{n})\Delta_{\theta_0}\psi_a
-R_{\theta_0}\psi_a
+\lambda_au_\infty^{\frac{2}{n}}\psi_a=0.
\end{equation}
(ii) For all $a,b\in\mathbb{N}$, we have
\begin{equation}\label{6.2}
\int_Mu_\infty^{\frac{2}{n}}\psi_a\psi_b\,dV_{\theta_0}=\left\{
                                                          \begin{array}{ll}
                                                            1, & \hbox{if $a=b$,} \\
                                                            0, & \hbox{if $a\neq b$.}
                                                          \end{array}
                                                        \right.
\end{equation}
(iii) The span of $\{\psi_a:a\in\mathbb{N}\}$ is dense in $L^2(M)$.\\
(iv) $\lambda_a\to\infty$ as $a\to\infty$.
\end{prop}
\begin{proof}
Consider the linear operator
$$\psi\mapsto u_\infty^{-\frac{2}{n}}\Big((2+\frac{2}{n})\Delta_{\theta_0}\psi-R_{\theta_0}\psi\Big).$$
The operator is symmetric with respect to the inner product
$$(\psi_1,\psi_2)\mapsto \int_Mu_\infty^{\frac{2}{n}}\psi_1\psi_2\,dV_{\theta_0}$$
on $L^2(M)$. Hence, the assertion follows from the spectral theorem.
\end{proof}

Let
$$A=\{a\in\mathbb{N}: \lambda_a\leq \frac{n+2}{n}r_\infty\}$$
which is a finite subset of $\mathbb{N}$ by Proposition \ref{Prop6.1}(iv). We denote $\Pi$ the projection
operator
\begin{equation}\label{6.3}
\begin{split}
\Pi f&=\sum_{a\not\in A}\left(\int_M\psi_a f\,dV_{\theta_0}\right) u_\infty^{\frac{2}{n}}\psi_a=f-\sum_{a\in A}\left(\int_M\psi_a f\,dV_{\theta_0}\right) u_\infty^{\frac{2}{n}}\psi_a.
\end{split}
\end{equation}

\begin{lem}\label{Lemma6.2}
For every $1\leq p<\infty$, we can find a constant $C$ such that
\begin{equation*}
\|f\|_{L^p(M)}\leq C \Big\|(2+\frac{2}{n})\Delta_{\theta_0}f
-R_{\theta_0}f
+\frac{n+2}{n}r_\infty u_\infty^{\frac{2}{n}}f\Big\|_{L^p(M)}
+C\sup_{a\in A}\left|\int_Mu_\infty^{\frac{2}{n}}\psi_a f\,dV_{\theta_0}\right|.
\end{equation*}
\end{lem}
\begin{proof}
Suppose it is not true. By compactness, we can find
a function $f\in L^p(M)$ satisfying $\|f\|_{L^p(M)}=1$,
\begin{equation}\label{6.4}
\int_Mu_\infty^{\frac{2}{n}}\psi_a f\,dV_{\theta_0}=0\hspace{2mm}\mbox{ for all }a\in A,
\end{equation}
and
\begin{equation}\label{6.5}
(2+\frac{2}{n})\Delta_{\theta_0}f
-R_{\theta_0}f
+\frac{n+2}{n}r_\infty u_\infty^{\frac{2}{n}}f=0
\end{equation}
in the sense of distributions. Multiply (\ref{6.5}) by $\psi_a$ and integrate over $M$, we obtain
for all $a\in\mathbb{N}$
\begin{equation}\label{6.6}
\begin{split}
0&=\int_M\psi_a\Big((2+\frac{2}{n})\Delta_{\theta_0}f
-R_{\theta_0}f
+\frac{n+2}{n}r_\infty u_\infty^{\frac{2}{n}}f\Big)dV_{\theta_0}\\
&=\Big(\frac{n+2}{n}r_\infty-\lambda_a\Big)\int_Mu_\infty^{\frac{2}{n}}\psi_a f\,dV_{\theta_0}
\end{split}
\end{equation}
by (\ref{6.1}). In particular, for $a\not\in A$, we have $\lambda_a>\displaystyle\frac{n+2}{n}r_\infty$, which implies that
$$\int_Mu_\infty^{\frac{2}{n}}\psi_a f\,dV_{\theta_0}=0\hspace{2mm}\mbox{ for all }a\not\in A.$$
Combining this with (\ref{6.4}), we have $f\equiv 0$ which contradicts $\|f\|_{L^p(M)}=1$.
This proves the assertion.
\end{proof}

\begin{lem}\label{Lemma6.3}
(i) There exists a constant $C$ such that
\begin{equation*}
\begin{split}
\|f\|_{L^{1+\frac{2}{n}}(M)}&\leq C
\Big\|\Pi\Big((2+\frac{2}{n})\Delta_{\theta_0}f
-R_{\theta_0}f
+\frac{n+2}{n}r_\infty u_\infty^{\frac{2}{n}}f\Big)\Big\|_{L^{\frac{(n+1)(n+2)}{n^2+2n+2}}(M)}\\
&\hspace{4mm}+C\sup_{a\in A}\left|\int_Mu_\infty^{\frac{2}{n}}\psi_a f\,dV_{\theta_0}\right|.
\end{split}
\end{equation*}
(ii) There exists a constant $C$ such that
\begin{equation*}
\begin{split}
\|f\|_{L^{1}(M)}&\leq C
\Big\|\Pi\Big((2+\frac{2}{n})\Delta_{\theta_0}f
-R_{\theta_0}f
+\frac{n+2}{n}r_\infty u_\infty^{\frac{2}{n}}f\Big)\Big\|_{L^1(M)}\\
&\hspace{4mm}+C\sup_{a\in A}\left|\int_Mu_\infty^{\frac{2}{n}}\psi_a f\,dV_{\theta_0}\right|.
\end{split}
\end{equation*}
\end{lem}
\begin{proof}
To prove (i), we apply the result of Bramanti-Braudolini in \cite{Bramanti1} to conclude that
$$\|f\|_{S_2^p(M)}
\leq
\Big\|(2+\frac{2}{n})\Delta_{\theta_0}f
-R_{\theta_0}f
+\frac{n+2}{n}r_\infty u_\infty^{\frac{2}{n}}f\Big\|_{L^{p}(M)}
+C\|f\|_{L^p(M)}$$
for $1<p<\infty$. By Folland-Stein embedding theorem in Proposition \ref{Theorem 2.1},
$S_2^p(M)$ is continuous embedded in $L^{p^*}(M)$ where
$\displaystyle\frac{1}{p^*}=\frac{1}{p}-\frac{2}{2n+2}$.
Therefore, if we choose $p=\displaystyle\frac{(n+1)(n+2)}{n^2+2n+2}>1$,
then $p^*=1+\displaystyle\frac{2}{n}$ and we have
\begin{equation}\label{6.7}
\begin{split}
\|f\|_{L^{1+\frac{2}{n}}(M)}
&\leq
\Big\|(2+\frac{2}{n})\Delta_{\theta_0}f
-R_{\theta_0}f
+\frac{n+2}{n}r_\infty u_\infty^{\frac{2}{n}}f\Big\|_{L^{\frac{(n+1)(n+2)}{n^2+2n+2}}(M)}\\
&\hspace{4mm}
+C\|f\|_{L^{\frac{(n+1)(n+2)}{n^2+2n+2}}(M)}.
\end{split}
\end{equation}
By (\ref{6.7}) and Lemma \ref{Lemma6.2} with $p=\displaystyle\frac{(n+1)(n+2)}{n^2+2n+2}$, we get
\begin{equation}\label{6.8}
\begin{split}
\|f\|_{L^{1+\frac{2}{n}}(M)}
&\leq
\Big\|(2+\frac{2}{n})\Delta_{\theta_0}f
-R_{\theta_0}f
+\frac{n+2}{n}r_\infty u_\infty^{\frac{2}{n}}f\Big\|_{L^{\frac{(n+1)(n+2)}{n^2+2n+2}}(M)}\\
&\hspace{4mm}
+C\sup_{a\in A}\left|\int_Mu_\infty^{\frac{2}{n}}\psi_a f\,dV_{\theta_0}\right|.
\end{split}
\end{equation}

By (\ref{6.6}) and the definition of $\Pi$, we have
\begin{equation*}
\begin{split}
&\Pi\Big((2+\frac{2}{n})\Delta_{\theta_0}f
-R_{\theta_0}f
+\frac{n+2}{n}r_\infty u_\infty^{\frac{2}{n}}f\Big)\\
&=(2+\frac{2}{n})\Delta_{\theta_0}f
-R_{\theta_0}f
+\frac{n+2}{n}r_\infty u_\infty^{\frac{2}{n}}f\\
&\hspace{4mm}-\sum_{a\in A}
\Big(\frac{n+2}{n}r_\infty-\lambda_a\Big)\left(\int_Mu_\infty^{\frac{2}{n}}\psi_a f\,dV_{\theta_0}\right)
u_\infty^{\frac{2}{n}}\psi_a,
\end{split}
\end{equation*}
which implies that
\begin{equation}\label{6.9}
\begin{split}
&\Big\|(2+\frac{2}{n})\Delta_{\theta_0}f
-R_{\theta_0}f
+\frac{n+2}{n}r_\infty u_\infty^{\frac{2}{n}}f\Big\|_{L^q(M)}\\
&\leq \Big\|\Pi\Big((2+\frac{2}{n})\Delta_{\theta_0}f
-R_{\theta_0}f
+\frac{n+2}{n}r_\infty u_\infty^{\frac{2}{n}}f\Big)\Big\|_{L^q(M)}+C\sup_{a\in A}
\left|\int_Mu_\infty^{\frac{2}{n}}\psi_a f\,dV_{\theta_0}\right|
\end{split}
\end{equation}
for $1\leq q<\infty$. Now combining (\ref{6.8}) and (\ref{6.9}) with $q=\displaystyle\frac{(n+1)(n+2)}{n^2+2n+2}$,
we prove (i).
To prove (ii), we note that by Lemma \ref{Lemma6.2} with $p=1$ we have
\begin{equation*}
\|f\|_{L^1(M)}\leq C \Big\|(2+\frac{2}{n})\Delta_{\theta_0}f
-R_{\theta_0}f
+\frac{n+2}{n}r_\infty u_\infty^{\frac{2}{n}}f\Big\|_{L^1(M)}
+C\sup_{a\in A}\left|\int_Mu_\infty^{\frac{2}{n}}\psi_a f\,dV_{\theta_0}\right|.
\end{equation*}
Combining this with (\ref{6.9}) with $q=1$, we prove (ii).
\end{proof}

%%%%%%%%%%%%%%%%%%%%%%
%%%%%%%%%%%%%%%%%%%%%%
\begin{lem}\label{Lemma6.4}
There exists a positive real number $\zeta$ with the following property: for every
vector $z=(z_1,z_2,...,z_{|A|})\in\mathbb{R}^{|A|}$ with $|z|\leq\zeta$, there exists a smooth function $\overline{u}_z$ such that \begin{equation*}
\int_Mu_\infty^{\frac{2}{n}}(\overline{u}_z-u_\infty)\psi_a\,dV_{\theta_0}=z_a
\end{equation*}
for all $a\in A$ and
\begin{equation*}
\Pi\Big((2+\frac{2}{n})\Delta_{\theta_0}\overline{u}_z
-R_{\theta_0}\overline{u}_z
+\frac{n+2}{n}r_\infty u_\infty^{\frac{2}{n}}\overline{u}_z\Big)=0.
\end{equation*}
Furthermore, the map $z\mapsto \overline{u}_z$ defined in $(i)$ is real analytic.
\end{lem}
\begin{proof}
This is the consequence of the implicit function theorem.
\end{proof}

%%%%%%%%%%%%%%%%%%%%%%%%%%%%%%%%%
%%%%%%%%%%%%%%%%%%%%%%%%%%%%%%%%%

\begin{lem}\label{Lemma6.5}
There exists a real number $0<\gamma<1$ such that
\begin{equation*}
\begin{split}
E(\overline{u}_z)-E(u_\infty)
\leq C\sup_{a\in A}\left|\int_M\psi_a
\Big((2+\frac{2}{n})\Delta_{\theta_0}\overline{u}_z
-R_{\theta_0}\overline{u}_z
+r_\infty \overline{u}_z^{1+\frac{2}{n}}\Big)dV_{\theta_0}\right|^{1+\gamma}
\end{split}
\end{equation*}
if $z$ is sufficiently small.
\end{lem}
\begin{proof}
Note that the function $z\mapsto E(\overline{u}_z)$ is real analytic by Lemma \ref{Lemma6.4}.
According to the results of Lojasiewicz (see \cite{Simon}), there exists a real number $0<\gamma<1$ such that
\begin{equation}\label{6.10}
|E(\overline{u}_z)-E(u_\infty)|\leq C\sup_{a\in A}\left|\frac{\partial}{\partial z_a}E(\overline{u}_z)\right|^{1+\gamma}
\end{equation}
if $z$ is sufficiently small.
The partial derivatives of the function $z\mapsto E(\overline{u}_z)$ are given by
\begin{equation}\label{6.11}
\begin{split}
\frac{\partial}{\partial z_a}E(\overline{u}_z)
&=-2\frac{\int_M\widetilde{\psi}_{a,z}
\Big((2+\frac{2}{n})\Delta_{\theta_0}\overline{u}_z
-R_{\theta_0}\overline{u}_z
+r_\infty \overline{u}_z^{1+\frac{2}{n}}\Big)dV_{\theta_0}}{\Big(\int_M \overline{u}_z^{2+\frac{2}{n}}dV_{\theta_0}\Big)^{\frac{n}{n+1}}}\\
&\hspace{4mm}-2(F(\overline{u}_z)-r_\infty)\frac{\int_M \overline{u}_z^{1+\frac{2}{n}}\widetilde{\psi}_{a,z}\,dV_{\theta_0}}{\Big(\int_M \overline{u}_z^{2+\frac{2}{n}}dV_{\theta_0}\Big)^{\frac{n}{n+1}}},
\end{split}
\end{equation}
where $\widetilde{\psi}_{a,z}=\frac{\partial}{\partial z_a}\overline{u}_z$ for $a\in A$.
Since
$\displaystyle\int_Mu_\infty^{\frac{2}{n}}(\overline{u}_z-u_\infty)\psi_b\,dV_{\theta_0}=z_b$
for all $b\in A$ by Lemma \ref{Lemma6.4}, by differentiating it with respect to $z_a$, we obtain
\begin{equation}\label{6.12}
\int_Mu_\infty^{\frac{2}{n}}\widetilde{\psi}_{a,z}\psi_b\,dV_{\theta_0}=
\left\{
  \begin{array}{ll}
    1, & \hbox{if $a=b$,} \\
    0, & \hbox{if $a\neq b$.}
  \end{array}
\right.
\end{equation}
On the other hand, by Lemma \ref{Lemma6.4}, we have
\begin{equation}\label{6.13}
\Pi\Big((2+\frac{2}{n})\Delta_{\theta_0}\overline{u}_z
-R_{\theta_0}\overline{u}_z
+\frac{n+2}{n}r_\infty u_\infty^{\frac{2}{n}}\overline{u}_z\Big)=0.
\end{equation}
Differentiate it with respect to $z_a$, we get
\begin{equation}\label{6.14}
\Pi\Big((2+\frac{2}{n})\Delta_{\theta_0}\widetilde{\psi}_{a,z}
-R_{\theta_0}\widetilde{\psi}_{a,z}
+\frac{n+2}{n}r_\infty u_\infty^{\frac{2}{n}}\widetilde{\psi}_{a,z}\Big)=0.
\end{equation}
By definition of $\Pi$ in (\ref{6.3}) and (\ref{6.13}),
we have
\begin{equation}\label{6.15}
\begin{split}
&(2+\frac{2}{n})\Delta_{\theta_0}\overline{u}_z
-R_{\theta_0}\overline{u}_z
+\frac{n+2}{n}r_\infty u_\infty^{\frac{2}{n}}\overline{u}_z\\
&=\sum_{a\in A}\left(\int_M\psi_a\Big((2+\frac{2}{n})\Delta_{\theta_0}\overline{u}_z
-R_{\theta_0}\overline{u}_z
+\frac{n+2}{n}r_\infty u_\infty^{\frac{2}{n}}\overline{u}_z\Big)dV_{\theta_0}\right)u_\infty^{\frac{2}{n}}\psi_a.
\end{split}
\end{equation}
By definition of $F$, we also have
\begin{equation}\label{6.16}
\begin{split}
F(\overline{u}_z)-r_\infty
&=
-\frac{\int_M\big((2+\frac{2}{n})\Delta_{\theta_0}\overline{u}_z-R_{\theta_0}\overline{u}_z+r_\infty \overline{u}_z^{1+\frac{2}{n}}\big)\overline{u}_zdV_{\theta_0}}
{\int_M\overline{u}_z^{2+\frac{2}{n}}dV_{\theta_0}}.
\end{split}
\end{equation}
Substituting (\ref{6.16}) into (\ref{6.11}), we obtain
\begin{equation*}
\begin{split}
&\frac{\partial}{\partial z_a}E(\overline{u}_z)\\
&=-2\frac{\int_M\psi_a\Big((2+\frac{2}{n})\Delta_{\theta_0}\overline{u}_z
-R_{\theta_0}\overline{u}_z
+\frac{n+2}{n}r_\infty u_\infty^{\frac{2}{n}}\overline{u}_z\Big)dV_{\theta_0}}{\Big(\int_M \overline{u}_z^{2+\frac{2}{n}}dV_{\theta_0}\Big)^{\frac{n}{n+1}}}\\
&\hspace{4mm}+2
\frac{\sum_{b\in A}\left(\int_M\psi_b\Big((2+\frac{2}{n})\Delta_{\theta_0}\overline{u}_z
-R_{\theta_0}\overline{u}_z
+\frac{n+2}{n}r_\infty u_\infty^{\frac{2}{n}}\overline{u}_z\Big)dV_{\theta_0}\right)\int_Mu_\infty^{\frac{2}{n}}\overline{u}_z\psi_bdV_{\theta_0}}
{\int_M\overline{u}_z^{2+\frac{2}{n}}dV_{\theta_0}}\\
&\hspace{4mm}
\cdot\frac{\int_M \overline{u}_z^{1+\frac{2}{n}}\widetilde{\psi}_{a,z}\,dV_{\theta_0}}{\Big(\int_M \overline{u}_z^{2+\frac{2}{n}}dV_{\theta_0}\Big)^{\frac{n}{n+1}}},
\end{split}
\end{equation*}
where the second equality follows from (\ref{6.15}), and the last equality follows from (\ref{6.12}).
Thus we conclude that
$$\sup_{a\in A}\left|\frac{\partial}{\partial z_a}E(\overline{u}_z)\right|
\leq C\sup_{a\in A}
\left|\int_M\psi_a
\Big((2+\frac{2}{n})\Delta_{\theta_0}\overline{u}_z
-R_{\theta_0}\overline{u}_z
+r_\infty \overline{u}_z^{1+\frac{2}{n}}\Big)dV_{\theta_0}\right|.$$
Combining this with (\ref{6.10}), we prove Lemma \ref{Lemma6.5}.
\end{proof}

Now, for every $\nu\in\mathbb{N}$, we denote by $\mathcal{A}_\nu$ the set of all pairs
$(z, (x_k,\varepsilon_k,\alpha_k)_{1\leq k\leq m})\in \mathbb{R}^{|A|}\times(M\times\mathbb{R}_+\times\mathbb{R}_+)^m$
such that
\begin{equation}\label{6.17}
|z|\leq\zeta
\end{equation}
and
\begin{equation}\label{6.18}
d(x_k,x_{k,\nu}^*)\leq\varepsilon_{k,\nu}^*,\hspace{4mm}\frac{1}{2}\leq\frac{\varepsilon_{k}}{\varepsilon_{k,\nu}^*}\leq 2,
\hspace{4mm}  \frac{1}{2}\leq\alpha_k\leq 2
\end{equation}
for all $1\leq k\leq m$. Moreover, we can find
$(z_\nu, (x_{k,\nu},\varepsilon_{k,\nu},\alpha_{k,\nu})_{1\leq k\leq m})\in \mathcal{A}_\nu$
such that
\begin{equation}\label{6.19}
\begin{split}
&\int_M\left((2+\frac{2}{n})\Big|\nabla_{\theta_0}\Big(u_\nu-\overline{u}_{z_\nu}-\sum_{k=1}^m\alpha_{k,\nu}\overline{u}_{(x_{k,\nu},\varepsilon_{k,\nu})}\Big)\Big|_{\theta_0}^2
\right.\\
&\hspace{2cm}\left.
+R_{\theta_0}\Big(u_\nu-\overline{u}_{z_\nu}-\sum_{k=1}^m\alpha_{k,\nu}\overline{u}_{(x_{k,\nu},\varepsilon_{k,\nu})}\Big)^2\right)dV_{\theta_0}\\
&\leq
\int_M\left((2+\frac{2}{n})\Big|\nabla_{\theta_0}\Big(u_\nu-\overline{u}_{z}-\sum_{k=1}^m\alpha_{k}\overline{u}_{(x_{k},\varepsilon_{k})}\Big)\Big|_{\theta_0}^2
\right.\\
&\hspace{2cm}\left.
+R_{\theta_0}\Big(u_\nu-\overline{u}_{z}-\sum_{k=1}^m\alpha_{k}\overline{u}_{(x_{k},\varepsilon_{k})}\Big)^2\right)dV_{\theta_0}
\end{split}
\end{equation}
for all $(z,(x_k,\varepsilon_k,\alpha_k)_{1\leq k\leq m})\in \mathcal{A}_\nu$.
%%%%%%%%%%%%%
%%%%%%%%%%%%%%
Then we can obtain the following two propositions, which are similar to Proposition \ref{Prop5.1} and Proposition \ref{Prop5.2}.

\begin{prop}\label{Prop6.6}
(i)  For all $i\neq j$, we have
$$\frac{\varepsilon_{i,\nu}^2}{\varepsilon_{j,\nu}^2}+\frac{\varepsilon_{j,\nu}^2}{\varepsilon_{i,\nu}^2}
+\frac{d(x_{i,\nu},x_{j,\nu})^4}{\varepsilon_{i,\nu}^2\varepsilon_{j,\nu}^2}\rightarrow\infty\hspace{2
mm}\mbox{ as }\hspace{2 mm}\nu\rightarrow\infty,$$ %
(ii) We have
$$\Big\|u_\nu-\overline{u}_{z_\nu}-\sum_{k=1}^m\alpha_{k,\nu}\overline{u}_{(x_{k,\nu},\varepsilon_{k,\nu})}\Big\|_{S^2_1(M)}\rightarrow
0\hspace{2 mm}\mbox{ as }\hspace{2 mm}\nu\rightarrow\infty.$$
\end{prop}

\begin{prop}\label{Prop6.7}
We have
$$|z_\nu|=o(1)$$
and
\begin{equation*}
d(x_{k,\nu},x_{k,\nu}^*)\leq o(1)\varepsilon_{k,\mu}^*,\hspace{4mm}\frac{\varepsilon_{k,\nu}}{\varepsilon_{k,\mu}^*}=1+o(1),
\hspace{4mm}  \alpha_{k,\nu}=1+o(1)
\end{equation*}
for all $1\leq k\leq m$. In particular, $(z_\nu,(x_{k,\nu},\varepsilon_{k,\nu},\alpha_{k,\nu})_{1\leq k\leq m})$
is an interior point of $\mathcal{A}_\nu$ if $\nu$ is sufficiently large.
\end{prop}
%%%%%%%%%%%%%%%%
%%%%%%%%%%%%%%%%%%

Now we decompose the function $u_\nu$ as
\begin{equation}\label{6.20.5}
u_\nu=v_\nu+w_\nu
\end{equation}
where
\begin{equation}\label{6.20}
v_\nu=\overline{u}_{z_\nu}+\sum_{k=1}^m\alpha_{k,\nu}\overline{u}_{(x_{k,\nu},\varepsilon_{k,\nu})}
\end{equation}
and
\begin{equation}\label{6.21}
w_\nu=u_\nu-\overline{u}_{z_\nu}-\sum_{k=1}^m\alpha_{k,\nu}\overline{u}_{(x_{k,\nu},\varepsilon_{k,\nu})}.
\end{equation}
By Proposition \ref{Prop6.6}(ii), the function $w_\nu$ satisfies
\begin{equation}\label{6.22}
\int_M\left((2+\frac{2}{n})|\nabla_{\theta_0}w_\nu|^2_{\theta_0}+R_{\theta_0}w_\nu^2\right)dV_{\theta_0}=o(1).
\end{equation}

\begin{prop}\label{Prop6.8}
For every $a\in A$, we have
\begin{equation}\label{6.23}
\left|\int_Mu_\infty^{\frac{2}{n}}\psi_a w_\nu\, dV_{\theta_0}\right|\leq o(1)
\int_M|w_\nu| dV_{\theta_0}.
\end{equation}
\end{prop}
\begin{proof}
As above, let $\widetilde{\psi}_{a,z}=\frac{\partial}{\partial z_a}\overline{u}_z$. By definition
of
$(z_\nu, (x_{k,\nu},\varepsilon_{k,\nu},\alpha_{k,\nu})_{1\leq k\leq m})$
in (\ref{6.19}), we have
\begin{equation*}
\begin{split}
0&=\frac{1}{2}\frac{d}{dz_a}
\int_M\left((2+\frac{2}{n})\Big|\nabla_{\theta_0}\Big(u_\nu-\overline{u}_{z_\nu}-\sum_{k=1}^m\alpha_{k,\nu}\overline{u}_{(x_{k,\nu},\varepsilon_{k,\nu})}\Big)\Big|_{\theta_0}^2
\right.
\\
&\left.\left.\hspace{2cm}+R_{\theta_0}\Big(u_\nu-\overline{u}_{z_\nu}-\sum_{k=1}^m\alpha_{k,\nu}\overline{u}_{(x_{k,\nu},\varepsilon_{k,\nu})}\Big)^2\right)dV_{\theta_0}
\right|_{z=z_\nu}\\
&=\int_M\Big((2+\frac{2}{n})\Delta_{\theta_0}\widetilde{\psi}_{a,z}
-R_{\theta_0}\widetilde{\psi}_{a,z}\Big)w_\nu\,dV_{\theta_0}.
\end{split}
\end{equation*}
This implies together with Proposition \ref{Prop6.1}(i) that
\begin{equation*}
\begin{split}
\lambda_a\int_Mu_\infty^{\frac{2}{n}}\psi_a w_\nu\, dV_{\theta_0}
&=\int_M\Big((2+\frac{2}{n})\Delta_{\theta_0}(\widetilde{\psi}_{a,z}-\psi_a)
-R_{\theta_0}(\widetilde{\psi}_{a,z}-\psi_a)\Big)w_\nu\,dV_{\theta_0}
\end{split}
\end{equation*}
Since $\lambda_a>0$, we conclude that
$$\left|\int_Mu_\infty^{\frac{2}{n}}\psi_a w_\nu\, dV_{\theta_0}\right|\leq o(1)
\int_M|w_\nu| dV_{\theta_0}$$
for all $a\in A$. This proves the assertion.
\end{proof}

\begin{prop}\label{Prop6.9}
If $\nu$ is sufficiently large, then
$$\frac{n+2}{n}r_\infty\int_M\Big(u_\infty^{\frac{2}{n}}+\sum_{k=1}^m\overline{u}_{(x_{k,\nu},\varepsilon_{k,\nu})}^{\frac{2}{n}}\Big)w_\nu^2\,dV_{\theta_0}\leq (1-c)
\int_M\left((2+\frac{2}{n})|\nabla_{\theta_0}w_\nu|^2_{\theta_0}+R_{\theta_0}w_\nu^2\right)dV_{\theta_0}$$
for some positive constant $c$ which is independent of $\nu$.
\end{prop}
\begin{proof}
Suppose this is not true. Upon rescaling, we obtain a sequence
of functions $\{\widetilde{w}_\nu:\nu\in\mathbb{N}\}$ such that
\begin{equation}\label{6.39}
\int_M\left((2+\frac{2}{n})|\nabla_{\theta_0}\widetilde{w}_\nu|^2_{\theta_0}+R_{\theta_0}\widetilde{w}_\nu^2\right)dV_{\theta_0}=1
\end{equation}
and
\begin{equation}\label{6.40}
\lim_{\nu\to\infty}\frac{n+2}{n}r_\infty\int_M
\Big(u_\infty^{\frac{2}{n}}+\sum_{k=1}^m\overline{u}_{(x_{k,\nu},\varepsilon_{k,\nu})}^{\frac{2}{n}}\Big)\widetilde{w}_\nu^2\,dV_{\theta_0}\geq 1.
\end{equation}
Note that
\begin{equation}\label{6.41}
\int_M|\widetilde{w}_\nu|^{2+\frac{2}{n}}dV_{\theta_0}\leq Y(M,\theta_0)^{-\frac{n+1}{n}}
\end{equation}
by (\ref{6.39}). In view of Proposition \ref{Prop6.6}, we can find a sequence $\{N_\nu:\nu\in\mathbb{N}\}$ such that
$N_\nu\to\infty$, $N_\nu\varepsilon_{j,\nu}\to 0$ for all $1\leq j\leq m$, and
\begin{equation}\label{6.42}
\frac{1}{N_\nu}\frac{\varepsilon_{j,\nu}+d(x_{i,\nu},x_{j,\nu})}{\varepsilon_{i,\nu}}\to\infty
\end{equation}
for all $i<j$. Let
\begin{equation}\label{6.43}
\Omega_{j,\nu}=B_{N_\nu\varepsilon_{j,\nu}}\setminus\bigcup_{i=1}^{j-1}B_{N_\nu\varepsilon_{i,\nu}}(x_{i,\nu})
\end{equation}
for every $1\leq j\leq m$. In view of (\ref{6.39}) and (\ref{6.40}), there are two cases to be considered:

\textit{Case 1.} Suppose that
\begin{equation}\label{6.44}
\lim_{\nu\to\infty}\int_Mu_\infty^{\frac{2}{n}}\widetilde{w}_\nu^2\,dV_{\theta_0}>0
\end{equation}
and
\begin{equation}\label{6.45}
\lim_{\nu\to\infty}\int_{M\setminus\bigcup_{j=1}^m\Omega_{j,\nu}}\left((2+\frac{2}{n})|\nabla_{\theta_0}\widetilde{w}_\nu|^2_{\theta_0}+R_{\theta_0}\widetilde{w}_\nu^2\right)dV_{\theta_0}
\leq\lim_{\nu\to\infty}\frac{n+2}{n}r_\infty\int_Mu_\infty^{\frac{2}{n}}\widetilde{w}_\nu^2\,dV_{\theta_0}.
\end{equation}
Let $\widetilde{w}$ be the weak limit of the sequence $\{\widetilde{w}_\nu: \nu\in\mathbb{N}\}$. Then, by
(\ref{6.44}) and (\ref{6.45}),
the function $\widetilde{w}$ satisfies
\begin{equation}\label{6.46}
\int_Mu_\infty^{\frac{2}{n}}\widetilde{w}^2\,dV_{\theta_0}>0
\end{equation}
and
$$\int_{M}\left((2+\frac{2}{n})|\nabla_{\theta_0}\widetilde{w}|^2_{\theta_0}+R_{\theta_0}\widetilde{w}^2\right)dV_{\theta_0}
\leq\frac{n+2}{n}r_\infty\int_Mu_\infty^{\frac{2}{n}}\widetilde{w}^2\,dV_{\theta_0}.$$
This implies
$$\sum_{a\in\mathbb{N}}\lambda_a\left(
\int_Mu_\infty^{\frac{2}{n}}\psi_a\widetilde{w}\,dV_{\theta_0}\right)^2
\leq\sum_{a\in\mathbb{N}}\frac{n+2}{n}r_\infty
\left(\int_Mu_\infty^{\frac{2}{n}}\psi_a\widetilde{w}\,dV_{\theta_0}\right)^2.$$
Using Proposition \ref{Prop6.8}, we obtain
$$
\int_Mu_\infty^{\frac{2}{n}}\psi_a\widetilde{w}\,dV_{\theta_0}=0$$
for all $a\in A$. Thus, we conclude that $\widetilde{w}(x)=0$ for all $x\in M$,
which contradicts to (\ref{6.46}).

\textit{Case 2.} Suppose that there exists an integer $1\leq j\leq m$ such that
\begin{equation*}
\lim_{\nu\to\infty}\int_M\overline{u}_{(x_{j,\nu},\varepsilon_{j,\nu})}^{\frac{2}{n}}\widetilde{w}_\nu^2\,dV_{\theta_0}>0
\end{equation*}
and
\begin{equation*}
\lim_{\nu\to\infty}\int_{\Omega_{j,\nu}}\left((2+\frac{2}{n})|\nabla_{\theta_0}\widetilde{w}_\nu|^2_{\theta_0}+R_{\theta_0}\widetilde{w}_\nu^2\right)dV_{\theta_0}
\leq\lim_{\nu\to\infty}\frac{n+2}{n}r_\infty\int_M\overline{u}_{(x_{j,\nu},\varepsilon_{j,\nu})}^{\frac{2}{n}}\widetilde{w}_\nu^2\,dV_{\theta_0}.
\end{equation*}
Now we can follow the same arguments in the proof of Proposition \ref{Prop5.4} to finish the proof.
More precisely, we can construct a function $\widehat{w}:\mathbb{H}^n\to\mathbb{R}$
which satisfies (\ref{5.17})-(\ref{5.21})
as in the proof of Proposition \ref{Prop5.4}, which is a contraction.
This proves the assertion.
\end{proof}

\begin{cor}\label{Cor6.10}
If $\nu$ is sufficiently large, then
$$\frac{n+2}{n}r_\infty\int_Mv_\nu^{\frac{2}{n}}w_\nu^2\,dV_{\theta_0}\leq (1-c)
\int_M\left((2+\frac{2}{n})|\nabla_{\theta_0}w_\nu|^2_{\theta_0}+R_{\theta_0}w_\nu^2\right)dV_{\theta_0}$$
for some positive constant $c$ which is independent of $\nu$.
\end{cor}
\begin{proof}
By (\ref{6.20}) and Proposition \ref{Prop6.7}, we have
$$\int_M\Big|v_\nu^{\frac{2}{n}}-u_\infty^{\frac{2}{n}}-\sum_{k=1}^m\overline{u}_{(x_{k,\nu},\varepsilon_{k,\nu})}^{\frac{2}{n}}\Big|dV_{\theta_0}=o(1).$$
Therefore, Corollary \ref{Cor6.10}
follows from Proposition \ref{Prop6.9}.
\end{proof}

\begin{lem}\label{Lemma6.11}
The difference $u_\nu-\overline{u}_{z_\nu}$ satisfies
the estimate
\begin{equation}\label{6.26.5}
\|u_\nu-\overline{u}_{z_\nu}\|_{L^{1+\frac{2}{n}}(M)}^{1+\frac{2}{n}}\leq
C\big\|u_\nu^{1+\frac{2}{n}}(R_{\theta_\nu}-r_\infty)\big\|_{L^{\frac{2n+2}{n+2}}(M)}^{1+\frac{2}{n}}
+C\sum_{k=1}^m\varepsilon_{k,\nu}^{n}
\end{equation}
if $\nu$ is sufficiently large.
\end{lem}
\begin{proof}
It follows from Lemma \ref{Lemma6.4} that
\begin{equation}\label{6.27}
\Pi\Big((2+\frac{2}{n})\Delta_{\theta_0}\overline{u}_{z_\nu}
-R_{\theta_0}\overline{u}_{z_\nu}
+\frac{n+2}{n}r_\infty u_\infty^{\frac{2}{n}}\overline{u}_{z_\nu}\Big)=0.
\end{equation}
Then we have
\begin{equation}\label{6.28}
\begin{split}
&\Pi\Big((2+\frac{2}{n})\Delta(u_\nu-\overline{u}_{z_\nu})-R_{\theta_0}(u_\nu-\overline{u}_{z_\nu})
+\frac{n+2}{n}r_\infty u_\infty^{\frac{2}{n}}(u_\nu-\overline{u}_{z_\nu})\Big)\\
&=\Pi\Big(-u_\nu^{1+\frac{2}{n}}(R_{\theta_\nu}-r_\infty)
+r_\infty\big(u_\nu^{1+\frac{2}{n}}+\frac{n+2}{n}\overline{u}_{z_\nu}^{\frac{2}{n}}(u_\nu-\overline{u}_{z_\nu})-\overline{u}_{z_\nu}^{1+\frac{2}{n}}\big)\\
&\hspace{8mm}-\frac{n+2}{n}r_\infty (\overline{u}_{z_\nu}^{\frac{2}{n}}-u_\infty^{\frac{2}{n}})(u_\nu-\overline{u}_{z_\nu})\Big)
\end{split}
\end{equation}
where the second equality follows from (\ref{6.27}). Applying Lemma \ref{Lemma6.3}(i) with $f=u_\nu-\overline{u}_{z_\nu}$, we get
\begin{equation}\label{6.29}
\begin{split}
&\|u_\nu-\overline{u}_{z_\nu}\|_{L^{1+\frac{2}{n}}(M)}\\
&\leq C\big\|u_\nu^{1+\frac{2}{n}}(R_{\theta_\nu}-r_\infty)\big\|_{L^{\frac{(n+1)(n+2)}{n^2+2n+2}}(M)}
+ C\big\|(\overline{u}_{z_\nu}^{\frac{2}{n}}-u_\infty^{\frac{2}{n}})(u_\nu-\overline{u}_{z_\nu})\big\|_{L^{\frac{(n+1)(n+2)}{n^2+2n+2}}(M)}\\
&\hspace{4mm}+ C\Big\|u_\nu^{1+\frac{2}{n}}+\frac{n+2}{n}\overline{u}_{z_\nu}^{\frac{2}{n}}(u_\nu-\overline{u}_{z_\nu})-\overline{u}_{z_\nu}^{1+\frac{2}{n}}
\Big\|_{L^{\frac{(n+1)(n+2)}{n^2+2n+2}}(M)}
\\
&\hspace{4mm}+C\sup_{a\in A}\left|\int_Mu_\infty^{\frac{2}{n}}\psi_a (u_\nu-\overline{u}_{z_\nu})dV_{\theta_0}\right|,
\end{split}
\end{equation}
where we have used (\ref{6.28}). Using the pointwise estimate
\begin{equation*}
\begin{split}
&\Big|u_\nu^{1+\frac{2}{n}}+\frac{n+2}{n}\overline{u}_{z_\nu}^{\frac{2}{n}}(u_\nu-\overline{u}_{z_\nu})-\overline{u}_{z_\nu}^{1+\frac{2}{n}}\Big|\\
&\leq C\overline{u}_{z_\nu}^{\max\{0,\frac{2}{n}-1\}}|u_\nu-\overline{u}_{z_\nu}|^{\min\{1+\frac{2}{n},2\}}+
C|u_\nu-\overline{u}_{z_\nu}|^{1+\frac{2}{n}},
\end{split}
\end{equation*}
we obtain
\begin{equation}\label{6.30}
\begin{split}
&\Big\|u_\nu^{1+\frac{2}{n}}+\frac{n+2}{n}\overline{u}_{z_\nu}^{\frac{2}{n}}(u_\nu-\overline{u}_{z_\nu})-\overline{u}_{z_\nu}^{1+\frac{2}{n}}
\Big\|_{L^{\frac{(n+1)(n+2)}{n^2+2n+2}}(M)}\\
&\leq C\Big\||u_\nu-\overline{u}_{z_\nu}|^{\min\{1+\frac{2}{n},2\}}+
|u_\nu-\overline{u}_{z_\nu}|^{1+\frac{2}{n}}\Big\|_{L^{\frac{(n+1)(n+2)}{n^2+2n+2}}(M)}.
\end{split}
\end{equation}
By H\"{o}lder's inequality, we have
\begin{equation}\label{6.31}
\begin{split}
&\Big\||u_\nu-\overline{u}_{z_\nu}|^{\min\{1+\frac{2}{n},2\}}+
|u_\nu-\overline{u}_{z_\nu}|^{1+\frac{2}{n}}\Big\|_{L^{\frac{(n+1)(n+2)}{n^2+2n+2}}(M)}\\
&\leq C\sum_{k=1}^m(N\varepsilon_{k,\nu})^{\frac{n^2}{n+2}}\,
\Big\||u_\nu-\overline{u}_{z_\nu}|^{\min\{1+\frac{2}{n},2\}}+
|u_\nu-\overline{u}_{z_\nu}|^{1+\frac{2}{n}}\Big\|_{L^{\frac{2n+2}{n+2}}(M)}\\
&\hspace{4mm}+C\Big\||u_\nu-\overline{u}_{z_\nu}|^{\min\{\frac{2}{n},1\}}+
|u_\nu-\overline{u}_{z_\nu}|^{\frac{2}{n}}\Big\|_{L^{n+1}(M\setminus\bigcup_{k=1}^mB_{N\varepsilon_{k,\nu}}(x_{k,\nu}))}
\|u_\nu-\overline{u}_{z_\nu}\|_{L^{1+\frac{2}{n}}(M)}.
\end{split}
\end{equation}
On the other hand, we have
\begin{equation}\label{6.32}
\begin{split}
&\|u_\nu-\overline{u}_{z_\nu}\|_{L^{2+\frac{2}{n}}(M\setminus\bigcup_{k=1}^mB_{N\varepsilon_{k,\nu}}(x_{k,\nu}))}\\
&\leq\sum_{k=1}^m\alpha_{k,\nu}\|\overline{u}_{(x_{k,\nu},\varepsilon_{k,\nu})}
\|_{L^{2+\frac{2}{n}}(M\setminus B_{N\varepsilon_{k,\nu}}(x_{k,\nu}))}
+\|w_\nu\|_{L^{2+\frac{2}{n}}(M)}\leq CN^{-n}+o(1),
\end{split}
\end{equation}
where the first equality follows from (\ref{6.21}), and the last inequality follows from (\ref{6.22}). Therefore, using (\ref{6.30})-(\ref{6.32}),
we get
\begin{equation}\label{6.33}
\begin{split}
&\Big\|u_\nu^{1+\frac{2}{n}}+\frac{n+2}{n}\overline{u}_{z_\nu}^{\frac{2}{n}}(u_\nu-\overline{u}_{z_\nu})-\overline{u}_{z_\nu}^{1+\frac{2}{n}}
\Big\|_{L^{\frac{(n+1)(n+2)}{n^2+2n+2}}(M)}\\
&\leq C\sum_{k=1}^m(N\varepsilon_{k,\nu})^{\frac{n^2}{n+2}}
+C\big(N^{-n\cdot\frac{2}{n}}+o(1)\big)\|u_\nu-\overline{u}_{z_\nu}\|_{L^{1+\frac{2}{n}}(M)}
\end{split}
\end{equation}
Moreover, we have
\begin{equation}\label{6.34}
\begin{split}
&\sup_{a\in A}\left|\int_Mu_\infty^{\frac{2}{n}}\psi_a (u_\nu-\overline{u}_{z_\nu})dV_{\theta_0}\right|\\
&=\sup_{a\in A}\left|\,\sum_{k=1}^m\alpha_{k,\nu}\int_Mu_\infty^{\frac{2}{n}}\psi_a \overline{u}_{(x_{k,\nu},\varepsilon_{k,\nu})}dV_{\theta_0}
+\int_Mu_\infty^{\frac{2}{n}}\psi_a w_\nu\,dV_{\theta_0}\right|\\
&\leq C\sum_{k=1}^m\varepsilon_{k,\nu}^n+o(1)\|w_\nu\|_{L^1(M)}\\
&= C\sum_{k=1}^m\varepsilon_{k,\nu}^n+o(1)\Big\|u_\nu-\overline{u}_{z_\nu}-\sum_{k=1}^m\alpha_{k,\nu}\overline{u}_{(x_{k,\nu},\varepsilon_{k,\nu})}\Big\|_{L^1(M)}\\
&\leq C\sum_{k=1}^m\varepsilon_{k,\nu}^n+o(1)\|u_\nu-\overline{u}_{z_\nu}\|_{L^1(M)}
\end{split}
\end{equation}
where the first and second equality follows from (\ref{6.21}). Combining (\ref{6.29}), (\ref{6.33}), and (\ref{6.34})
\begin{equation*}
\begin{split}
&\|u_\nu-\overline{u}_{z_\nu}\|_{L^{1+\frac{2}{n}}(M)}\\
&\leq C\big\|u_\nu^{1+\frac{2}{n}}(R_{\theta_\nu}-r_\infty)\big\|_{L^{\frac{(n+1)(n+2)}{n^2+2n+2}}(M)}
+ C\big\|(\overline{u}_{z_\nu}^{\frac{2}{n}}-u_\infty^{\frac{2}{n}})(u_\nu-\overline{u}_{z_\nu})\big\|_{L^{\frac{(n+1)(n+2)}{n^2+2n+2}}(M)}\\
&\hspace{4mm}
+C\sum_{k=1}^m(N\varepsilon_{k,\nu})^{\frac{n^2}{n+2}}
+C\big(N^{-2}+o(1)\big)\|u_\nu-\overline{u}_{z_\nu}\|_{L^{1+\frac{2}{n}}(M)}\\
&\hspace{4mm}+C\sum_{k=1}^m\varepsilon_{k,\nu}^n+o(1)\|u_\nu-\overline{u}_{z_\nu}\|_{L^1(M)}\\
&\leq C\big\|u_\nu^{1+\frac{2}{n}}(R_{\theta_\nu}-r_\infty)\big\|_{L^{\frac{2n+2}{n+2}}(M)}
+C\sum_{k=1}^m(N\varepsilon_{k,\nu})^{\frac{n^2}{n+2}}+C\sum_{k=1}^m\varepsilon_{k,\nu}^n\\
&\hspace{4mm}
+C\big(N^{-2}+o(1)\big)\|u_\nu-\overline{u}_{z_\nu}\|_{L^{1+\frac{2}{n}}(M)},
\end{split}
\end{equation*}
where the last inequality follows from
H\"{o}lder's inequality and the fact that $|z_\nu|=o(1)$ by Proposition \ref{Prop6.7}. Hence, if we choose $N$ sufficiently large,
we obtain
\begin{equation*}
\|u_\nu-\overline{u}_{z_\nu}\|_{L^{1+\frac{2}{n}}(M)}\leq C\big\|u_\nu^{1+\frac{2}{n}}(R_{\theta_\nu}-r_\infty)\big\|_{L^{\frac{2n+2}{n+2}}(M)}
+C\sum_{k=1}^m\varepsilon_{k,\nu}^{\frac{n^2}{n+2}},
\end{equation*}
which implies (\ref{6.26.5}). This proves the assertion.
\end{proof}

%%%%%%%%%%%%%%%%%%%%%
%%%%%%%%%%%%%%%%%%%%%%
 Then, we have the following lemma. The proof takes the similar procedure as Lemma 6.12 in \cite{Brendle4}, we omit its proof here.

\begin{lem}\label{Lemma6.12}
The difference $u_\nu-\overline{u}_{z_\nu}$ satisfies
the estimate
\begin{equation*}
\|u_\nu-\overline{u}_{z_\nu}\|_{L^{1}(M)}\leq
C\big\|u_\nu^{1+\frac{2}{n}}(R_{\theta_\nu}-r_\infty)\big\|_{L^{\frac{2n+2}{n+2}}(M)}^{1+\frac{n}{2}}
+C\sum_{k=1}^m\varepsilon_{k,\nu}^{n}
\end{equation*}
if $\nu$ is sufficiently large.
\end{lem}
Now, we can prove the following lemma.

\begin{lem}\label{Lemma6.13}
We have
\begin{equation*}
\begin{split}
&\sup_{a\in A}\left|\int_M\psi_a\Big((2+\frac{2}{n})\Delta_{\theta_0}\overline{u}_{z_\nu}
-R_{\theta_0}\overline{u}_{z_\nu}
+r_\infty\overline{u}_{z_\nu}^{1+\frac{2}{n}}\Big)dV_{\theta_0}\right|\\
&\leq C\left(\int_Mu_\nu^{2+\frac{2}{n}}|R_{\theta_\nu}-r_\infty|^{\frac{2n+2}{n+2}}dV_{\theta_0}\right)^{\frac{n+2}{2n+2}}
+C\sum_{k=1}^m\varepsilon_{k,\nu}^n
\end{split}
\end{equation*}
if $\nu$ is sufficiently large.
\end{lem}
\begin{proof}
Note that
\begin{equation}\label{6.37}
\begin{split}
&\int_M\psi_a\Big((2+\frac{2}{n})\Delta_{\theta_0}\overline{u}_{z_\nu}
-R_{\theta_0}\overline{u}_{z_\nu}
+r_\infty\overline{u}_{z_\nu}^{1+\frac{2}{n}}\Big)dV_{\theta_0}\\
&=\int_M\psi_a\Big((2+\frac{2}{n})\Delta_{\theta_0}u_\nu
-R_{\theta_0}u_\nu
+r_\infty u_\nu^{1+\frac{2}{n}}\Big)dV_{\theta_0}\\
&\hspace{4mm}-\lambda_a\int_Mu_\infty^{\frac{2}{n}}\psi_a(\overline{u}_{z_\nu}-u_\nu)dV_{\theta_0}
+r_\infty\int_M\psi_a(\overline{u}_{z_\nu}^{1+\frac{2}{n}}-u_\nu^{1+\frac{2}{n}})dV_{\theta_0}
\end{split}
\end{equation}
where we have used integration by parts and (\ref{6.1}). Using the identity
$$(2+\frac{2}{n})\Delta_{\theta_0}u_\nu
-R_{\theta_0}u_\nu
+r_\infty u_\nu^{1+\frac{2}{n}}=-u_\nu^{1+\frac{2}{n}}(R_{\theta_\nu}-r_\infty),$$
we can rewrite (\ref{6.37}) as
\begin{equation*}
\begin{split}
&\int_M\psi_a\Big((2+\frac{2}{n})\Delta_{\theta_0}\overline{u}_{z_\nu}
-R_{\theta_0}\overline{u}_{z_\nu}
+r_\infty\overline{u}_{z_\nu}^{1+\frac{2}{n}}\Big)dV_{\theta_0}\\
&=-\int_M\psi_au_\nu^{1+\frac{2}{n}}(R_{\theta_\nu}-r_\infty)dV_{\theta_0}-\lambda_a\int_Mu_\infty^{\frac{2}{n}}\psi_a(\overline{u}_{z_\nu}-u_\nu)dV_{\theta_0}\\
&\hspace{4mm}
+r_\infty\int_M\psi_a(\overline{u}_{z_\nu}^{1+\frac{2}{n}}-u_\nu^{1+\frac{2}{n}})dV_{\theta_0}.
\end{split}
\end{equation*}
From this, together with the pointwise estimate
$$|\overline{u}_{z_\nu}^{1+\frac{2}{n}}-u_\nu^{1+\frac{2}{n}}|\leq C\overline{u}_{z_\nu}^{\frac{2}{n}}|\overline{u}_{z_\nu}-u_\nu|+C|\overline{u}_{z_\nu}-u_\nu|^{1+\frac{2}{n}},$$
we conclude
\begin{equation*}
\begin{split}
&\sup_{a\in A}\left|\int_M\psi_a\Big((2+\frac{2}{n})\Delta_{\theta_0}\overline{u}_{z_\nu}
-R_{\theta_0}\overline{u}_{z_\nu}
+r_\infty\overline{u}_{z_\nu}^{1+\frac{2}{n}}\Big)dV_{\theta_0}\right|\\
&\leq C\|u_\nu^{1+\frac{2}{n}}(R_{\theta_\nu}-r_\infty)\|_{L^{\frac{2n+2}{n+2}}(M)}
+C\|\overline{u}_{z_\nu}-u_\nu\|_{L^1(M)}
+C\|\overline{u}_{z_\nu}-u_\nu\|^{1+\frac{2}{n}}_{L^{1+\frac{2}{n}}(M)}.
\end{split}
\end{equation*}
Hence, it follows from Lemma \ref{Lemma6.11} and \ref{Lemma6.12} that
\begin{equation}\label{6.37.5}
\begin{split}
&\sup_{a\in A}\left|\int_M\psi_a\Big((2+\frac{2}{n})\Delta_{\theta_0}\overline{u}_{z_\nu}
-R_{\theta_0}\overline{u}_{z_\nu}
+r_\infty\overline{u}_{z_\nu}^{1+\frac{2}{n}}\Big)dV_{\theta_0}\right|\\
&\leq C\|u_\nu^{1+\frac{2}{n}}(R_{\theta_\nu}-r_\infty)\|_{L^{\frac{2n+2}{n+2}}(M)}
+C\|u_\nu^{1+\frac{2}{n}}(R_{\theta_\nu}-r_\infty)\|_{L^{\frac{2n+2}{n+2}}(M)}^{1+\frac{2}{n}}
+C\sum_{k=1}^m\varepsilon_{k,\nu}^n.
\end{split}
\end{equation}
Now Lemma \ref{Lemma6.13} follows from (\ref{6.37.5}) because
\begin{equation*}
\begin{split}
\|u_\nu^{1+\frac{2}{n}}(R_{\theta_\nu}-r_\infty)\|_{L^{\frac{2n+2}{n+2}}(M)}
&=\int_M|R_{\theta_\nu}-r_\infty|^{\frac{2n+2}{n+2}}dV_{\theta_\nu}\to 0
\end{split}
\end{equation*}
as $\nu\to\infty$ by (\ref{4}).
\end{proof}

\begin{prop}\label{Prop6.14}
The energy of $\overline{u}_{z_\nu}$ satisfies the estimate
\begin{equation*}
E(\overline{u}_{z_\nu})-E(u_\infty)\leq
 C\left(\int_Mu_\nu^{2+\frac{2}{n}}|R_{\theta_\nu}-r_\infty|^{\frac{2n+2}{n+2}}dV_{\theta_0}\right)^{\frac{n+2}{2n+2}(1+\gamma)}
+C\sum_{k=1}^m\varepsilon_{k,\nu}^{n(1+\gamma)}
\end{equation*}
if $\nu$ is sufficiently large.
\end{prop}
\begin{proof}
It follows immediately from Lemma \ref{Lemma6.5} and \ref{Lemma6.13}.
\end{proof}

\begin{prop}\label{Prop6.15}
The energy of $v_\nu$ satisfies the estimate
\begin{equation}\label{6.38}
E(v_\nu)\leq \Big(E(\overline{u}_{z_\nu})^{n+1}+\sum_{k=1}^mE(\overline{u}_{(x_{k,\nu},\varepsilon_{k,\nu})})^{n+1}\Big)^{\frac{1}{n+1}}
-C\sum_{k=1}^m\varepsilon_{k,\nu}^{n}
\end{equation}
if $\nu$ is sufficiently large.
\end{prop}
\begin{proof}
By H\"{o}lder's inequality, we have
\begin{equation}\label{6.52}
\begin{split}
&\left(E(\overline{u}_{z_\nu})^{n+1}+\sum_{k=1}^mE(\overline{u}_{(x_{k,\nu},\varepsilon_{k,\nu})})^{n+1}\right)^{\frac{1}{n+1}}
\left(\int_Mv_\nu^{2+\frac{2}{n}}dV_{\theta_0}\right)^{\frac{n}{n+1}}\\
&\geq \int_M\left(F(\overline{u}_{z_\nu})^{n+1}\overline{u}_{z_\nu}^{2+\frac{2}{n}}+\sum_{k=1}^mF(\overline{u}_{(x_{k,\nu},\varepsilon_{k,\nu})})^{n+1}
\,\overline{u}_{(x_{k,\nu},\varepsilon_{k,\nu})}^{2+\frac{2}{n}}\right)^{\frac{1}{n+1}}v_{\nu}^2\,dV_{\theta_0}\\
&\geq \int_MF(\overline{u}_{z_\nu})
\overline{u}_{z_\nu}^{2+\frac{2}{n}}dV_{\theta_0}+\int_M\sum_{k=1}^m\alpha_{k,\nu}^2F(\overline{u}_{(x_{k,\nu},\varepsilon_{k,\nu})})
\overline{u}_{(x_{k,\nu},\varepsilon_{k,\nu})}^{2+\frac{2}{n}}dV_{\theta_0}\\
&\hspace{4mm}+
2\int_M\sum_{k=1}^m\alpha_{k,\nu}
 \Big(F(\overline{u}_{z_\nu})^{n+1}
\,\overline{u}_{z_\nu}^{2+\frac{2}{n}}
+F(\overline{u}_{(x_{k,\nu},\varepsilon_{k,\nu})})^{n+1}
\,\overline{u}_{(x_{k,\nu},\varepsilon_{k,\nu})}^{2+\frac{2}{n}}\Big)^{\frac{1}{n+1}}
\overline{u}_{z_\nu}\overline{u}_{(x_{k,\nu},\varepsilon_{k,\nu})}dV_{\theta_0}\\
&\hspace{4mm}+
2\int_M\sum_{1\leq i<j\leq m}\alpha_{i,\nu}\alpha_{j,\nu}
 \Big(F(\overline{u}_{(x_{i,\nu},\varepsilon_{i,\nu})})^{n+1}
\,\overline{u}_{(x_{i,\nu},\varepsilon_{i,\nu})}^{2+\frac{2}{n}}\\
&\hspace{12mm}
+F(\overline{u}_{(x_{j,\nu},\varepsilon_{j,\nu})})^{n+1}
\,\overline{u}_{(x_{j,\nu},\varepsilon_{j,\nu})}^{2+\frac{2}{n}}\Big)^{\frac{1}{n+1}}
\overline{u}_{(x_{i,\nu},\varepsilon_{i,\nu})}\overline{u}_{(x_{j,\nu},\varepsilon_{j,\nu})}dV_{\theta_0}.
\end{split}
\end{equation}
Using the inequality
\begin{equation*}
\begin{split}
&\Big(F(\overline{u}_{z_\nu})^{n+1}
\,\overline{u}_{z_\nu}^{2+\frac{2}{n}}
+F(\overline{u}_{(x_{k,\nu},\varepsilon_{k,\nu})})^{n+1}
\,\overline{u}_{(x_{k,\nu},\varepsilon_{k,\nu})}^{2+\frac{2}{n}}\Big)^{\frac{1}{n+1}}
\overline{u}_{z_\nu}\overline{u}_{(x_{k,\nu},\varepsilon_{k,\nu})}\\
&\geq F(\overline{u}_{z_\nu})\overline{u}_{z_\nu}^{1+\frac{2}{n}}\overline{u}_{(x_{k,\nu},\varepsilon_{k,\nu})}
+C\,\varepsilon_{k,\nu}^{-n-2}1_{\{d(x_{k,\nu},x)\leq \varepsilon_{k,\nu}\}},
\end{split}
\end{equation*}
we obtain
\begin{equation}\label{6.53}
\begin{split}
&\int_M\Big(F(\overline{u}_{z_\nu})^{n+1}
\,\overline{u}_{z_\nu}^{2+\frac{2}{n}}
+F(\overline{u}_{(x_{k,\nu},\varepsilon_{k,\nu})})^{n+1}
\,\overline{u}_{(x_{k,\nu},\varepsilon_{k,\nu})}^{2+\frac{2}{n}}\Big)^{\frac{1}{n+1}}
\overline{u}_{z_\nu}\overline{u}_{(x_{k,\nu},\varepsilon_{k,\nu})}dV_{\theta_0}\\
&\geq \int_MF(\overline{u}_{z_\nu})\overline{u}_{z_\nu}^{1+\frac{2}{n}}\overline{u}_{(x_{k,\nu},\varepsilon_{k,\nu})}dV_{\theta_0}
+C\,\varepsilon_{k,\nu}^{n}
\end{split}
\end{equation}
if $\nu$ is sufficiently large. We next consider $i<j$.
Again we get (\ref{save}).
Substituting (\ref{save}) and (\ref{6.53}) into (\ref{6.52}), we get
 \begin{equation}\label{6.55}
\begin{split}
&\left(E(\overline{u}_{z_\nu})^{n+1}+\sum_{k=1}^mE(\overline{u}_{(x_{k,\nu},\varepsilon_{k,\nu})})^{n+1}\right)^{\frac{1}{n+1}}
\left(\int_Mv_\nu^{2+\frac{2}{n}}dV_{\theta_0}\right)^{\frac{n}{n+1}}\\
&\geq \int_MF(\overline{u}_{z_\nu})
\overline{u}_{z_\nu}^{2+\frac{2}{n}}dV_{\theta_0}+\int_M\sum_{k=1}^m\alpha_{k,\nu}^2F(\overline{u}_{(x_{k,\nu},\varepsilon_{k,\nu})})
\overline{u}_{(x_{k,\nu},\varepsilon_{k,\nu})}^{2+\frac{2}{n}}dV_{\theta_0}\\
&\hspace{4mm}+
2\int_M\sum_{k=1}^m\alpha_{k,\nu}
F(\overline{u}_{z_\nu})
\overline{u}_{z_\nu}^{1+\frac{2}{n}}\overline{u}_{(x_{k,\nu},\varepsilon_{k,\nu})}dV_{\theta_0}\\
&\hspace{4mm}+
2\int_M\sum_{1\leq i<j\leq m}\alpha_{i,\nu}\alpha_{j,\nu}
F(\overline{u}_{(x_{j,\nu},\varepsilon_{j,\nu})})
\overline{u}_{(x_{i,\nu},\varepsilon_{i,\nu})}\overline{u}_{(x_{j,\nu},\varepsilon_{j,\nu})}^{1+\frac{2}{n}}dV_{\theta_0}\\
&\hspace{4mm}
+C\sum_{k=1}^m\alpha_{k,\nu}\varepsilon_{k,\nu}^{n}
+C \sum_{1\leq i<j\leq m}\left(\frac{\varepsilon_{i,\nu}^2\varepsilon_{j,\nu}^2}{\varepsilon_{j,\nu}^4+d(x_{i,\nu},x_{j,\nu})^4}\right)^{\frac{n}{2}}.
\end{split}
\end{equation}
By the definition of $v_\nu$ in (\ref{6.20}), we have
\begin{equation}\label{6.51}
\begin{split}
&E(v_\nu)\left(\int_Mv_\nu^{2+\frac{2}{n}}dV_{\theta_0}\right)^{\frac{n}{n+1}}\\
&=\int_MF(\overline{u}_{z_\nu})
\overline{u}_{z_\nu}^{2+\frac{2}{n}}dV_{\theta_0}+
\int_M\sum_{k=1}^m\alpha_{k,\nu}^2F(\overline{u}_{(x_{k,\nu},\varepsilon_{k,\nu})})
\overline{u}_{(x_{k,\nu},\varepsilon_{k,\nu})}^{2+\frac{2}{n}}dV_{\theta_0}\\
&\hspace{4mm}-
2\int_M\sum_{k=1}^m\alpha_{k,\nu}
 \overline{u}_{(x_{k,\nu},\varepsilon_{k,\nu})}\Big((2+\frac{2}{n})\Delta_{\theta_0}
\overline{u}_{z_\nu}
-R_{\theta_0}\overline{u}_{z_\nu}\Big)dV_{\theta_0}\\
&\hspace{4mm}-
2\int_M\sum_{1\leq i<j\leq m}\alpha_{i,\nu}\alpha_{j,\nu}
 \overline{u}_{(x_{i,\nu},\varepsilon_{i,\nu})}\Big((2+\frac{2}{n})\Delta_{\theta_0}
 \overline{u}_{(x_{j,\nu},\varepsilon_{j,\nu})}
-R_{\theta_0}\overline{u}_{(x_{j,\nu},\varepsilon_{j,\nu})}\Big)dV_{\theta_0}.
\end{split}
\end{equation}
Substituting (\ref{6.55}) into (\ref{6.51}), we obtain
\begin{equation}\label{6.56}
\begin{split}
&E(v_\nu)\left(\int_Mv_\nu^{2+\frac{2}{n}}dV_{\theta_0}\right)^{\frac{n}{n+1}}\\
&\leq \left(E(\overline{u}_{z_\nu})^{n+1}+\sum_{k=1}^mE(\overline{u}_{(x_{k,\nu},\varepsilon_{k,\nu})})^{n+1}\right)^{\frac{1}{n+1}}
\left(\int_Mv_\nu^{2+\frac{2}{n}}dV_{\theta_0}\right)^{\frac{n}{n+1}}\\
&\hspace{4mm}-
2\int_M\sum_{k=1}^m\alpha_{k,\nu}
 \overline{u}_{(x_{k,\nu},\varepsilon_{k,\nu})}\Big((2+\frac{2}{n})\Delta_{\theta_0}
\overline{u}_{z_\nu}
-R_{\theta_0}\overline{u}_{z_\nu}+F(\overline{u}_{z_\nu})\overline{u}_{z_\nu}^{1+\frac{2}{n}}\Big)dV_{\theta_0}\\
&\hspace{4mm}-
2\int_M\sum_{1\leq i<j\leq m}\alpha_{i,\nu}\alpha_{j,\nu}
 \overline{u}_{(x_{i,\nu},\varepsilon_{i,\nu})}\\
&\hspace{1.5cm}
\cdot\left((2+\frac{2}{n})\Delta_{\theta_0}
 \overline{u}_{(x_{j,\nu},\varepsilon_{j,\nu})}
-R_{\theta_0}\overline{u}_{(x_{j,\nu},\varepsilon_{j,\nu})}
+F(\overline{u}_{(x_{j,\nu},\varepsilon_{j,\nu})})\overline{u}_{(x_{j,\nu},\varepsilon_{j,\nu})}^{1+\frac{2}{n}}\right)dV_{\theta_0}\\
&\hspace{4mm}
-C\sum_{k=1}^m\alpha_{k,\nu}\varepsilon_{k,\nu}^{n}
-C \sum_{1\leq i<j\leq m}\left(\frac{\varepsilon_{i,\nu}^2\varepsilon_{j,\nu}^2}{\varepsilon_{j,\nu}^4+d(x_{i,\nu},x_{j,\nu})^4}\right)^{\frac{n}{2}}.
\end{split}
\end{equation}
Note that
\begin{equation}\label{6.57}
\int_M
 \overline{u}_{(x_{k,\nu},\varepsilon_{k,\nu})}\Big|(2+\frac{2}{n})\Delta_{\theta_0}
\overline{u}_{z_\nu}
-R_{\theta_0}\overline{u}_{z_\nu}+F(\overline{u}_{z_\nu})\overline{u}_{z_\nu}^{1+\frac{2}{n}}\Big|\,dV_{\theta_0}
\leq o(1)\varepsilon_{k,\nu}^{n}.
\end{equation}
Moreover, since $F(\overline{u}_{(x_{j,\nu},\varepsilon_{j,\nu})})=r_\infty+o(1)$ by Lemma \ref{LemmaB.6}, it follows from Lemma \ref{LemmaB.4}
and \ref{LemmaB.5} that
\begin{equation}\label{6.58}
\begin{split}
&\int_M
 \overline{u}_{(x_{i,\nu},\varepsilon_{i,\nu})}\Big|(2+\frac{2}{n})\Delta_{\theta_0}
 \overline{u}_{(x_{j,\nu},\varepsilon_{j,\nu})}
-R_{\theta_0}\overline{u}_{(x_{j,\nu},\varepsilon_{j,\nu})}
+F(\overline{u}_{(x_{j,\nu},\varepsilon_{j,\nu})})
\overline{u}_{(x_{j,\nu},\varepsilon_{j,\nu})}^{1+\frac{2}{n}}\Big|\,dV_{\theta_0}\\
&\leq \int_M
 \overline{u}_{(x_{i,\nu},\varepsilon_{i,\nu})}\Big|(2+\frac{2}{n})\Delta_{\theta_0}
 \overline{u}_{(x_{j,\nu},\varepsilon_{j,\nu})}
-R_{\theta_0}\overline{u}_{(x_{j,\nu},\varepsilon_{j,\nu})}
+r_\infty
\overline{u}_{(x_{j,\nu},\varepsilon_{j,\nu})}^{1+\frac{2}{n}}\Big|\,dV_{\theta_0}\\
&\hspace{4mm}+|F(\overline{u}_{(x_{j,\nu},\varepsilon_{j,\nu})})-r_\infty|
\int_M\overline{u}_{(x_{i,\nu},\varepsilon_{i,\nu})}\overline{u}_{(x_{j,\nu},\varepsilon_{j,\nu})}^{1+\frac{2}{n}}dV_{\theta_0}\\
&\leq C(\delta^4+\delta^{2n}+\frac{\varepsilon_{j,\nu}^2}{\delta^2})\left(\frac{\varepsilon_{i,\nu}^2\varepsilon_{j,\nu}^2}{\varepsilon_{j,\nu}^4+d(x_{i,\nu},x_{j,\nu})^4}\right)^{\frac{n}{2}}
+
o(1)\left(\frac{\varepsilon_{i,\nu}^2\varepsilon_{j,\nu}^2}{\varepsilon_{j,\nu}^4+d(x_{i,\nu},x_{j,\nu})^4}\right)^{\frac{n}{2}}
\end{split}
\end{equation}
for $i<j$. Therefore, the assertion follows from (\ref{6.56}), (\ref{6.57}) and (\ref{6.58})
by choosing $\delta$ sufficiently small.
\end{proof}

\begin{cor}\label{Cor6.16}
We have
\begin{equation*}
E(v_\nu)\leq \Big(E(u_\infty)^{n+1}+mY(S^{2n+1})^{n+1}\Big)^{\frac{1}{n+1}}
+C\left(\int_Mu_\nu^{2+\frac{2}{n}}|R_{\theta_\nu}-r_\infty|^{\frac{2n+2}{n+2}}dV_{\theta_0}\right)^{\frac{n+2}{2n+2}(1+\gamma)}
\end{equation*}
if $\nu$ is sufficiently large.
\end{cor}
\begin{proof}
By Proposition \ref{Prop6.14}, we have
\begin{equation*}
E(\overline{u}_{z_\nu})-E(u_\infty)\leq
 C\left(\int_Mu_\nu^{2+\frac{2}{n}}|R_{\theta_\nu}-r_\infty|^{\frac{2n+2}{n+2}}dV_{\theta_0}\right)^{\frac{n+2}{2n+2}(1+\gamma)}
+C\sum_{k=1}^m\varepsilon_{k,\nu}^{n(1+\gamma)}
\end{equation*}
if $\nu$ is sufficiently large.
By Proposition \ref{PropB.3}, we have
$$E(\overline{u}_{(x_{k,\nu},\varepsilon_{k,\nu})})\leq Y(S^{2n+1}).$$
Substituting these into (\ref{6.38}), we get the result.
\end{proof}

\section{Proof of Proposition \ref{Prop3.3}}\label{section7}

In this section, we give the proof of Proposition \ref{Prop3.3}.
Note that
\begin{equation*}
\begin{split}
r_{\theta_\nu}
&=\int_M\Big((2+\frac{2}{n})|\nabla_{\theta_0}v_\nu|^2_{\theta_0}+R_{\theta_0}v_\nu^2\Big)dV_{\theta_0}+2\int_MR_{\theta_\nu}u_\nu^{1+\frac{2}{n}}w_\nu\,dV_{\theta_0}\\
&\hspace{4mm}-\int_M\Big((2+\frac{2}{n})|\nabla_{\theta_0}w_\nu|^2_{\theta_0}+R_{\theta_0}w_\nu^2\Big)dV_{\theta_0},
\end{split}
\end{equation*}
where the first equality follows from (\ref{6}), (\ref{9}), (\ref{3}), and integration by parts, and the second equality
follows from (\ref{6.20.5}).
This implies that
\begin{equation}\label{7.1}
\begin{split}
r_{\theta_\nu}
&=E(v_\nu)\left(\int_M v_\nu^{2+\frac{2}{n}}dV_{\theta_0}
\right)^{\frac{n}{n+1}}+2\int_Mu_\nu^{1+\frac{2}{n}}(R_{\theta_\nu}-r_\infty)w_\nu\,dV_{\theta_0}\\
&\hspace{4mm}-\int_M\Big((2+\frac{2}{n})|\nabla_{\theta_0}w_\nu|^2_{\theta_0}+R_{\theta_0}w_\nu^2-\frac{n+2}{n}r_\infty v_\nu^{\frac{2}{n}}w_\nu^2\Big)dV_{\theta_0}\\
&\hspace{4mm}+r_\infty\int_M\Big(-\frac{n+2}{n}r_\infty v_\nu^{\frac{2}{n}}w_\nu^2+2(v_\nu+w_\nu)^{1+\frac{2}{n}}w_\nu\Big)dV_{\theta_0}.
\end{split}
\end{equation}
Note that $\displaystyle x^{\frac{n}{n+1}}-1\leq \frac{n}{n+1}x-\frac{n}{n+1}$ for $x\geq 0$. Therefore,
\begin{equation*}
\begin{split}
\left(\int_M v_\nu^{2+\frac{2}{n}}dV_{\theta_0}
\right)^{\frac{n}{n+1}}-1
&\leq
\int_M\Big(\frac{n}{n+1}v_\nu^{2+\frac{2}{n}}-\frac{n}{n+1}(v_\nu+w_\nu)^{2+\frac{2}{n}}\Big)dV_{\theta_0}
\end{split}
\end{equation*}
where we have used (\ref{3}). Multiplying this by $r_\infty$ and adding it to (\ref{7.1}), we obtain
\begin{equation}\label{7.2}
\begin{split}
r_{\theta_\nu}
&\leq r_\infty+(E(v_\nu)-r_\infty)\left(\int_M v_\nu^{2+\frac{2}{n}}dV_{\theta_0}
\right)^{\frac{n}{n+1}}+2\int_Mu_\nu^{1+\frac{2}{n}}(R_{\theta_\nu}-r_\infty)w_\nu\,dV_{\theta_0}\\
&\hspace{4mm}-\int_M\Big((2+\frac{2}{n})|\nabla_{\theta_0}w_\nu|^2_{\theta_0}+R_{\theta_0}w_\nu^2-\frac{n+2}{n}r_\infty v_\nu^{\frac{2}{n}}w_\nu^2\Big)dV_{\theta_0}\\
&\hspace{4mm}+r_\infty\int_M\Big(\frac{n}{n+1}v_\nu^{2+\frac{2}{n}}-\frac{n+2}{n}r_\infty v_\nu^{\frac{2}{n}}w_\nu^2\\
&\hspace{2cm}+2(v_\nu+w_\nu)^{1+\frac{2}{n}}w_\nu
-\frac{n}{n+1}(v_\nu+w_\nu)^{2+\frac{2}{n}}\Big)dV_{\theta_0}.
\end{split}
\end{equation}
Using H\"{o}lder's inequality, we obtain
\begin{equation}\label{7.3}
\begin{split}
&\int_Mu_\nu^{1+\frac{2}{n}}(R_{\theta_\nu}-r_\infty)w_\nu\,dV_{\theta_0}\\
&\leq\left(\int_Mu_\nu^{2+\frac{2}{n}}|R_{\theta_\nu}-r_\infty|^{\frac{2n+2}{n+2}}dV_{\theta_0}\right)^{\frac{n+2}{2n+2}}
\left(\int_M|w_\nu|^{2+\frac{2}{n}}dV_{\theta_0}\right)^{\frac{n}{2n+2}}.
\end{split}
\end{equation}
Moreover, it follows from from Corollary \ref{Cor5.5} and \ref{Cor6.10} that
\begin{equation}\label{7.4}
\begin{split}
&\int_M\Big((2+\frac{2}{n})|\nabla_{\theta_0}w_\nu|^2_{\theta_0}+R_{\theta_0}w_\nu^2-\frac{n+2}{n}r_\infty v_\nu^{\frac{2}{n}}w_\nu^2\Big)dV_{\theta_0}\\
&\geq c\int_M\left((2+\frac{2}{n})|\nabla_{\theta_0}w_\nu|^2_{\theta_0}+R_{\theta_0}w_\nu^2\right)dV_{\theta_0}\\
&\geq cY(M,\theta_0)\left(\int_M|w_\nu|^{2+\frac{2}{n}}dV_{\theta_0}\right)^{\frac{n}{n+1}}.
\end{split}
\end{equation}
Finally it follows from the pointwise estimate
\begin{equation*}
\begin{split}
&\left|\frac{n}{n+1}v_\nu^{2+\frac{2}{n}}-\frac{n+2}{n}r_\infty v_\nu^{\frac{2}{n}}w_\nu^2+2(v_\nu+w_\nu)^{1+\frac{2}{n}}w_\nu
-\frac{n}{n+1}(v_\nu+w_\nu)^{2+\frac{2}{n}}\right|\\
&\leq C v_\nu^{\max\{0,2+\frac{2}{n}-3\}}|w_\nu|^{\min\{3,2+\frac{2}{n}\}}+C|w_\nu|^{2+\frac{2}{n}}
\end{split}
\end{equation*}
that
\begin{equation}\label{7.5}
\begin{split}
&\int_M\Big(\frac{n}{n+1}v_\nu^{2+\frac{2}{n}}-\frac{n+2}{n}r_\infty v_\nu^{\frac{2}{n}}w_\nu^2+2(v_\nu+w_\nu)^{1+\frac{2}{n}}w_\nu
-\frac{n}{n+1}(v_\nu+w_\nu)^{2+\frac{2}{n}}\Big)dV_{\theta_0}\\
&\leq C\int_Mv_\nu^{\max\{0,2+\frac{2}{n}-3\}}|w_\nu|^{\min\{3,2+\frac{2}{n}\}}dV_{\theta_0}+C\int_M|w_\nu|^{2+\frac{2}{n}}dV_{\theta_0}\\
&\leq C\left(\int_M|w_\nu|^{2+\frac{2}{n}}dV_{\theta_0}\right)^{\frac{n}{n+1}\min\{\frac{3}{2},\frac{n+1}{n}\}}.
\end{split}
\end{equation}
Substituting (\ref{7.3}), (\ref{7.4}), and  (\ref{7.5}) into (\ref{7.2}), we obtain
\begin{equation}\label{7.6}
\begin{split}
r_{\theta_\nu}
&\leq r_\infty+(E(v_\nu)-r_\infty)\left(\int_M v_\nu^{2+\frac{2}{n}}dV_{\theta_0}
\right)^{\frac{n}{n+1}}\\
&\hspace{4mm}+2\left(\int_Mu_\nu^{2+\frac{2}{n}}|R_{\theta_\nu}-r_\infty|^{\frac{2n+2}{n+2}}dV_{\theta_0}\right)^{\frac{n+2}{2n+2}}
\left(\int_M|w_\nu|^{2+\frac{2}{n}}dV_{\theta_0}\right)^{\frac{n}{2n+2}}\\
&\hspace{4mm}-cY(M,\theta_0)\left(\int_M|w_\nu|^{2+\frac{2}{n}}dV_{\theta_0}\right)^{\frac{n}{n+1}}
+C\left(\int_M|w_\nu|^{2+\frac{2}{n}}dV_{\theta_0}\right)^{\frac{n}{n+1}\min\{\frac{3}{2},\frac{n+1}{n}\}}\\
&\leq r_\infty+(E(v_\nu)-r_\infty)\left(\int_M v_\nu^{2+\frac{2}{n}}dV_{\theta_0}
\right)^{\frac{n}{n+1}}
+\left(\int_Mu_\nu^{2+\frac{2}{n}}|R_{\theta_\nu}-r_\infty|^{\frac{2n+2}{n+2}}dV_{\theta_0}\right)^{\frac{n+2}{n+1}}
\end{split}
\end{equation}
where we have used Young's inequality. By (\ref{4.8}), Corollary \ref{Cor5.7}, and Corollary \ref{Cor6.16}, we
get
\begin{equation*}
E(v_\nu)\leq r_\infty
+C\left(\int_Mu_\nu^{2+\frac{2}{n}}|R_{\theta_\nu}-r_\infty|^{\frac{2n+2}{n+2}}dV_{\theta_0}\right)^{\frac{n+2}{2n+2}(1+\gamma)}.
\end{equation*}
Substituting this into (\ref{7.6}), we obtain
\begin{equation*}
r_{\theta_\nu}\leq r_\infty
+C\left(\int_Mu_\nu^{2+\frac{2}{n}}|R_{\theta_\nu}-r_\infty|^{\frac{2n+2}{n+2}}dV_{\theta_0}\right)^{\frac{n+2}{2n+2}(1+\gamma)}
\end{equation*}
since $\displaystyle\int_Mu_\nu^{2+\frac{2}{n}}|R_{\theta_\nu}-r_\infty|^{\frac{2n+2}{n+2}}dV_{\theta_0}\to 0$ as $\nu\to\infty$ by (\ref{4}).
This completes the proof of Proposition \ref{Prop3.3}.

\appendix

\section{}\label{Appendix}

First we consider the case when $n=1$.
We review the definition of CR normal coordinates
and recall some of its properties.
Give any $x\in M$, we can find a contact form $\widehat{\theta}_x=\varphi_x^{\frac{2}{n}}\theta_0$ conformal to $\theta_0$,
where $\widehat{\theta}_x$ is a contact form defined
in $(z,t)$ which is the CR normal coordinates centered at $x$.
On the other hand, $\widehat{\theta}_x$
satisfies the following properties: (see Theorem 3.7 in P.172 of \cite{Dragomir} and Proposition 6.5 in \cite{Cheng&Malchiodi&Yang})
\begin{equation}\label{B.0}
\begin{split}
\widehat{\theta}_x&=\big(1+O(\rho_x(y))\big)\theta_{\mathbb{H}^n},\\
dV_{\widehat{\theta}_x}&=\big(1+O(\rho_x(y))\big)dV_{\theta_{\mathbb{H}^n}},\\
W_{k}&=\big(1+O(\rho_x(y)^4)\big)Z_k+O(\rho_x(y)^4)\overline{Z}_k+O(\rho_x(y)^5)\frac{\partial}{\partial t}\hspace{2mm}\mbox{ for }1\leq k\leq n,
\end{split}
\end{equation}
in $\{(z,t):(t^2+|z|^4)^{\frac{1}{4}}<\widehat{\rho}\}$ for some $\widehat{\rho}>0$.
Here $(z,t)$ represents the point  $y\in M$
in the CR normal coordinates centered at $x$,
$\rho_x(y)$ is the distance in the CR normal coordinates at $x$,
which implies that
\begin{equation}\label{B.4.5}
\rho_{x_k}(y)=(t^2+|z|^4)^{\frac{1}{4}}.
\end{equation}
Also, $\theta_{\mathbb{H}^n}=dt + \sqrt{-1}\sum_{j=1}^n(z_jd\overline{z}_j-\overline{z}_jdz_j)$ is the standard contact form of the Heisenberg group $\mathbb{H}^n=\{(z,t)=(z_1,...,z_n,t)\in\mathbb{C}^n\times\mathbb{R}\}$.
Moreover, $\{W_{k}\}$ is a pseudo-Hermitian frame, i.e. $\{W_{k}\}$ is a local frame of $\widehat{\theta}_x$ satisfying $-\sqrt{-1}d\widehat{\theta}_x(W_{k},\overline{W}_{l})=\delta_{kl}$
(see P.165 in \cite{Dragomir}), and
$Z_k=\frac{\partial}{\partial z_k}+\sqrt{-1}\overline{z}_k\frac{\partial}{\partial t}$.
We also have the following expression for the CR conformal sub-Laplacian: (see P.114 in \cite{Gamara2})
\begin{equation}\label{B.1.5}
-(2+\frac{2}{n})\Delta_{\widehat{\theta}_x}+R_{\widehat{\theta}_x}=-(2+\frac{2}{n})\Delta_{\theta_{\mathbb{H}^n}}+O(\rho_x(y)^2).
\end{equation}
 On the other hand,
the Webster scalar curvature of $\widehat{\theta}_x$ satisfies (see P.38 in \cite{Cheng&Malchiodi&Yang})
\begin{equation}\label{B.1}
|R_{\widehat{\theta}_x}(y)|\leq C\rho_x(y)^2,
\end{equation}
where $\rho_x(y)$ is the distance in the CR normal coordinates at $x$.
Let $G_x$ be the Green's function with pole at $x$. Then we have (see (105) in \cite{Cheng&Malchiodi&Yang})
\begin{equation}\label{B.2}
-(2+\frac{2}{n})\Delta_{\widehat{\theta}_x}G_x(y)+R_{\widehat{\theta}_x} G_x(y)=0\hspace{2mm}\mbox{ for }y\neq x.
\end{equation}
Moreover, the Green's function satisfies the estimates (see Proposition 5.2 and Proposition 5.3 in \cite{Cheng&Malchiodi&Yang})
\begin{equation}\label{B.3}
|G_x(y)-\rho_x(y)^{-2n}-A_x|\leq C\rho_x(y)
\end{equation}
and
\begin{equation}\label{B.4}
|\nabla_{\widehat{\theta}_x}(G_x(y)-\rho_x(y)^{-2n})|_{\widehat{\theta}_x}\leq C,
\end{equation}
where $A_x$ is the CR mass. Note that
$\rho_x(y)$, the distance in the CR normal coordinates at $x$,
and $d(x,y)$, the Carnot-Carath\'{e}odory distance on $M$
with respect to the contact form $\theta_0$
are equivalent, i.e. there exists a uniform constant $C_0$ such that
\begin{equation}\label{B.5.5}
\frac{1}{C_0}d(x,y)\leq \rho_x(y)\leq C_0 d(x,y)\hspace{2mm}\mbox{ for all }x, y\in M.
\end{equation}

When $M$ is spherical, give any $x\in M$,
we can find a smooth function $\varphi_x$ such that
$\theta_{\mathbb{H}^n}=\varphi_x^{\frac{2}{n}}\theta_0$ in a neighborhood of $x$. That is,
\begin{equation}\label{B.0a}
\widehat{\theta}_x=\theta_{\mathbb{H}^n}
\end{equation} in the above notation.
Note also that
$R_{\widehat{\theta}_x}=R_{\theta_{\mathbb{H}^n}}\equiv 0$ in this case. Therefore,  (\ref{B.0})
and (\ref{B.1})
are still true. On the other hand, the Green's function $G_x(y)$ defined as in (\ref{B.2}) satisfies
(\ref{B.3}) and (\ref{B.4}) by Lemma 5.1 in \cite{Cheng&Chiu&Yang}.

Under the assumptions of Theorem \ref{main}, the CR mass is positive, i.e.
$A_x>0$ for all $x\in M$ by the CR positive mass theorem (see Theorem 1.1 in \cite{Cheng&Malchiodi&Yang} and Corollary C in \cite{Cheng&Chiu&Yang}).
Note that the function $x\mapsto A_x$ is continuous.
See Proposition 3.3 in \cite{Wang} for the proof when
$M$ is spherical.
See also Remark I.1.2 in \cite{Habermann&Jost}
and Proposition 3.5 in \cite{Habermann}
for the proof of the corresponding statement in the Riemannian case.
Hence, we have
$$\inf_{x\in M}A_x>0.$$

Suppose that we are given a set of pairs $(x_k,\varepsilon_k)_{1\leq k\leq m}$. For every $1\leq k\leq m$, we define
$\overline{u}_{(x_k,\varepsilon_k)}$ by
\begin{equation}\label{testfcn}
\overline{u}_{(x_k,\varepsilon_k)}(y)=\varphi_{x_k}(y)\,\overline{U}_{(x_k,\varepsilon_k)}(y).
\end{equation}
Here
\begin{equation}\label{testfun2}
\overline{U}_{(x_k,\varepsilon_k)}(y)=\left(\frac{n(2n+2)}{r_\infty}\right)^{\frac{n}{2}}
\varepsilon_k^n
\left[\frac{\chi_\delta(\rho_{x_k}(y))}{(t^2+(\varepsilon_k^2+|z|^2)^2)^{\frac{n}{2}}}+
\Big(1-\chi_\delta(\rho_{x_k}(y))\Big)G_{x_k}(y)\right],
\end{equation}
where $\chi_\delta(s)=\chi(\frac{s}{\delta})$ and $\chi:\mathbb{R}\to[0,1]$ is a cut-off
function satisfying $\chi(s)=1$ for $s\leq 1$ and $\chi(s)=0$ for $s\geq 2$.
On the other hand, $\delta$ is a positive real number such that $\varepsilon_k\ll\delta$ for all $1\leq k\leq m$.

\begin{prop}\label{PropB.1}
For $n=1$, we have
\begin{equation*}
\begin{split}
&\left|\vphantom{\left(\frac{n(2n+2)}{r_\infty}\right)^{\frac{n}{2}}}(2+\frac{2}{n})\Delta_{\widehat{\theta}_{x_k}}\overline{U}_{(x_k,\varepsilon_k)}(y)
-R_{\widehat{\theta}_{x_k}}\overline{U}_{(x_k,\varepsilon_k)}(y)
+r_\infty\overline{U}_{(x_k,\varepsilon_k)}(y)^{1+\frac{2}{n}}\right.\\
&\left.
+\left(\frac{n(2n+2)}{r_\infty}\right)^{\frac{n}{2}}
\varepsilon_k^n A_{x_k}(2+\frac{2}{n})\Delta_{\widehat{\theta}_{x_k}}\chi_\delta(\rho_{x_k}(y))
\right|\\
&\leq C\left(\frac{\varepsilon_k^2}{t^2+(\varepsilon_k^2+|z|^2)^2}\right)^{\frac{n}{2}}\rho_{x_k}(y)^2\,1_{\{\rho_{x_k}(y)\leq 2\delta\}}
+C\frac{\varepsilon_k^n}{\delta}\,1_{\{\delta\leq\rho_{x_k}(y)\leq 2\delta\}}\\
&\hspace{4mm}
+C\left(\frac{\varepsilon_k^2}{t^2+(\varepsilon_k^2+|z|^2)^2}\right)^{\frac{n+2}{2}}1_{\{\rho_{x_k}(y)\geq\delta\}}.
\end{split}
\end{equation*}
\end{prop}
\begin{proof}
By definition of $\overline{U}_{(x_k,\varepsilon_k)}$, we have
\begin{equation*}
\begin{split}
&(2+\frac{2}{n})\Delta_{\widehat{\theta}_{x_k}}\overline{U}_{(x_k,\varepsilon_k)}(y)
-R_{\widehat{\theta}_{x_k}}\overline{U}_{(x_k,\varepsilon_k)}(y)
+r_\infty\overline{U}_{(x_k,\varepsilon_k)}(y)^{1+\frac{2}{n}}\\
&
+\left(\frac{n(2n+2)}{r_\infty}\right)^{\frac{n}{2}}
\varepsilon_k^n A_{x_k}\cdot(2+\frac{2}{n})\Delta_{\widehat{\theta}_{x_k}}\chi_\delta(\rho_{x_k}(y))
\\
&=\left(\frac{n(2n+2)}{r_\infty}\right)^{\frac{n}{2}}
\varepsilon_k^n\,\big(I_1+I_2+I_3+I_4+I_5+I_6+I_7+I_8\big).
\end{split}
\end{equation*}
Here,
\begin{equation*}
\begin{split}
I_1&=\chi_\delta(\rho_{x_k}(y))\left[\vphantom{\left(\frac{1}{(t^2+(\varepsilon_k^2+|z|^2)^2)^{\frac{n}{2}}}\right)^{1+\frac{2}{n}}}
(2+\frac{2}{n})
\Delta_{\widehat{\theta}_{x_k}}\frac{1}{(t^2+(\varepsilon_k^2+|z|^2)^2)^{\frac{n}{2}}}\right.\\
&\left.\hspace{12mm}+(2n+2)n\varepsilon_k^2
\left(\frac{1}{(t^2+(\varepsilon_k^2+|z|^2)^2)^{\frac{n}{2}}}\right)^{1+\frac{2}{n}}\right],\\
I_2&=-\chi_\delta(\rho_{x_k}(y))R_{\widehat{\theta}_{x_k}}\frac{1}{(t^2+(\varepsilon_k^2+|z|^2)^2)^{\frac{n}{2}}},\\
I_3&=-(2+\frac{2}{n})\Delta_{\widehat{\theta}_{x_k}}\chi_\delta(\rho_{x_k}(y))\big(
G_{x_k}(y)-\rho_{x_k}(y)^{-2n}-A_{x_k}\big),\\
I_4&=(2+\frac{2}{n})\Delta_{\widehat{\theta}_{x_k}}\chi_\delta(\rho_{x_k}(y))\Big(
\frac{1}{(t^2+(\varepsilon_k^2+|z|^2)^2)^{\frac{n}{2}}}-\rho_{x_k}(y)^{-2n}\Big),\\
I_5&=-2(2+\frac{2}{n})\langle\nabla_{\widehat{\theta}_{x_k}}\chi_\delta(\rho_{x_k}(y)),\nabla_{\widehat{\theta}_{x_k}}(G_x(y)-\rho_x(y)^{-2n})\rangle_{\widehat{\theta}_{x_k}},\\
I_6&=2(2+\frac{2}{n})\Big\langle\nabla_{\widehat{\theta}_{x_k}}\chi_\delta(\rho_{x_k}(y)),\nabla_{\widehat{\theta}_{x_k}}
\Big(
\frac{1}{(t^2+(\varepsilon_k^2+|z|^2)^2)^{\frac{n}{2}}}-\rho_{x_k}(y)^{-2n}\Big)\Big\rangle_{\widehat{\theta}_{x_k}},\\
I_7&=(2n+2)n\varepsilon_k^2\left[
\left(\frac{\chi_\delta(\rho_{x_k}(y))}{(t^2+(\varepsilon_k^2+|z|^2)^2)^{\frac{n}{2}}}+
\Big(1-\chi_\delta(\rho_{x_k}(y))\Big)G_{x_k}(y)\right)^{1+\frac{2}{n}}\right.\\
&\hspace{12mm}\left.
-\chi_\delta(\rho_{x_k}(y))\left(\frac{1}{(t^2+(\varepsilon_k^2+|z|^2)^2)^{\frac{n}{2}}}\right)^{1+\frac{2}{n}}\right],\\
I_8&=\Big(1-\chi_\delta(\rho_{x_k}(y)\Big)\left[
(2+\frac{2}{n})\Delta_{\widehat{\theta}_{x_k}}G_{x_k}(y)-R_{\widehat{\theta}_{x_k}} G_{x_k}(y)\right].
\end{split}
\end{equation*}
Since
\begin{equation}\label{B.16}
\begin{split}
\Delta_{\theta_{\mathbb{H}^n}}\left(\frac{1}{(t^2+(\varepsilon_k^2+|z|^2)^2)^{\frac{n}{2}}}\right)
&=-\frac{n^2\varepsilon_k^2}{(t^2+(\varepsilon_k^2+|z|^2)^2)^{1+\frac{n}{2}}}
\end{split}
\end{equation}
we have
\begin{equation*}
\begin{split}
|I_1|&\leq
\chi_\delta(\rho_{x_k}(y))\left|\vphantom{\left(\frac{1}{(t^2+(\varepsilon_k^2+|z|^2)^2)^{\frac{n}{2}}}\right)^{1+\frac{2}{n}}}
(2+\frac{2}{n})
(\Delta_{\theta_{\mathbb{H}^n}}+O(\rho_{x_k}(y)^2))\left(\frac{1}{(t^2+(\varepsilon_k^2+|z|^2)^2)^{\frac{n}{2}}}\right)\right.\\
&\left.\hspace{12mm}+(2n+2)n\varepsilon_k^2
\left(\frac{1}{(t^2+(\varepsilon_k^2+|z|^2)^2)^{\frac{n}{2}}}\right)^{1+\frac{2}{n}}\right|\\
&\leq 1_{\{\rho_{x_k}(y)\leq 2\delta\}}\frac{C\rho_{x_k}(y)^2}{(t^2+(\varepsilon_k^2+|z|^2)^2)^{\frac{n}{2}}}
\end{split}
\end{equation*}
by (\ref{B.0}) and (\ref{B.4.5}).
By (\ref{B.1}), we have
\begin{equation*}
|I_2|\leq C\rho_{x_k}(y)^2
\frac{1}{(t^2+(\varepsilon_k^2+|z|^2)^2)^{\frac{n}{2}}}\,1_{\{\rho_{x_k}(y)\leq 2\delta\}}.
\end{equation*}
By (\ref{B.3}), we have
\begin{equation*}
\begin{split}
|I_3|&\leq (2+\frac{2}{n})|\Delta_{\widehat{\theta}_{x_k}}\chi_\delta(\rho_{x_k}(y))|\cdot\big|
G_{x_k}(y)-\rho_{x_k}(y)^{-2n}-A_{x_k}\big|\\
&\leq C\frac{1_{\{\delta\leq\rho_{x_k}(y)\leq 2\delta\}} }{\delta^2}\rho_{x_k}(y)
\leq \frac{C}{\delta}1_{\{\delta\leq\rho_{x_k}(y)\leq 2\delta\}}.
\end{split}
\end{equation*}
Note that
\begin{equation*}
\begin{split}
|I_4|&=(2+\frac{2}{n})|\Delta_{\widehat{\theta}_{x_k}}\chi_\delta(\rho_{x_k}(y))|\cdot\left|
\frac{1}{(t^2+(\varepsilon_k^2+|z|^2)^2)^{\frac{n}{2}}}-\rho_{x_k}(y)^{-2n}\right|\\
&\leq C\frac{1_{\{\delta\leq\rho_{x_k}(y)\leq 2\delta\}}}{\delta^2}
\cdot\left|
\frac{1}{(t^2+(\varepsilon_k^2+|z|^2)^2)^{\frac{n}{2}}}-\frac{1}{(t^2+|z|^4)^{\frac{n}{2}}}\right|\\
&\leq C\frac{1_{\{\delta\leq\rho_{x_k}(y)\leq 2\delta\}}}{\delta^2}\cdot
\frac{\varepsilon_k^{2n}}{\delta^{4n}}\\
&
\leq \frac{C}{\delta}1_{\{\delta\leq\rho_{x_k}(y)\leq 2\delta\}}
\end{split}
\end{equation*}
by (\ref{B.4.5}) and the assumption that $\varepsilon_k\ll\delta$.
Note also that
\begin{equation*}
\begin{split}
|I_6|&\leq C|\nabla_{\widehat{\theta}_{x_k}}\chi_\delta(\rho_{x_k}(y))|_{\widehat{\theta}_{x_k}}\cdot\left|
\nabla_{\widehat{\theta}_{x_k}}\Big(
\frac{1}{(t^2+(\varepsilon_k^2+|z|^2)^2)^{\frac{n}{2}}}-\rho_{x_k}(y)^{-2n}\Big)\right|_{\widehat{\theta}_{x_k}}\\
&\leq C\frac{1_{\{\delta\leq\rho_{x_k}(y)\leq 2\delta\}}}{\delta}
\cdot|z|^2
\frac{\varepsilon_k^2}{(t^2+|z|^4)^{\frac{n}{2}+1}}\\
&\leq C\frac{\varepsilon_k^2}{\delta^{2n+3}}1_{\{\delta\leq\rho_{x_k}(y)\leq 2\delta\}}
\leq C\frac{1_{\{\delta\leq\rho_{x_k}(y)\leq 2\delta\}}}{\delta}
\end{split}
\end{equation*}
by (\ref{B.0}), (\ref{B.4.5}), and the assumption that $\varepsilon_k\ll\delta$.
By (\ref{B.4}), we have
\begin{equation*}
|I_5|\leq C|\nabla_{\widehat{\theta}_{x_k}}\chi_\delta(\rho_{x_k}(y))|_{\widehat{\theta}_{x_k}}\big|
\nabla_{\widehat{\theta}_{x_k}}(
G_{x_k}(y)-\rho_{x_k}(y)^{-2n})\big|_{\widehat{\theta}_{x_k}}\leq C\frac{1_{\{\delta\leq\rho_{x_k}(y)\leq 2\delta\}}}{\delta}.
\end{equation*}
We also have
\begin{equation*}
\begin{split}
|I_7|&=(2n+2)n\varepsilon_k^2\left|
\left(\frac{\chi_\delta(\rho_{x_k}(y))}{(t^2+(\varepsilon_k^2+|z|^2)^2)^{\frac{n}{2}}}+
\Big(1-\chi_\delta(\rho_{x_k}(y))\Big)G_{x_k}(y)\right)^{1+\frac{2}{n}}\right.\\
&\hspace{12mm}\left.
-\chi_\delta(\rho_{x_k}(y))\left(\frac{1}{(t^2+(\varepsilon_k^2+|z|^2)^2)^{\frac{n}{2}}}\right)^{1+\frac{2}{n}}\right|\\
&\leq C\varepsilon_k^2\left(
\left(\frac{1}{(t^2+(\varepsilon_k^2+|z|^2)^2)^{\frac{n}{2}}}\right)^{1+\frac{2}{n}}+
G_{x_k}(y)^{1+\frac{2}{n}}\right)1_{\{\rho_{x_k}(y)\geq\delta\}}\\
&\leq C\frac{\varepsilon_k^2}{(t^2+(\varepsilon_k^2+|z|^2)^2)^{1+\frac{2}{n}}}1_{\{\rho_{x_k}(y)\geq\delta\}}
\end{split}
\end{equation*}
by (\ref{B.3}), and
$$I_8=0$$
by (\ref{B.2}).
From this, the assertion follows.
\end{proof}

\begin{prop}\label{PropB.1a}
When $M$ is spherical, we have
\begin{equation*}
\begin{split}
&\left|\vphantom{\left(\frac{n(2n+2)}{r_\infty}\right)^{\frac{n}{2}}}(2+\frac{2}{n})\Delta_{\widehat{\theta}_{x_k}}\overline{U}_{(x_k,\varepsilon_k)}(y)
-R_{\widehat{\theta}_{x_k}}\overline{U}_{(x_k,\varepsilon_k)}(y)
+r_\infty\overline{U}_{(x_k,\varepsilon_k)}(y)^{1+\frac{2}{n}}\right.\\
&\left.
+\left(\frac{n(2n+2)}{r_\infty}\right)^{\frac{n}{2}}
\varepsilon_k^n A_{x_k}(2+\frac{2}{n})\Delta_{\widehat{\theta}_{x_k}}\chi_\delta(\rho_{x_k}(y))
\right|\\
&\leq C\frac{\varepsilon_k^n}{\delta}\,1_{\{\delta\leq\rho_{x_k}(y)\leq 2\delta\}}
+C\left(\frac{\varepsilon_k^2}{t^2+(\varepsilon_k^2+|z|^2)^2}\right)^{\frac{n+2}{2}}1_{\{\rho_{x_k}(y)\geq\delta\}}.
\end{split}
\end{equation*}
\end{prop}
\begin{proof}
The proof is the same as the proof of Proposition \ref{PropB.1}, except we need to prove that $I_1=0$.
But this follows from (\ref{B.0a}) and (\ref{B.16}).
\end{proof}

\begin{cor}\label{CorB.2}
We have
\begin{equation*}
\begin{split}
&\left|(2+\frac{2}{n})\Delta_{\widehat{\theta}_{x_k}}\overline{U}_{(x_k,\varepsilon_k)}(y)
-R_{\widehat{\theta}_{x_k}}\overline{U}_{(x_k,\varepsilon_k)}(y)
+r_\infty\overline{U}_{(x_k,\varepsilon_k)}(y)^{1+\frac{2}{n}}\right|\\
&\leq C\left(\frac{\varepsilon_k^2}{t^2+(\varepsilon_k^2+|z|^2)^2}\right)^{\frac{n}{2}}\rho_{x_k}(y)^2\,1_{\{\rho_{x_k}(y)\leq 2\delta\}}
+C\frac{\varepsilon_k^n}{\delta^2}\,1_{\{\delta\leq\rho_{x_k}(y)\leq 2\delta\}}\\
&\hspace{4mm}
+C\left(\frac{\varepsilon_k^2}{t^2+(\varepsilon_k^2+|z|^2)^2}\right)^{\frac{n+2}{2}}1_{\{\rho_{x_k}(y)\geq\delta\}}.
\end{split}
\end{equation*}
\end{cor}
\begin{proof}
Combining Proposition \ref{PropB.1},  Proposition \ref{PropB.1a} and the following estimate:
\begin{equation*}
\left|\left(\frac{n(2n+2)}{r_\infty}\right)^{\frac{n}{2}}
\varepsilon_k^n A_{x_k}(2+\frac{2}{n})\Delta_{\widehat{\theta}_{x_k}}\chi_\delta(\rho_{x_k}(y))\right|
\leq C\frac{\varepsilon_k^n}{\delta^2}\,1_{\{\delta\leq\rho_{x_k}(y)\leq 2\delta\}},
\end{equation*}
Corollary \ref{CorB.2} follows.
\end{proof}

\begin{prop}\label{PropB.3}
If
$\delta$ is sufficiently small, then
$$E(\overline{u}_{(x_k,\varepsilon_k)})\leq Y(S^{2n+1})-cA_{x_k}\varepsilon_k^{2n}
+C\delta^2\varepsilon_k^{2n}+C\delta\varepsilon_k^{2n}+C\delta^{-2n-2}\varepsilon_k^{2n+2}$$
for some $c>0$.
\end{prop}
\begin{proof}
When $n=1$, it follows from Proposition \ref{PropB.1} and integration by parts that
\begin{equation}\label{B.5}
\begin{split}
&\int_M\Big((2+\frac{2}{n})|\nabla_{\widehat{\theta}_{x_k}}\overline{U}_{(x_k,\varepsilon_k)}|^2_{\widehat{\theta}_{x_k}}
+R_{\widehat{\theta}_{x_k}}\overline{U}_{(x_k,\varepsilon_k)}^2-r_\infty\overline{U}_{(x_k,\varepsilon_k)}^{2+\frac{2}{n}}\Big)dV_{\widehat{\theta}_{x_k}}\\
&=-\int_M\Big((2+\frac{2}{n})\Delta_{\widehat{\theta}_{x_k}}\overline{U}_{(x_k,\varepsilon_k)}
-R_{\widehat{\theta}_{x_k}}\overline{U}_{(x_k,\varepsilon_k)}+r_\infty\overline{U}_{(x_k,\varepsilon_k)}^{1+\frac{2}{n}}\Big)
\overline{U}_{(x_k,\varepsilon_k)}dV_{\widehat{\theta}_{x_k}}\\
&\leq\left(\frac{n(2n+2)}{r_\infty}\right)^{\frac{n}{2}}
\varepsilon_k^n A_{x_k}(2+\frac{2}{n})
\int_M\Delta_{\widehat{\theta}_{x_k}}\chi_\delta(\rho_{x_k}(y))\overline{U}_{(x_k,\varepsilon_k)}dV_{\widehat{\theta}_{x_k}}
\\
&\hspace{4mm}+ C\int_{\{\rho_{x_k}(y)\leq 2\delta\}}\left(\frac{\varepsilon_k^2}{t^2+(\varepsilon_k^2+|z|^2)^2}\right)^{\frac{n}{2}}\rho_{x_k}(y)^2
\overline{U}_{(x_k,\varepsilon_k)}dV_{\widehat{\theta}_{x_k}}\\
&\hspace{4mm}
+C\frac{\varepsilon_k^n}{\delta}\int_{\{\delta\leq\rho_{x_k}(y)\leq 2\delta\}}\overline{U}_{(x_k,\varepsilon_k)}dV_{\widehat{\theta}_{x_k}}\\
&\hspace{4mm}
+C\int_{\{\rho_{x_k}(y)\geq\delta\}}\left(\frac{\varepsilon_k^2}{t^2+(\varepsilon_k^2+|z|^2)^2}\right)^{\frac{n+2}{2}}
\overline{U}_{(x_k,\varepsilon_k)}dV_{\widehat{\theta}_{x_k}}.
\end{split}
\end{equation}
We are going to estimate each of the terms on the right hand side of (\ref{B.5}). Since
\begin{equation*}
\begin{split}
&\left|\int_M\Delta_{\widehat{\theta}_{x_k}}\chi_\delta(\rho_{x_k}(y))
\Big(\overline{U}_{(x_k,\varepsilon_k)}-\rho_{x_k}(y)^{-2n}\Big)dV_{\widehat{\theta}_{x_k}}\right|\\
&\leq \frac{C}{\delta^2}\int_{\{\delta\leq\rho_{x_k}(y)\leq 2\delta\}}
\left|\frac{1}{(t^2+(\varepsilon_k^2+|z|^2)^2)^{\frac{n}{2}}}-\frac{1}{(t^2+|z|^4)^{\frac{n}{2}}}\right|dV_{\widehat{\theta}_{x_k}}\\
&\hspace{4mm}+
\frac{C}{\delta^2}\int_{\{\delta\leq\rho_{x_k}(y)\leq 2\delta\}}|G_{x_k}(y)-\rho_{x_k}(y)^{-2n}|dV_{\widehat{\theta}_{x_k}}\\
&\leq \frac{C}{\delta^2}\int_\delta^{2\delta}\frac{\varepsilon_k^{2n}r^{2n-1}dr}{r^{4n}}
\leq C\delta^{-2n-2}\varepsilon_k^{2n}
\end{split}
\end{equation*}
by (\ref{B.3}) and (\ref{B.4.5}), we have
\begin{equation}\label{B.7}
\begin{split}
&\int_M\Delta_{\widehat{\theta}_{x_k}}\chi_\delta(\rho_{x_k}(y))\overline{U}_{(x_k,\varepsilon_k)}dV_{\widehat{\theta}_{x_k}}\\
&=-\int_M\big\langle\nabla_{\widehat{\theta}_{x_k}}\rho_{x_k}(y)^{-2n},
\nabla_{\widehat{\theta}_{x_k}}\chi_\delta(\rho_{x_k}(y))\big\rangle_{\widehat{\theta}_{x_k}}
dV_{\widehat{\theta}_{x_k}}
+C\delta^{-2n-2}\varepsilon_k^{2n}\\
&=2n\int_{\{\delta\leq\rho_{x_k}(y)\leq 2\delta\}}\chi'_\delta(\rho_{x_k}(y))\rho_{x_k}(y)^{-2n-1}|\nabla_{\widehat{\theta}_{x_k}}\rho_{x_k}(y)|^2_{\widehat{\theta}_{x_k}}
dV_{\widehat{\theta}_{x_k}}
+C\delta^{-2n-2}\varepsilon_k^{2n}\\
&=-\frac{C}{\delta}\int_\delta^{2\delta}
dr
+C\delta^{-2n-2}\varepsilon_k^{2n}=-C_0
+C\delta^{-2n-2}\varepsilon_k^{2n}
\end{split}
\end{equation}
by integration by parts. Here $C_0$ is a positive constant depending only on $n$.
On the other hand, it follows form the definition of $\overline{U}_{(x_k,\varepsilon_k)}$ that
\begin{equation}\label{B.8}
\begin{split}
&C\int_{\{\rho_{x_k}(y)\leq 2\delta\}}\left(\frac{\varepsilon_k^2}{t^2+(\varepsilon_k^2+|z|^2)^2}\right)^{\frac{n}{2}}\rho_{x_k}(y)^2
\overline{U}_{(x_k,\varepsilon_k)}dV_{\widehat{\theta}_{x_k}}\\
&\leq C\varepsilon_k^{2n}\int_{\{\rho_{x_k}(y)\leq 2\delta\}}\frac{\rho_{x_k}(y)^2}{(t^2+|z|^4)^{n}}
dV_{\widehat{\theta}_{x_k}}=C\delta^{4-2n}\varepsilon_k^{2n},
\end{split}
\end{equation}
where we have used the assumption that $n=1$. Note also that
\begin{equation}\label{B.9}
\begin{split}
&C\frac{\varepsilon_k^n}{\delta}\int_{\{\delta\leq\rho_{x_k}(y)\leq 2\delta\}}\overline{U}_{(x_k,\varepsilon_k)}dV_{\widehat{\theta}_{x_k}}\\
&\leq C\frac{\varepsilon_k^{2n}}{\delta}
\int_{\{\delta\leq\rho_{x_k}(y)\leq 2\delta\}}\Big(\rho_{x_k}(y)^{-2n}+A_{x_k}+C\rho_{x_k}(y)\Big)dV_{\widehat{\theta}_{x_k}}\\
&\leq C\frac{\varepsilon_k^{2n}}{\delta}\int_\delta^{2\delta}\frac{r^{2n+1}dr}{r^{2n}}=C\delta\varepsilon_k^{2n},
\end{split}
\end{equation}
and
\begin{equation}\label{B.10}
\begin{split}
&C\int_{\{\rho_{x_k}(y)\geq\delta\}}\left(\frac{\varepsilon_k^2}{t^2+(\varepsilon_k^2+|z|^2)^2}\right)^{\frac{n+2}{2}}
\overline{U}_{(x_k,\varepsilon_k)}dV_{\widehat{\theta}_{x_k}}\\
&\leq C\varepsilon_k^{2n+2}\int_\delta^\infty\frac{r^{2n+1}dr}{r^{4(\frac{n+2}{2})}\cdot r^{2n}}
=C\varepsilon_k^{2n+2}\int_\delta^\infty\frac{dr}{r^{2n+3}}
=C\delta^{-2n-2}\varepsilon_k^{2n+2},
\end{split}
\end{equation}
where we have used (\ref{B.3}) and the definition of $\overline{U}_{(x_k,\varepsilon_k)}$.
Combining (\ref{B.5})-(\ref{B.10}), we have
\begin{equation}\label{B.11}
\begin{split}
&\int_M\Big((2+\frac{2}{n})|\nabla_{\widehat{\theta}_{x_k}}\overline{U}_{(x_k,\varepsilon_k)}|^2_{\widehat{\theta}_{x_k}}
+R_{\widehat{\theta}_{x_k}}\overline{U}_{(x_k,\varepsilon_k)}^2\Big)dV_{\widehat{\theta}_{x_k}}\\
&\leq r_\infty
\int_M\overline{U}_{(x_k,\varepsilon_k)}^{2+\frac{2}{n}}dV_{\widehat{\theta}_{x_k}}
-\left(\frac{n(2n+2)}{r_\infty}\right)^{n}
\varepsilon_k^{2n} A_{x_k}(4+\frac{4}{n})C_0\\
&\hspace{4mm}
+C\delta^2\varepsilon_k^{2n}+C\delta\varepsilon_k^{2n}+C\delta^{-2n-2}\varepsilon_k^{2n+2}\\
&\leq r_\infty
\left(\int_M\overline{U}_{(x_k,\varepsilon_k)}^{2+\frac{2}{n}}dV_{\widehat{\theta}_{x_k}}\right)^{\frac{n}{n+1}}
\left(\int_M\overline{U}_{(x_k,\varepsilon_k)}^{2+\frac{2}{n}}dV_{\widehat{\theta}_{x_k}}\right)^{\frac{1}{n+1}}
\\
&\hspace{4mm}-\left(\frac{n(2n+2)}{r_\infty}\right)^{n}
\varepsilon_k^{2n} A_{x_k}(2+\frac{2}{n})C_0
+C\delta^2\varepsilon_k^{2n}+C\delta\varepsilon_k^{2n}+C\delta^{-2n-2}\varepsilon_k^{2n+2}
\end{split}
\end{equation}
where the last inequality follows from H\"{o}lder's inequality.
It follows from the definition of $\overline{U}_{(x_k,\varepsilon_k)}$ that
\begin{equation}\label{B.12}
\begin{split}
\int_M\overline{U}_{(x_k,\varepsilon_k)}^{2+\frac{2}{n}}dV_{\widehat{\theta}_{x_k}}
&\leq \left(\frac{n(2n+2)}{r_\infty}\right)^{n+1}\left(
\int_M\chi_\delta(\rho_{x_k}(y))
\left(\frac{\varepsilon_k^2}{t^2+(\varepsilon_k^2+|z|^2)^2}\right)^{n+1}dV_{\widehat{\theta}_{x_k}}\right.\\
&\hspace{4mm}\left.+C\vphantom{\frac{\varepsilon_k^2}{t^2+(\varepsilon_k^2+|z|^2)^2}}
\int_M(1-\chi_\delta(\rho_{x_k}(y)))\varepsilon_k^{2n+2}G_{x_k}(y)^{2+\frac{2}{n}}dV_{\widehat{\theta}_{x_k}}\right).
\end{split}
\end{equation}
Note that
\begin{equation}\label{B.13}
\begin{split}
&
\int_M(1-\chi_\delta(\rho_{x_k}(y)))\varepsilon_k^{2n+2}G_{x_k}(y)^{2+\frac{2}{n}}dV_{\widehat{\theta}_{x_k}}\\
&\leq \varepsilon_k^{2n+2}\int_{\rho_{x_k}(y)\geq\delta}\Big(\rho_{x_k}(y)^{-2n}+A_{x_k}+C\rho_{x_k}(y)\Big)^{2+\frac{2}{n}}dV_{\widehat{\theta}_{x_k}}\\
&=C\varepsilon_k^{2n+2}\int_\delta^\infty\frac{r^{2n+1}dr}{r^{2n(2+\frac{2}{n})}}
=C\varepsilon_k^{2n+2}\int_\delta^\infty\frac{dr}{r^{2n+3}}=C\delta^{-2n-2}\varepsilon_k^{2n+2}
\end{split}
\end{equation}
by (\ref{B.3}).
Note also that
\begin{equation}\label{B.14}
\begin{split}
&\int_M\chi_\delta(\rho_{x_k}(y))
\left(\frac{\varepsilon_k^2}{t^2+(\varepsilon_k^2+|z|^2)^2}\right)^{n+1}dV_{\widehat{\theta}_{x_k}}\\
&=\int_{\mathbb{H}^n}\left(\frac{1}{(t^2+(1+|z|^2)^2)^{\frac{n}{2}}}\right)^{2+\frac{2}{n}}dV_{\theta_{\mathbb{H}^n}}
+O(\varepsilon_k^{2n+2})\\
&=\left(\frac{Y(S^{2n+1})}{n(2n+2)}\right)^{n+1}+O(\varepsilon_k^{2n+2}),
\end{split}
\end{equation}
where we have used the change of variables $(\frac{t}{\varepsilon_k^2},\frac{z}{\varepsilon_k})\mapsto(t,z)$.
Combining (\ref{B.12})-(\ref{B.14}), we get
\begin{equation*}
\int_M\overline{U}_{(x_k,\varepsilon_k)}^{2+\frac{2}{n}}dV_{\widehat{\theta}_{x_k}}
\leq\left(\frac{Y(S^{2n+1})}{r_\infty}\right)^{n+1}+C\delta^{-2n-2}\varepsilon_k^{2n+2}.
\end{equation*}
Substituting this into (\ref{B.11}), we obtain
\begin{equation*}
\begin{split}
&\int_M\Big((2+\frac{2}{n})|\nabla_{\widehat{\theta}_{x_k}}\overline{U}_{(x_k,\varepsilon_k)}|^2_{\widehat{\theta}_{x_k}}
+R_{\widehat{\theta}_{x_k}}\overline{U}_{(x_k,\varepsilon_k)}^2\Big)dV_{\widehat{\theta}_{x_k}}\\
&\leq Y(S^{2n+1})
\left(\int_M\overline{U}_{(x_k,\varepsilon_k)}^{2+\frac{2}{n}}dV_{\widehat{\theta}_{x_k}}\right)^{\frac{n}{n+1}}
-\left(\frac{n(2n+2)}{r_\infty}\right)^{n}
\varepsilon_k^{2n} A_{x_k}(2+\frac{2}{n})C_0
\\
&\hspace{4mm}
+C\delta^2\varepsilon_k^{2n}+C\delta\varepsilon_k^{2n}+C\delta^{-2n-2}\varepsilon_k^{2n+2}.
\end{split}
\end{equation*}
This proves Proposition \ref{PropB.3} for the case when $n=1$.

When $M$ is spherical, it follows from Proposition \ref{PropB.1a}
and integration by parts that
\begin{equation}\label{B.5a}
\begin{split}
&\int_M\Big((2+\frac{2}{n})|\nabla_{\widehat{\theta}_{x_k}}\overline{U}_{(x_k,\varepsilon_k)}|^2_{\widehat{\theta}_{x_k}}
+R_{\widehat{\theta}_{x_k}}\overline{U}_{(x_k,\varepsilon_k)}^2-r_\infty\overline{U}_{(x_k,\varepsilon_k)}^{2+\frac{2}{n}}\Big)dV_{\widehat{\theta}_{x_k}}\\
&\leq\left(\frac{n(2n+2)}{r_\infty}\right)^{\frac{n}{2}}
\varepsilon_k^n A_{x_k}(2+\frac{2}{n})
\int_M\Delta_{\widehat{\theta}_{x_k}}\chi_\delta(\rho_{x_k}(y))\overline{U}_{(x_k,\varepsilon_k)}dV_{\widehat{\theta}_{x_k}}
\\
&\hspace{4mm}
+C\frac{\varepsilon_k^n}{\delta}\int_{\{\delta\leq\rho_{x_k}(y)\leq 2\delta\}}\overline{U}_{(x_k,\varepsilon_k)}dV_{\widehat{\theta}_{x_k}}\\
&\hspace{4mm}
+C\int_{\{\rho_{x_k}(y)\geq\delta\}}\left(\frac{\varepsilon_k^2}{t^2+(\varepsilon_k^2+|z|^2)^2}\right)^{\frac{n+2}{2}}
\overline{U}_{(x_k,\varepsilon_k)}dV_{\widehat{\theta}_{x_k}}.
\end{split}
\end{equation}
Combining (\ref{B.7}), (\ref{B.9}), (\ref{B.10}), and (\ref{B.5a}), we also obtain
(\ref{B.11}). Now we can follow the same argument as above to finish the proof.
\end{proof}

\begin{lem}\label{LemmaB.6}
We have
$$F(\overline{u}_{(x_{k},\varepsilon_{k})})=r_\infty+o(1).$$
\end{lem}
\begin{proof}
It follows from (\ref{B.11}) that
\begin{equation*}
\begin{split}
&\int_M\Big((2+\frac{2}{n})|\nabla_{\widehat{\theta}_{x_k}}\overline{U}_{(x_k,\varepsilon_k)}|^2_{\widehat{\theta}_{x_k}}
+R_{\widehat{\theta}_{x_k}}\overline{U}_{(x_k,\varepsilon_k)}^2\Big)dV_{\widehat{\theta}_{x_k}}\\
&\leq r_\infty
\int_M\overline{U}_{(x_k,\varepsilon_k)}^{2+\frac{2}{n}}dV_{\widehat{\theta}_{x_k}}
-\left(\frac{n(2n+2)}{r_\infty}\right)^{n}
\varepsilon_k^{2n} A_{x_k}(2+\frac{2}{n})C_0\\
&\hspace{4mm}
+C\delta^2\varepsilon_k^{2n}+C\delta\varepsilon_k^{2n}+C\delta^{-2n-2}\varepsilon_k^{2n+2},
\end{split}
\end{equation*}
which implies that
\begin{equation*}
\int_M\Big((2+\frac{2}{n})|\nabla_{\theta_0}\overline{u}_{(x_k,\varepsilon_k)}|^2_{\theta_0}
+R_{\theta_0}\overline{u}_{(x_k,\varepsilon_k)}^2\Big)dV_{\theta_0}\leq r_\infty
\int_M\overline{u}_{(x_k,\varepsilon_k)}^{2+\frac{2}{n}}dV_{\theta_0}
+o(1).
\end{equation*}
This implies that $F(\overline{u}_{(x_{k},\varepsilon_{k})})\leq r_\infty+o(1)$.
On the other hand, following the proof of Proposition \ref{PropB.3}, one can prove that
\begin{equation*}
\begin{split}
&\int_M\Big((2+\frac{2}{n})|\nabla_{\widehat{\theta}_{x_k}}\overline{U}_{(x_k,\varepsilon_k)}|^2_{\widehat{\theta}_{x_k}}
+R_{\widehat{\theta}_{x_k}}\overline{U}_{(x_k,\varepsilon_k)}^2\Big)dV_{\widehat{\theta}_{x_k}}\\
&\geq r_\infty
\int_M\overline{U}_{(x_k,\varepsilon_k)}^{2+\frac{2}{n}}dV_{\widehat{\theta}_{x_k}}
-\left(\frac{n(2n+2)}{r_\infty}\right)^{n}
\varepsilon_k^{2n} A_{x_k}(2+\frac{2}{n})C_0\\
&\hspace{4mm}
-C\delta^2\varepsilon_k^{2n}-C\delta\varepsilon_k^{2n}-C\delta^{-2n-2}\varepsilon_k^{2n+2},
\end{split}
\end{equation*}
by using the same arguments to obtain (\ref{B.11}). This implies that
\begin{equation*}
\int_M\Big((2+\frac{2}{n})|\nabla_{\theta_0}\overline{u}_{(x_k,\varepsilon_k)}|^2_{\theta_0}
+R_{\theta_0}\overline{u}_{(x_k,\varepsilon_k)}^2\Big)dV_{\theta_0}\geq r_\infty
\int_M\overline{u}_{(x_k,\varepsilon_k)}^{2+\frac{2}{n}}dV_{\theta_0}
+o(1),
\end{equation*}
which gives $F(\overline{u}_{(x_{k},\varepsilon_{k})})\geq r_\infty+o(1)$. This proves the assertion.
\end{proof}

\begin{lem}\label{LemmaB.4}
We have the estimate
$$\int_M\overline{u}_{(x_i,\varepsilon_i)}\overline{u}_{(x_j,\varepsilon_j)}^{1+\frac{2}{n}}dV_{\theta_0}
\leq C\left(\frac{\varepsilon_i^2\varepsilon_j^2}{\varepsilon_j^4+d(x_i,x_j)^4}\right)^{\frac{n}{2}}.$$
\end{lem}
\begin{proof}
It follows from the definition of $\overline{u}_{(x_i,\varepsilon_i)}$
that $\overline{u}_{(x_i,\varepsilon_i)}(y)=\varphi_{x_i}(y)\overline{U}_{(x_i,\varepsilon_i)}(y)$ and
\begin{equation*}
\overline{U}_{(x_i,\varepsilon_i)}(y)\leq\left(\frac{n(2n+2)}{r_\infty}\right)^{\frac{n}{2}}
\varepsilon_i^n
\left[\frac{\chi_\delta(\rho_{x_i}(y))}{(\varepsilon_i^4+d(x_i,y)^4)^{\frac{n}{2}}}+
\Big(1-\chi_\delta(\rho_{x_i}(y))\Big)G_{x_i}(y)\right].
\end{equation*}
On the set $U=:\{2d(x_i,y)\leq\varepsilon_j+d(x_i,x_j)\}$, we have
\begin{equation}\label{B.15}
\varepsilon_j+d(y,x_j)\geq \varepsilon_j+d(x_i,x_j)-d(x_i,y)
\geq\frac{1}{2}(\varepsilon_j+d(x_i,x_j)).
\end{equation}
Therefore, we have
\begin{equation*}
\begin{split}
&\int_M\left(\frac{\varepsilon_i^2}{\varepsilon_i^4+d(x_i,y)^4}\right)^{\frac{n}{2}}
\left(\frac{\varepsilon_j^2}{\varepsilon_j^4+d(y,x_j)^4}\right)^{1+\frac{n}{2}}dV_{\theta_0}\\
&\leq\left(\int_{U}+\int_{M\setminus U}\right)
\left(\frac{\varepsilon_i^2}{\varepsilon_i^4+d(x_i,y)^4}\right)^{\frac{n}{2}}
\left(\frac{\varepsilon_j^2}{\varepsilon_j^4+d(y,x_j)^4}\right)^{1+\frac{n}{2}}dV_{\theta_0}\\
&\leq C\frac{\varepsilon_i^n\varepsilon_j^n}{(\varepsilon_j^4+d(x_i,x_j)^4)^{\frac{n+1}{2}}}
\int_0^{\frac{\varepsilon_j+d(x_i,x_j)}{2}}\frac{r^{2n+1}dr}{(\varepsilon_i^4+r^4)^{\frac{n}{2}}}\\
&\hspace{4mm}
+C\frac{\varepsilon_i^n\varepsilon_j^{n+2}}{(\varepsilon_j^4+d(x_i,x_j)^4)^{\frac{n}{2}}}
\int_0^{\infty}\frac{r^{2n+1}dr}{(\varepsilon_j^4+r^4)^{1+\frac{n}{2}}}\\
&\leq C\left(\frac{\varepsilon_i^2\varepsilon_j^2}{\varepsilon_j^4+d(x_i,x_j)^4}\right)^{\frac{n}{2}}
\end{split}
\end{equation*}
because
\begin{equation*}
\begin{split}
\int_0^{\frac{\varepsilon_j+d(x_i,x_j)}{2}}\frac{r^{2n+1}dr}{(\varepsilon_i^4+r^4)^{\frac{n}{2}}}
\leq\int_0^{\frac{\varepsilon_j+d(x_i,x_j)}{2}}rdr
=\frac{(\varepsilon_j+d(x_i,x_j))^2}{8}
\leq (\varepsilon_j^4+d(x_i,x_j)^4)^\frac{1}{2}
\end{split}
\end{equation*}
and
\begin{equation*}
\begin{split}
\int_0^{\infty}\frac{r^{2n+1}dr}{(\varepsilon_j^4+r^4)^{1+\frac{n}{2}}}
&\leq\int_0^{\varepsilon_j}\frac{r^{2n+1}dr}{\varepsilon_j^4r^{4\cdot\frac{n}{2}}}
+\int_{\varepsilon_j}^{\infty}\frac{r^{2n+1}dr}{r^{4(1+\frac{n}{2}})}=\frac{C}{\varepsilon_j^2}.
\end{split}
\end{equation*}
This proves the assertion.
\end{proof}

\begin{lem}\label{LemmaB.5}
We have the estimate
\begin{equation*}
\begin{split}
&\int_M\overline{u}_{(x_i,\varepsilon_i)}\left|(2+\frac{2}{n})\Delta_{\theta_0}\overline{u}_{(x_j,\varepsilon_j)}
-R_{\theta_0}\overline{u}_{(x_j,\varepsilon_j)}+r_\infty\overline{u}_{(x_j,\varepsilon_j)}^{1+\frac{2}{n}}\right|dV_{\theta_0}\\
&\leq C(\delta^4+\delta^{2n}+\frac{\varepsilon_j^2}{\delta^2})\left(\frac{\varepsilon_i^2\varepsilon_j^2}{\varepsilon_j^4+d(x_i,x_j)^4}\right)^{\frac{n}{2}}.
\end{split}
\end{equation*}
\end{lem}
\begin{proof}
Recall $\widehat{\theta}_x=\varphi_x^{\frac{2}{n}}\theta_0$ for some smooth function $\varphi_x$ on $M$. Therefore,
we can find a positive constant $c_0$ such that
$c_0^{-n}\leq\varphi_x\leq c_0^{n},$
 which implies
$\frac{1}{c_0}d(x,y)\leq\rho_x(y)\leq c_0d(x,y)$ for all $x, y\in M$. Hence, it follows from Corollary \ref{CorB.2} that
\begin{equation*}
\begin{split}
&\left|(2+\frac{2}{n})\Delta_{\theta_0}\overline{u}_{(x_j,\varepsilon_j)}
-R_{\theta_0}\overline{u}_{(x_j,\varepsilon_j)}+r_\infty\overline{u}_{(x_j,\varepsilon_j)}^{1+\frac{2}{n}}\right|\\
&\leq
C(\delta^2+\delta^{2n-2})\left(\frac{\varepsilon_j^2}{\varepsilon_j^4+d(x_j,y)^4}\right)^{\frac{n}{2}}1_{\{d(x_j,y)\leq 2c_0\delta\}}
+C\left(\frac{\varepsilon_j^2}{\varepsilon_j^4+d(x_j,y)^4}\right)^{\frac{n+2}{2}}1_{\{d(x_j,y)\geq\frac{\delta}{c_0}\}}.
\end{split}
\end{equation*}
On the set $U:=\{2d(x_i,y)\leq \varepsilon_j+d(x_i,x_j)\}\cap\{d(y,x_j)\leq2c_0\delta\}$, we have
$$d(x_i,y)\leq\frac{1}{2}(\varepsilon_j+d(x_i,x_j))\leq \varepsilon_j+d(y,x_j)\leq 4c_0\delta.$$
From this, it follows that
\begin{equation*}
\begin{split}
&\int_{\{d(y,x_j)\leq2c_0\delta\}}
\left(\frac{\varepsilon_i^2}{\varepsilon_i^4+d(x_i,y)^4}\right)^{\frac{n}{2}}
\left(\frac{\varepsilon_j^2}{\varepsilon_j^4+d(y,x_j)^4}\right)^{\frac{n}{2}}dV_{\theta_0}\\
&\leq \left(\int_{U}
+
\int_{\{d(y,x_j)\leq2c_0\delta\}\setminus U}\right)
\left(\frac{\varepsilon_i^2}{\varepsilon_i^4+d(x_i,y)^4}\right)^{\frac{n}{2}}
\left(\frac{\varepsilon_j^2}{\varepsilon_j^4+d(y,x_j)^4}\right)^{\frac{n}{2}}dV_{\theta_0}\\
&\leq C\int_{\{d(x_i,y)\leq4c_0\delta\}}
\left(\frac{\varepsilon_i^2}{\varepsilon_i^4+d(x_i,y)^4}\right)^{\frac{n}{2}}
\left(\frac{\varepsilon_j^2}{\varepsilon_j^4+d(x_i,x_j)^4}\right)^{\frac{n}{2}}dV_{\theta_0}\\
&\hspace{4mm}+C\int_{\{d(y,x_j)\leq2c_0\delta\}}
\left(\frac{\varepsilon_i^2}{\varepsilon_j^4+d(x_i,x_j)^4}\right)^{\frac{n}{2}}
\left(\frac{\varepsilon_j^2}{\varepsilon_j^4+d(y,x_j)^4}\right)^{\frac{n}{2}}dV_{\theta_0}\\
&\leq C\delta^2\left(\frac{\varepsilon_i^2\varepsilon_j^2}{\varepsilon_j^4+d(x_i,x_j)^4}\right)^{\frac{n}{2}}.
\end{split}
\end{equation*}
Similarly, on the set $V:=\{2d(x_i,y)\leq \varepsilon_j+d(x_i,x_j)\}\cap\{d(y,x_j)\geq\frac{\delta}{c_0}\}$, we have
$$\varepsilon_j+d(y,x_j)\geq\varepsilon_j+d(x_i,x_j)-d(x_i,y)\geq\frac{1}{2}(\varepsilon_j+d(x_i,x_j)).$$
This implies
\begin{equation*}
\begin{split}
&\int_{\{d(y,x_j)\geq\frac{\delta}{c_0}\}}
\left(\frac{\varepsilon_i^2}{\varepsilon_i^4+d(x_i,y)^4}\right)^{\frac{n}{2}}
\left(\frac{\varepsilon_j^2}{\varepsilon_j^4+d(y,x_j)^4}\right)^{\frac{n+2}{2}}dV_{\theta_0}\\
&\leq \left(\int_{V}
+
\int_{\{d(y,x_j)\geq\frac{\delta}{c_0}\}\setminus V}\right)
\left(\frac{\varepsilon_i^2}{\varepsilon_i^4+d(x_i,y)^4}\right)^{\frac{n}{2}}
\left(\frac{\varepsilon_j^2}{\varepsilon_j^4+d(y,x_j)^4}\right)^{\frac{n+2}{2}}dV_{\theta_0}\\
&\leq C\int_{\{2d(x_i,y)\leq \varepsilon_j+d(x_i,x_j)\}}
\left(\frac{\varepsilon_i^2}{\varepsilon_i^4+d(x_i,y)^4}\right)^{\frac{n}{2}}
\frac{\varepsilon_j^{n+2}}{\delta^2(\varepsilon_i^4+d(x_i,x_j)^4)^{\frac{n+1}{2}}}dV_{\theta_0}\\
&\hspace{4mm}+C\int_{\{d(y,x_j)\geq\frac{\delta}{c_0}\}}
\left(\frac{\varepsilon_i^2}{\varepsilon_j^4+d(x_i,x_j)^4}\right)^{\frac{n}{2}}
\frac{\varepsilon_j^{n+2}}{(\varepsilon_j^4+d(y,x_j)^4)^{1+\frac{n}{2}}}dV_{\theta_0}\\
&\leq C\frac{\varepsilon_j^2}{\delta^2}\left(\frac{\varepsilon_i^2\varepsilon_j^2}{\varepsilon_j^4+d(x_i,x_j)^4}\right)^{\frac{n}{2}}
\end{split}
\end{equation*}
because
\begin{equation*}
\begin{split}
\int_0^{\frac{\varepsilon_j+d(x_i,x_j)}{2}}\frac{r^{2n+1}dr}{(\varepsilon_j^4+r^4)^{\frac{n}{2}}}
\leq \int_0^{\frac{\varepsilon_j+d(x_i,x_j)}{2}}\frac{r^{2n+1}dr}{r^{4\cdot\frac{n}{2}}}
\leq (\varepsilon_j^4+d(x_i,x_j)^4)^{\frac{1}{2}}
\end{split}
\end{equation*}
and
\begin{equation*}
\begin{split}
\int_{\frac{\delta}{c_0}}^{\infty}\frac{r^{2n+1}dr}{(\varepsilon_j^4+r^4)^{1+\frac{n}{2}}}
\leq \int_{\frac{\delta}{c_0}}^{\infty}\frac{r^{2n+1}dr}{r^{4(1+\frac{n}{2})}}=\frac{C}{\delta^2}.
\end{split}
\end{equation*}
This proves the assertion.
\end{proof}

\section{}\label{AppendixB}

Here we prove  Theorem \ref{Theorem4.1}. We discuss both of the cases when $n=1$ or $M$ is spherical.
For convenience, we denote the CR conformal sub-Laplacian of a contact form $\theta$ by
$L_\theta$, i.e.
$$L_\theta=-(2+\frac{2}{n})\Delta_{\theta}+R_{\theta}.$$
Let $\{u_\nu\}$ be a sequence of positive functions satisfying (\ref{4.1}) and (\ref{4.2}).
Note that $u_\nu$ is uniformly bounded in $S_1^2(M)$. To see this,
it follows from integration by parts and H\"{o}lder's inequality that
\begin{equation*}
\begin{split}
&\int_M\left((2+\frac{2}{n})|\nabla_{\theta_0}u_\nu|^2_{\theta_0}+R_{\theta_0}u_\nu^2\right) dV_{\theta_0}\\
&\leq \left(\int_M|L_{\theta_0}u_\nu-r_\infty u_\nu^{1+\frac{2}{n}}|^{\frac{2n+2}{n+2}} dV_{\theta_0}
\right)^{\frac{n+2}{2n+2}}\left(\int_Mu_\nu^{2+\frac{2}{n}}\right)^{\frac{n}{2n+2}}
+r_\infty\int_Mu_\nu^{2+\frac{2}{n}} dV_{\theta_0},
\end{split}
\end{equation*}
which is uniformly bounded by (\ref{4.1}) and (\ref{4.2}).
Thus, by passing to a subsequence if necessary, we assume that
$u_\nu$ converges to $u_\infty$ weakly in $S_1^2(M)$.
Since the Folland-Stein embedding $S_1^2(M)\hookrightarrow L^s(M)$ is compact for $1<s<2+\displaystyle\frac{2}{n}$
(see Proposition 5.6 in \cite{Jerison&Lee3}),
$u_\nu$ converges to $u_\infty$ in $L^s(M)$ for $1<s<2+\displaystyle\frac{2}{n}$.

\begin{prop}
The function
$u_\infty$ is a smooth nonnegative function satisfying $(\ref{4.3})$.
\end{prop}
\begin{proof}
Note that for any $\varphi\in C^\infty(M)$
\begin{equation}\label{A1}
\int_M \varphi u_\nu^{1+\frac{2}{n}}dV_{\theta_0}\to\int_M \varphi u_\infty^{1+\frac{2}{n}}dV_{\theta_0}\hspace{2mm}\mbox{ as }\nu\to\infty.
\end{equation}
Indeed, there exists a constant $c$ such that
\begin{equation*}
\left|u_\nu^{1+\frac{2}{n}}-u_\infty^{1+\frac{2}{n}}\right|\leq c|u_\nu-u_\infty|\left(|u_\nu|^{\frac{2}{n}}+|u_\infty|^{\frac{2}{n}}\right)
\end{equation*}
Hence, it follows from H\"{o}lder's inequality that
\begin{equation}\label{A2}
\begin{split}
&\left|
\int_M \varphi u_\nu^{1+\frac{2}{n}}dV_{\theta_0}-\int_M \varphi u_\infty^{1+\frac{2}{n}}dV_{\theta_0}\right|\\
&\leq \int_M \left|u_\nu^{1+\frac{2}{n}}-u_\infty^{1+\frac{2}{n}}\right||\varphi|dV_{\theta_0}\\
&\leq c\|\varphi\|_{L^\infty(M)}\|u_\nu-u_\infty\|_{L^\beta(M)}
\left\||u_\nu|^{\frac{2}{n}}+|u_\infty|^{\frac{2}{n}}\right\|_{L^{\beta'}(M)}
\end{split}
\end{equation}
where $\beta$ are chosen
such that $1+\displaystyle\frac{2}{n}<\beta<2+\frac{2}{n}$
and $\displaystyle\frac{1}{\beta}+\frac{1}{\beta'}=1$.
Therefore,
$\displaystyle\frac{2}{n}\beta'<2+\frac{2}{n}$. Combining all these, we can conclude that
the right hand side of (\ref{A2}) tends to $0$ as $\nu\to\infty$. This proves (\ref{A1}).
On the other hand, for any $\varphi\in C^\infty(M)$, we have
\begin{equation}\label{A3}
\begin{split}
\int_M\varphi L_{\theta_0}u_\nu dV_{\theta_0}
=\int_M u_\nu L_{\theta_0}\varphi dV_{\theta_0}\to \int_M u_\infty L_{\theta_0}\varphi dV_{\theta_0}=\int_M\varphi L_{\theta_0}u_\infty dV_{\theta_0}
\end{split}
\end{equation}
as $\nu\to\infty$.
Therefore, it follows from (\ref{A1}) and (\ref{A3}) that for any $\varphi\in C^\infty(M)$
\begin{equation*}
\int_M\left(L_{\theta_0}u_\nu-r_\infty u_\nu^{1+\frac{2}{n}}\right)\varphi dV_{\theta_0}
\to \int_M\left(L_{\theta_0}u_\infty-r_\infty u_\infty^{1+\frac{2}{n}}\right)\varphi dV_{\theta_0}
\end{equation*}
as $\nu\to\infty$.
Combining this with (\ref{4.2}), we can conclude that
$u_\infty$ satisfies (\ref{4.3}).
Since $u_\nu$ is nonnegative, $u_\infty$ is also nonnegative.
It follows from  Theorem 3.22 in \cite{Dragomir}
that $u_\infty$ is smooth. This proves the assertion.
\end{proof}

\begin{prop}\label{propB.2.5}
There holds
\begin{equation}\label{A8}
u_\nu^{1+\frac{2}{n}}-u_\infty^{1+\frac{2}{n}}-|u_\nu-u_\infty|^{\frac{2}{n}}(u_\nu-u_\infty)
\to 0\hspace{2mm}\mbox{ in }L^{\frac{2n+2}{n+2}}(M)\mbox{ as }\nu\to\infty.
\end{equation}
\end{prop}
\begin{proof}
We use the following inequality:
\begin{equation}\label{A4}
\left||a+b|^{\frac{2}{n}}-|a|^{\frac{2}{n}}\right|\leq c\left(|b|^{\frac{2}{n}}+\max\{|a|,|b|\}^{\frac{2}{n}-1}|b|\right)\hspace{2mm}
\mbox{ for all }a,b.
\end{equation}
Applying (\ref{A4}), we get
\begin{equation}\label{A5}
\begin{split}
u_\nu^{1+\frac{2}{n}}&=u_\nu^{\frac{2}{n}}u_\infty+(u_\nu-u_\infty+u_\infty)^{\frac{2}{n}}(u_\nu-u_\infty)\\
&=u_\nu^{\frac{2}{n}}u_\infty+|u_\nu-u_\infty|^{\frac{2}{n}}(u_\nu-u_\infty)\\
&\hspace{4mm}+O\left(u_\infty^{\frac{2}{n}}|u_\nu-u_\infty|+\max\{|u_\nu-u_\infty|,|u_\infty|\}^{\frac{2}{n}-1}|u_\nu-u_\infty|u_\infty\right)\\
&=u_\nu^{\frac{2}{n}}u_\infty+|u_\nu-u_\infty|^{\frac{2}{n}}(u_\nu-u_\infty)
+O\left(u_\infty^{\frac{2}{n}}|u_\nu-u_\infty|+|u_\nu-u_\infty|^{\frac{2}{n}}u_\infty\right).
\end{split}
\end{equation}
Applying (\ref{A4}) again, we obtain
\begin{equation}\label{A6}
\begin{split}
u_\nu^{\frac{2}{n}}u_\infty
&=(u_\nu-u_\infty+u_\infty)^{\frac{2}{n}}u_\infty\\
&=u_\infty^{1+\frac{2}{n}}+O\left(u_\infty|u_\nu-u_\infty|^{\frac{2}{n}}+\max\{|u_\nu-u_\infty|,|u_\infty|\}^{\frac{2}{n}-1}|u_\nu-u_\infty|u_\infty\right)\\
&=u_\infty^{1+\frac{2}{n}}+O\left(u_\infty^{\frac{2}{n}}|u_\nu-u_\infty|+|u_\nu-u_\infty|^{\frac{2}{n}}u_\infty\right).
\end{split}
\end{equation}
Combining (\ref{A5}) and (\ref{A6}),
we get
\begin{equation}\label{A7}
u_\nu^{1+\frac{2}{n}}-u_\infty^{1+\frac{2}{n}}-|u_\nu-u_\infty|^{\frac{2}{n}}(u_\nu-u_\infty)
=O\left(u_\infty^{\frac{2}{n}}|u_\nu-u_\infty|+|u_\nu-u_\infty|^{\frac{2}{n}}u_\infty\right).
\end{equation}
Since $u_\infty$ is smooth and $u_\nu$ converges to $u_\infty$ in $L^s(M)$ for $1<s<2+\displaystyle\frac{2}{n}$,
(\ref{A8})  follows from (\ref{A7}).
\end{proof}

If $u_\nu$ converges to $u_\infty$ strongly in $S_1^2(M)$,
then Theorem \ref{Theorem4.1} follows.
Therefore, we assume that
$u_\nu$ converges to $u_\infty$ weakly in $S_1^2(M)$,
but not strongly in $S_1^2(M)$. On the other hand, it follows from (\ref{4.2}), (\ref{4.3}) and
(\ref{A8}) that
\begin{equation*}
\begin{split}
&\int_M\Big|L_{\theta_0}(u_\nu-u_\infty)-r_\infty |u_\nu-u_\infty|^{\frac{2}{n}}(u_\nu-u_\infty)\Big|^{\frac{2n+2}{n+2}}dV_{\theta_0}\\
&\leq C\int_M\Big|L_{\theta_0}u_\nu-r_\infty u_\nu^{1+\frac{2}{n}}\Big|^{\frac{2n+2}{n+2}}dV_{\theta_0}\\
&\hspace{4mm}+C\,
r_\infty\int_M\Big|u_\nu^{1+\frac{2}{n}}-u_\infty^{1+\frac{2}{n}}-|u_\nu-u_\infty|^{\frac{2}{n}}(u_\nu-u_\infty)\Big|^{\frac{2n+2}{n+2}}dV_{\theta_0}
\to 0
\end{split}
\end{equation*}
as $\nu\to \infty$.
That is to say, if we let $v_\nu:=u_\nu-u_\infty$,
then $\{v_\nu\}$ is a sequence of functions such that
$v_\nu$ converges to $0$ weakly in $S_1^2(M)$,
but not strongly in $S_1^2(M)$, and
satisfies
\begin{equation}\label{Assumption}
\int_M\Big|L_{\theta_0}v_\nu-r_\infty |v_\nu|^{\frac{2}{n}}v_\nu\Big|^{\frac{2n+2}{n+2}}dV_{\theta_0}
\to 0\hspace{2mm}\mbox{ as }\nu\to\infty.
\end{equation}

We are going to extract bubbles from $v_\nu$. To do this,
we mainly follow the proof of Proposition 8 in \cite{Gamara1}.
The argument is almost the same except with some small modifications.
However, there are some parts in the argument of \cite{Gamara1} which are not very precise.
So we are going to provide all the details of the proof.

As before, we denote
$$B_r(x)=\{y\in M:d(x,y)<r\},$$
where $d$ is the Carnot-Carath\'{e}odory distance on $M$
with respect to the contact form $\theta_0$.

\begin{lem}\label{lemA.3}
There exists $x\in M$ such that for every $\rho>0$,
there exists $\delta(\rho)>0$ such that
\begin{equation}\label{A9}
\liminf_{\nu\to\infty}\int_{B_\rho(x)}\left((2+\frac{2}{n})|\nabla_{\theta_0}v_\nu|_{\theta_0}^2+R_{\theta_0}v_\nu^2\right)dV_{\theta_0}
\geq \delta(\rho).
\end{equation}
\end{lem}
\begin{proof}
Suppose that for all $x\in M$, there is $\rho(x)>0$ such that
$$\liminf_{\nu\to\infty}\int_{B_{\rho(x)}(x)}\left((2+\frac{2}{n})|\nabla_{\theta_0}v_\nu|_{\theta_0}^2+R_{\theta_0}v_\nu^2\right)dV_{\theta_0}
=0.$$
Since $M$ is compact,
we can cover $M$ with finite number of $B_{\rho(x_i)}(x_i)$ with $i=1,..., L$.
As a result, there exists a subsequence of $\{v_\nu\}$, which is still denoted by $\{v_\nu\}$, such that
$$\int_{M}\left((2+\frac{2}{n})|\nabla_{\theta_0}v_\nu|_{\theta_0}^2+R_{\theta_0}v_\nu^2\right)dV_{\theta_0}
\to 0$$
as $\nu\to\infty$. This contradicts
to the assumption that $v_\nu$ does not converge to $0$ strongly in $S_1^2(M)$.
This proves the assertion.
\end{proof}

\begin{lem}\label{lemA.4}
The constant $\delta(\rho)$ in Lemma \ref{lemA.3} can be taken
to be $a_0 r_\infty^{-n}
Y(M,\theta_0)^{n+1}$, where $a_0$ is any positive real number strictly less than $1$.
\end{lem}

Before we give the proof of Lemma \ref{lemA.4}, we remark that  our constant in Lemma \ref{lemA.4}
is different from that of \cite{Gamara1}. But one can see from the following arguments that a uniform
constant will be sufficient for our purpose.

\begin{proof}[Proof of Lemma \ref{lemA.4}]
The proof is similar to the proof of Lemma 10 in \cite{Gamara1}.
Let $\psi_\nu:\mathbb{R}\to[0,1]$ be a $C^1$ cut-off function such that
$$\psi_\nu(s)=\left\{
                \begin{array}{ll}
                  1, & \hbox{ for $s\leq \rho_\nu$;} \\
                  0, & \hbox{ for $s\geq \rho_\nu+\delta_\nu$,}
                \end{array}
              \right.
$$
where $\rho<\rho_\nu<2\rho$ and $0<\delta_\nu<\rho$.
Fix a point $x\in M$ where (\ref{A9}) holds for $x$.
Define the function $\phi_\nu$ by
$\phi_\nu(y)=\psi_\nu(d(x,y))$ where
$d$ is the Carnot-Carath\'{e}odory distance on $M$
with respect to the contact form $\theta_0$.
Then we have
\begin{equation}\label{A10}
|\nabla_{\theta_0}\phi_\nu|_{\theta_0}\leq \frac{C}{\delta_\nu}.
\end{equation}

By H\"{o}lder's inequality, we have
\begin{equation}\label{A11}
\int_M\Big(L_{\theta_0}v_\nu-r_\infty |v_\nu|^{\frac{2}{n}}v_\nu\Big)v_\nu\phi_\nu dV_{\theta_0}
\leq
\|\alpha_\nu\|_{L^{\frac{2n+2}{n+2}}(M)}\left(\int_M |v_\nu\phi_\nu|^{2+\frac{2}{n}}dV_{\theta_0}\right)^{\frac{n}{2n+2}}
\end{equation}
where
$\alpha_\nu=L_{\theta_0}v_\nu-r_\infty |v_\nu|^{\frac{2}{n}}v_\nu.$
By integration by parts, the left hand side of (\ref{A11}) can be written as
\begin{equation}\label{A12}
\begin{split}
&\int_M\Big(L_{\theta_0}v_\nu-r_\infty |v_\nu|^{\frac{2}{n}}v_\nu\Big)v_\nu\phi_\nu dV_{\theta_0}\\
&=(2+\frac{2}{n})\int_M|\nabla_{\theta_0} v_\nu|_{\theta_0}^2\phi_\nu dV_{\theta_0}
+(2+\frac{2}{n})\int_Mv_\nu\langle\nabla_{\theta_0} v_\nu,\nabla_{\theta_0} \phi_\nu\rangle_{\theta_0} dV_{\theta_0}\\
&\hspace{4mm}
+\int_MR_{\theta_0}v_\nu^2\phi_\nu dV_{\theta_0}-r_\infty\int_M|v_\nu|^{2+\frac{2}{n}}\phi_\nu dV_{\theta_0}.
\end{split}
\end{equation}
It follows from (\ref{A11}) and (\ref{A12}) that
\begin{equation}\label{A13}
\begin{split}
&(2+\frac{2}{n})\int_M|\nabla_{\theta_0} v_\nu|_{\theta_0}^2\phi_\nu dV_{\theta_0}
+(2+\frac{2}{n})\int_Mv_\nu\langle\nabla_{\theta_0} v_\nu,\nabla_{\theta_0} \phi_\nu\rangle_{\theta_0} dV_{\theta_0}
+\int_MR_{\theta_0}v_\nu^2\phi_\nu dV_{\theta_0}\\
&\hspace{4mm}
\leq r_\infty\int_M|v_\nu|^{2+\frac{2}{n}}\phi_\nu dV_{\theta_0}
+\|\alpha_\nu\|_{L^{\frac{2n+2}{n+2}}(M)}\left(\int_M |v_\nu\phi_\nu|^{2+\frac{2}{n}}dV_{\theta_0}\right)^{\frac{n}{2n+2}}.
\end{split}
\end{equation}
Also, by Folland-Stein embedding theorem in Proposition \ref{Theorem 2.1}, we have
\begin{equation}\label{A14}
\left(\int_M |v_\nu\phi_\nu|^{2+\frac{2}{n}}dV_{\theta_0}\right)^{\frac{n}{2n+2}}
\leq C\left(\int_M\left(|\nabla_{\theta_0}(v_\nu\phi_\nu)|^2_{\theta_0}+R_{\theta_0}(v_\nu\phi_\nu)^2\right)dV_{\theta_0}
\right)^{\frac{1}{2}}.
\end{equation}
By Cauchy-Schwarz inequality, we have
\begin{equation}\label{A15}
\left(\int_M|\nabla_{\theta_0}(v_\nu\phi_\nu)|^2_{\theta_0}dV_{\theta_0}\right)^{\frac{1}{2}}
\leq \left(\int_M|\nabla_{\theta_0}v_\nu|^2_{\theta_0}\phi_\nu^2dV_{\theta_0}\right)^{\frac{1}{2}}
+\left(\int_M|\nabla_{\theta_0}\phi_\nu|^2_{\theta_0}v_\nu^2dV_{\theta_0}\right)^{\frac{1}{2}}.
\end{equation}
Substituting (\ref{A14}) and (\ref{A15}) into (\ref{A13}), we get
\begin{equation}\label{A17}
\begin{split}
&(2+\frac{2}{n})\int_M|\nabla_{\theta_0} v_\nu|_{\theta_0}^2\phi_\nu dV_{\theta_0}
+(2+\frac{2}{n})\int_Mv_\nu\langle\nabla_{\theta_0} v_\nu,\nabla_{\theta_0} \phi_\nu\rangle_{\theta_0} dV_{\theta_0}
+\int_MR_{\theta_0}v_\nu^2\phi_\nu dV_{\theta_0}\\
&\hspace{4mm}
\leq r_\infty\int_M|v_\nu|^{2+\frac{2}{n}}\phi_\nu dV_{\theta_0}
+C\,\|\alpha_\nu\|_{L^{\frac{2n+2}{n+2}}(M)}\left[
\left(\int_M|\nabla_{\theta_0}v_\nu|^2_{\theta_0}\phi_\nu^2dV_{\theta_0}\right)^{\frac{1}{2}}\right.\\
&\hspace{8mm}\left.+\left(\int_M|\nabla_{\theta_0}\phi_\nu|^2_{\theta_0}v_\nu^2dV_{\theta_0}\right)^{\frac{1}{2}}
+\left(\int_MR_{\theta_0}(v_\nu\phi_\nu)^2dV_{\theta_0}\right)^{\frac{1}{2}}\right].
\end{split}
\end{equation}
By passing to a subsequence, we obtain from (\ref{A9}) that
\begin{equation}\label{A18}
\int_{M}\left((2+\frac{2}{n})|\nabla_{\theta_0}v_\nu|_{\theta_0}^2\phi_\nu+R_{\theta_0}v_\nu^2\phi_\nu\right)dV_{\theta_0}
\geq \frac{\delta(\rho)}{2}.
\end{equation}
Since $\{v_\nu\}$ is uniformly bounded in $S_1^2(M)$, it follows from (\ref{Assumption}) that
\begin{equation}\label{A19}
\|\alpha_\nu\|_{L^{\frac{2n+2}{n+2}}(M)}\left[\left(\int_M|\nabla_{\theta_0}v_\nu|^2_{\theta_0}\phi_\nu^2dV_{\theta_0}\right)^{\frac{1}{2}}
+\left(\int_MR_{\theta_0}(v_\nu\phi_\nu)^2dV_{\theta_0}\right)^{\frac{1}{2}}\right]
\to 0
\end{equation}
as $\nu\to\infty$. By (\ref{A18}) and (\ref{A19}),
we can rewrite (\ref{A17}) as follows:
\begin{equation}\label{A20}
\begin{split}
&(1+o(1))\int_{M}\left((2+\frac{2}{n})|\nabla_{\theta_0}v_\nu|_{\theta_0}^2\phi_\nu+R_{\theta_0}v_\nu^2\phi_\nu\right)dV_{\theta_0}\\
&+(2+\frac{2}{n})\int_Mv_\nu\langle\nabla_{\theta_0} v_\nu,\nabla_{\theta_0} \phi_\nu\rangle_{\theta_0} dV_{\theta_0}\\
&
\leq r_\infty\int_M|v_\nu|^{2+\frac{2}{n}}\phi_\nu dV_{\theta_0}
+C\,\|\alpha_\nu\|_{L^{\frac{2n+2}{n+2}}(M)}\left(\int_M|\nabla_{\theta_0}\phi_\nu|^2_{\theta_0}v_\nu^2dV_{\theta_0}\right)^{\frac{1}{2}}.
\end{split}
\end{equation}
Note that
\begin{equation}\label{A21}
\int_M|\nabla_{\theta_0}\phi_\nu|^2_{\theta_0}v_\nu^2dV_{\theta_0}
\leq \frac{C}{\delta_\nu^2}\int_Mv_\nu^2dV_{\theta_0}
\end{equation}
by (\ref{A10}).
Hence,
\begin{equation}\label{A22}
\begin{split}
&\left|\int_Mv_\nu\langle\nabla_{\theta_0} v_\nu,\nabla_{\theta_0} \phi_\nu\rangle_{\theta_0} dV_{\theta_0}\right|\\
&\leq \left(\int_M|\nabla_{\theta_0}v_\nu|^2_{\theta_0}dV_{\theta_0}\right)^{\frac{1}{2}}
 \left(\int_M|\nabla_{\theta_0}\phi_\nu|^2_{\theta_0}v_\nu^2dV_{\theta_0}\right)^{\frac{1}{2}}
\leq \frac{C}{\delta_\nu}\left(\int_Mv_\nu^2dV_{\theta_0}\right)^{\frac{1}{2}}
\end{split}
\end{equation}
by H\"{o}lder's inequality and
the fact that $\{v_\nu\}$ is uniformly bounded in $S_1^2(M)$.

If we choose $\delta_\nu=\displaystyle\left(\int_M v_\nu^2 dV_{\theta_0}\right)^{\frac{1}{4}}$, then
\begin{equation}\label{A23}
\frac{C}{\delta_\nu}\left(\int_Mv_\nu^2dV_{\theta_0}\right)^{\frac{1}{2}}
\to 0\hspace{2mm}\mbox{ as }\nu\to\infty,
\end{equation}
since $v_\nu\to 0$ in $L^2(M)$ as $\nu\to \infty$.
Using (\ref{A21})-(\ref{A23}), we can rewrite (\ref{A20})
as
\begin{equation}\label{A24}
(1+o(1))\int_{M}\left((2+\frac{2}{n})|\nabla_{\theta_0}v_\nu|_{\theta_0}^2\phi_\nu+R_{\theta_0}v_\nu^2\phi_\nu\right)dV_{\theta_0}
\leq r_\infty\int_M|v_\nu|^{2+\frac{2}{n}}\phi_\nu dV_{\theta_0}+o(1).
\end{equation}
In view of (\ref{A18}), we deduce from (\ref{A24}) that
\begin{equation}\label{A25}
(1+o(1))\int_{M}\left((2+\frac{2}{n})|\nabla_{\theta_0}v_\nu|_{\theta_0}^2\phi_\nu+R_{\theta_0}v_\nu^2\phi_\nu\right)dV_{\theta_0}
\leq r_\infty\int_M|v_\nu|^{2+\frac{2}{n}}\phi_\nu dV_{\theta_0}.
\end{equation}

We introduce $\gamma_\nu>0$ which will be chosen latter. If $\gamma_\nu$ is sufficiently small,
it follows from (\ref{A25}) that
\begin{equation}\label{A27}
(1+o(1))\int_{M}\left((2+\frac{2}{n})|\nabla_{\theta_0}v_\nu|_{\theta_0}^2\phi_\nu+R_{\theta_0}v_\nu^2\phi_\nu\right)dV_{\theta_0}
\leq r_\infty\int_{M}|v_\nu|^{2+\frac{2}{n}}(\phi_\nu+\gamma_\nu) dV_{\theta_0}.
\end{equation}
By the definition of the CR Yamabe constant in (\ref{12}), we have
$$\int_{M}
\left((2+\frac{2}{n})|\nabla_{\theta_0}w|_{\theta_0}^2
+R_{\theta_0}w^2\right)dV_{\theta_0}\geq
Y(M,\theta_0)\|w\|_{L^{2+\frac{2}{n}}(M)}^2.$$
Applying this to the function $w=|v_\nu|(\phi_\nu+\gamma_\nu)^{\frac{n}{2n+2}}$, we obtain
\begin{equation}\label{A28}
\begin{split}
&Y(M,\theta_0)
\left(\int_{M}|v_\nu|^{2+\frac{2}{n}}(\phi_\nu+\gamma_\nu) dV_{\theta_0}\right)^{\frac{n}{n+1}}\\
&\leq(2+\frac{2}{n})\int_{M}
\Big(|\nabla_{\theta_0}v_\nu|_{\theta_0}^2(\phi_\nu+\gamma_\nu)^{\frac{n}{n+1}}
+v_\nu^2\big|\nabla_{\theta_0}\big((\phi_\nu+\gamma_\nu)^{\frac{n}{2n+2}}\big)\big|_{\theta_0}^2\\
&\hspace{8mm}+2|v_\nu|(\phi_\nu+\gamma_\nu)^{\frac{n}{2n+2}}\langle\nabla_{\theta_0}v_\nu,\nabla_{\theta_0}\big((\phi_\nu+\gamma_\nu)^{\frac{n}{2n+2}}\big)\rangle_{\theta_0}
\Big)dV_{\theta_0}\\
&\hspace{4mm}+\int_{M}R_{\theta_0}v_\nu^2(\phi_\nu+\gamma_\nu)^{\frac{n}{n+1}}dV_{\theta_0}.
\end{split}
\end{equation}
By (\ref{A21}), (\ref{A23}) and the fact that $\displaystyle\frac{n}{n+1}-2<0$, we can
choose $\gamma_k$ sufficiently small such that
\begin{equation*}
\int_{M}v_\nu^2\gamma_\nu^{\frac{n}{n+1}-2}|\nabla_{\theta_0}\phi_\nu|_{\theta_0}^2=o(1),
\end{equation*}
which implies that
\begin{equation}\label{A29}
\begin{split}
\int_{M}
v_\nu^2\big|\nabla_{\theta_0}\big((\phi_\nu+\gamma_\nu)^{\frac{n}{2n+2}}\big)\big|_{\theta_0}^2dV_{\theta_0}
&=\int_{M}v_\nu^2(\phi_\nu+\gamma_\nu)^{\frac{n}{n+1}-2}|\nabla_{\theta_0}\phi_\nu|_{\theta_0}^2dV_{\theta_0}\\
&\leq \int_{M}v_\nu^2\gamma_\nu^{\frac{n}{n+1}-2}|\nabla_{\theta_0}\phi_\nu|_{\theta_0}^2dV_{\theta_0}=o(1).
\end{split}
\end{equation}
It follows from (\ref{A29}),  H\"{o}lder's inequality
and the fact that $\{v_\nu\}$ is uniformly bounded in $S_1^2(M)$ that
\begin{equation}\label{A30}
\begin{split}
&\left|\int_{M}
|v_\nu|(\phi_\nu+\gamma_\nu)^{\frac{n}{2n+2}}\langle\nabla_{\theta_0}v_\nu,\nabla_{\theta_0}\big((\phi_\nu+\gamma_\nu)^{\frac{n}{2n+2}}\big)\rangle_{\theta_0}dV_{\theta_0}\right|\\
&\leq\left(\int_{M}
|\nabla_{\theta_0}v_\nu|_{\theta_0}^2(\phi_\nu+\gamma_\nu)^{\frac{n}{n+1}}dV_{\theta_0}\right)^{\frac{1}{2}}\left(\int_{M}
v_\nu^2\big|\nabla_{\theta_0}\big((\phi_\nu+\gamma_\nu)^{\frac{n}{2n+2}}\big)\big|_{\theta_0}^2dV_{\theta_0}\right)^{\frac{1}{2}}=o(1).
\end{split}
\end{equation}
Putting (\ref{A29}) and (\ref{A30}) into (\ref{A28}), we get
\begin{equation*}
\begin{split}
&Y(M,\theta_0)
\left(\int_{M}|v_\nu|^{2+\frac{2}{n}}(\phi_\nu+\gamma_\nu) dV_{\theta_0}\right)^{\frac{n}{n+1}}\\
&\leq\int_{M}
\left((2+\frac{2}{n})|\nabla_{\theta_0}v_\nu|_{\theta_0}^2(\phi_\nu+\gamma_\nu)^{\frac{n}{n+1}}+
R_{\theta_0}v_\nu^2(\phi_\nu+\gamma_\nu)^{\frac{n}{n+1}}\right)dV_{\theta_0}+o(1).
\end{split}
\end{equation*}
Combining this with (\ref{A27}) and using (\ref{A18}), we obtain
\begin{equation}\label{A31}
\begin{split}
&(1+o(1))\int_{M}\left((2+\frac{2}{n})|\nabla_{\theta_0}v_\nu|_{\theta_0}^2\phi_\nu+R_{\theta_0}v_\nu^2\phi_\nu\right)dV_{\theta_0}\\
&\leq r_\infty
Y(M,\theta_0)^{-\frac{n+1}{n}}
\left(\int_{M}
\left((2+\frac{2}{n})|\nabla_{\theta_0}v_\nu|_{\theta_0}^2\phi_\nu^{\frac{n}{n+1}}+
R_{\theta_0}v_\nu^2\phi_\nu^{\frac{n}{n+1}}\right)dV_{\theta_0}\right)^{\frac{n+1}{n}}.
\end{split}
\end{equation}
It follows from (\ref{A31}) and the definition of $\phi_\nu$ that
\begin{equation}\label{A32}
\begin{split}
&(1+o(1))\int_{B_{\rho_\nu}(x)}\left((2+\frac{2}{n})|\nabla_{\theta_0}v_\nu|_{\theta_0}^2+R_{\theta_0}v_\nu^2\right)dV_{\theta_0}\\
&\leq r_\infty
Y(M,\theta_0)^{-\frac{n+1}{n}}
\left(\int_{B_{\rho_\nu}(x)}
\left((2+\frac{2}{n})|\nabla_{\theta_0}v_\nu|_{\theta_0}^2+
R_{\theta_0}v_\nu^2\right)dV_{\theta_0}\right)^{\frac{n+1}{n}}\\
&\hspace{4mm}
+r_\infty Y(M,\theta_0)^{-\frac{n+1}{n}}\left(\int_{B_{\rho_\nu+\delta_\nu}(x)-B_{\rho_\nu}(x)}
\left((2+\frac{2}{n})|\nabla_{\theta_0}v_\nu|_{\theta_0}^2+
R_{\theta_0}v_\nu^2\right)dV_{\theta_0}\right)^{\frac{n+1}{n}}\\
&:=r_\infty Y(M,\theta_0)^{-\frac{n+1}{n}}(I+II).
\end{split}
\end{equation}

It follows from (\ref{A18}) that $I\geq C\delta(\rho)^{\frac{n+1}{n}}$.
We have the following two cases:\\
Case (i). If $II=o(I)$, then it follows from (\ref{A32}) that
\begin{equation*}
\begin{split}
&(1+o(1))\int_{B_{\rho_\nu}(x)}\left((2+\frac{2}{n})|\nabla_{\theta_0}v_\nu|_{\theta_0}^2+R_{\theta_0}v_\nu^2\right)dV_{\theta_0}\\
&\leq r_\infty
Y(M,\theta_0)^{-\frac{n+1}{n}}
\left(\int_{B_{\rho_\nu}(x)}
\left((2+\frac{2}{n})|\nabla_{\theta_0}v_\nu|_{\theta_0}^2+
R_{\theta_0}v_\nu^2\right)dV_{\theta_0}\right)^{\frac{n+1}{n}},
\end{split}
\end{equation*}
which implies that
\begin{equation*}
\int_{B_{\rho_\nu}(x)}\left((2+\frac{2}{n})|\nabla_{\theta_0}v_\nu|_{\theta_0}^2+R_{\theta_0}v_\nu^2\right)dV_{\theta_0}\\
\geq a_0r_\infty^{-n}
Y(M,\theta_0)^{n+1}
\end{equation*}
where $a_0$ is any positive real number strictly less than $1$.\\
Case (ii).
Suppose that there exists a fixed constant $C$ such that
$II\geq C\delta(\rho)^{\frac{n+1}{n}}$
for all choices of $\rho_\nu\in [\rho, 2\rho]$ with
$\delta_\nu=\displaystyle\left(\int_M v_\nu^2 dV_{\theta_0}\right)^{\frac{1}{4}}$.
In fact, it can not occur since $\delta_\nu\rightarrow 0$ as $\nu\rightarrow 0$, and $\int_M\left((2+\frac{2}{n})|\nabla_{\theta_0}v_\nu|_{\theta_0}^2+R_{\theta_0}v_\nu^2\right)dV_{\theta_0}$ is uniformly bounded.
This proves Lemma \ref{lemA.4}.
\end{proof}

It follows from Lemma \ref{lemA.4} that for any $x\in M$ satisfying (\ref{A9})
and for any given $a_0<1$, and $\nu$ sufficiently large, there exists a  $\rho_\nu(x)$ such that
\begin{equation*}
\int_{B_{\rho_\nu(x)}(x)}\left((2+\frac{2}{n})|\nabla_{\theta_0}v_\nu|_{\theta_0}^2+R_{\theta_0}v_\nu^2\right)dV_{\theta_0}
= a_0 r_\infty^{-n}
Y(M,\theta_0)^{n+1}.
\end{equation*}
Which means for and $\rho>\rho_\nu(x)$, we have
\begin{equation*}
\int_{B_{\rho_\nu(x)}(x)}\left((2+\frac{2}{n})|\nabla_{\theta_0}v_\nu|_{\theta_0}^2+R_{\theta_0}v_\nu^2\right)dV_{\theta_0}
>a_0 r_\infty^{-n}
Y(M,\theta_0)^{n+1}.
\end{equation*}
Then for every $\nu$ sufficiently large, we define:
\begin{equation}\label{A36}
\rho_{1,\nu}=\inf\rho_\nu(x),
\end{equation}
where the infimum is taken among all $x\in M$  satisfying (\ref{A9}), which is a closed set. Thus the infimum is attained. That is, there exists
$x^*_{1,\nu}\in M$ such that
\begin{equation}\label{A36.5}
\rho_{1,\nu}=\rho_\nu(x^*_{1,\nu}).
\end{equation}
For any $\rho_0>0$, by the proof of Lemma \ref{lemA.4},
we have
\begin{equation*}
\int_{B_{\widetilde{\rho}}(x)}\left((2+\frac{2}{n})|\nabla_{\theta_0}v_\nu|_{\theta_0}^2+R_{\theta_0}v_\nu^2\right)dV_{\theta_0}
\geq a_0 r_\infty^{-n}
Y(M,\theta_0)^{n+1}
\end{equation*}
for all $\nu$ sufficiently large. Since $\rho_0$ is arbitrary, and $\rho_\nu(x)\leq \rho_0$, then we have the following lemma:
\begin{lem}\label{lemA.5}
The sequence $\{\rho_{1,\nu}\}$ converges to zero as $\nu\to\infty$.
\end{lem}

As explained in Appendix \ref{Appendix},
 we can find $\rho>0$ which is independent of $x$ (since $M$ is compact)
 with $\widehat{\theta}_x=\varphi_x^{\frac{2}{n}}\theta_0$
in a neighborhood $B_{3\delta}(x)$ of $x$ such that
(\ref{B.0}) and (\ref{B.1}) are satisfied when $n=1$
and (\ref{B.0a}) is satisfied when $M$ is spherical.
We define a sequence of functions $\{ \widetilde{v}_\nu\}$
in $B_{2\delta}(x)$ as follows:
$$\widetilde{v}_\nu=\varphi_x^{-1}v_\nu.$$
Since $\widehat{\theta}_x=\varphi_x^{\frac{2}{n}}\theta_0$, by the CR transformation law, we have
\begin{equation}\label{A37}
L_{\theta_0}v_\nu
=\varphi_x^{1+\frac{2}{n}}L_{\widehat{\theta}_x}\widetilde{v}_\nu,
\end{equation}
which holds in $B_{2\delta}(x)\subset M$.
It follows from (\ref{B.5.5}) that (\ref{A37}) holds
in the CR normal coordinates $\{(z,t):(t^2+|z|^4)^{\frac{1}{4}}<\widehat{\rho}\}\subset\mathbb{H}^n$, where $\widehat{\rho}>0$ is independent of $x$.
By (\ref{Assumption}) and (\ref{A37}), we have
\begin{equation}\label{A38}
\int_{\{(t^2+|z|^4)^{\frac{1}{4}}<\widehat{\rho}\}}\Big|L_{\widehat{\theta}_x}\widetilde{v}_\nu-r_\infty |\widetilde{v}_\nu|^{\frac{2}{n}}\widetilde{v}_\nu\Big|^{\frac{2n+2}{n+2}}dV_{\widehat{\theta}_x}
\to 0\hspace{2mm}\mbox{ as }\nu\to\infty.
\end{equation}
By the properties of $v_\nu$, we know that  $\widetilde{v}_\nu$
 is bounded in $L^{2+\frac{2}{n}}$
and $\widetilde{v}_\nu\to 0$ in $L^s$
for all $s<2+\displaystyle\frac{2}{n}$ as $\nu\to\infty$.

When $n=1$, it follows from (\ref{B.0}), (\ref{B.1.5}), (\ref{B.1})  and (\ref{A38})  that
\begin{equation}\label{A40}
\begin{split}
&\int_{\{(t^2+|z|^4)^{\frac{1}{4}}<\widehat{\rho}\}}\left|L_{\theta_{\mathbb{H}^n}}\widetilde{v}_\nu-r_\infty |\widetilde{v}_\nu|^{\frac{2}{n}}\widetilde{v}_\nu\right|^{\frac{2n+2}{n+2}}dV_{\theta_{\mathbb{H}^n}}\\
&\leq C\int_{\{(t^2+|z|^4)^{\frac{1}{4}}<\widehat{\rho}\}}\left|L_{\theta_{\mathbb{H}^n}}
\widetilde{v}_\nu-L_{\widehat{\theta}_x}\widetilde{v}_\nu\right|^{\frac{2n+2}{n+2}}dV_{\theta_{\mathbb{H}^n}}\\
&\hspace{4mm}
+C\int_{\{(t^2+|z|^4)^{\frac{1}{4}}<\widehat{\rho}\}}\left|L_{\widehat{\theta}_x}\widetilde{v}_\nu-r_\infty \widetilde{v}_\nu^{1+\frac{2}{n}}\right|^{\frac{2n+2}{n+2}}\Big|dV_{\theta_{\mathbb{H}^n}}-dV_{\widehat{\theta}_x}\Big|\\
&\hspace{4mm}
+C\int_{\{(t^2+|z|^4)^{\frac{1}{4}}<\widehat{\rho}\}}\left|L_{\widehat{\theta}_x}\widetilde{v}_\nu-r_\infty |\widetilde{v}_\nu|^{\frac{2}{n}}\widetilde{v}_\nu\right|^{\frac{2n+2}{n+2}}dV_{\widehat{\theta}_x}
\to 0
\end{split}
\end{equation}
as $\nu\to\infty$.
When $M$ is spherical, then it follows from
(\ref{B.0a}) and (\ref{A38}) that
\begin{equation}\label{A41}
\int_{\{(t^2+|z|^4)^{\frac{1}{4}}<\widehat{\rho}\}}\Big|L_{\theta_{\mathbb{H}^n}}\widetilde{v}_\nu-r_\infty |\widetilde{v}_\nu|^{\frac{2}{n}}\widetilde{v}_\nu\Big|^{\frac{2n+2}{n+2}}dV_{\theta_{\mathbb{H}^n}}
\to 0\hspace{2mm}\mbox{ as }\nu\to\infty,
\end{equation}
since $R_{\theta_{\mathbb{H}^n}}=0$.

Let $\widehat{\chi}$ be a cut off function
such that
$\widehat{\chi}(s)=1$ if $0\leq s\leq \displaystyle\frac{\widehat{\rho}}{2}$
and $0$  if $s\geq \widehat{\rho}$.
Let $\{\widetilde{V}_\nu\}$ be a sequence of functions in $\mathbb{H}^n$ defined by
\begin{equation}\label{A44}
\widetilde{V}_\nu(z,t)=
\left\{
  \begin{array}{ll}
    (\rho_{1,\nu})^n\,\widehat{\chi}\Big(\rho_{1,\nu}(t^2+|z|^4)^{\frac{1}{4}}\Big)\,\widetilde{v}_\nu\big(\rho_{1,\nu}z,(\rho_{1,\nu})^2t\big), & \hbox{ for $(t^2+|z|^4)^{\frac{1}{4}}<\displaystyle\frac{\widehat{\rho}}{\rho_{1,\nu}}$;} \\
    0, & \hbox{ otherwise,}
  \end{array}
\right.
\end{equation}
where $\rho_{1,\nu}$ is defined as in (\ref{A36}).
Then we have the following:
\begin{prop}\label{propA.6}
(i) For any fixed ball $B$ of $\mathbb{H}^n$, we have
$$\int_{B}\left|L_{\theta_{\mathbb{H}^n}}\widetilde{V}_\nu-r_\infty |\widetilde{V}_\nu|^{\frac{2}{n}}\widetilde{V}_\nu\right|^{\frac{2n+2}{n+2}}dV_{\theta_{\mathbb{H}^n}}
\to 0\hspace{2mm}\mbox{ as }\nu\to\infty.$$
(ii) There exists a constant $C$ such that
$$\int_{\mathbb{H}^n}\left(|\nabla_{\theta_{\mathbb{H}^n}}\widetilde{V}_\nu|^2_{\theta_{\mathbb{H}^n}}+ |\widetilde{V}_\nu|^{2+\frac{2}{n}}\right)dV_{\theta_{\mathbb{H}^n}}\leq C\hspace{2mm}
\mbox{
for all }\nu.$$
\end{prop}
\begin{proof}
To prove (i), we fix $\rho>0$. Since $\rho_{1,\nu}\to 0$ as $\nu\to\infty$ by Lemma \ref{lemA.5},
there exists $N$ such that
$\displaystyle\rho\leq \frac{\widehat{\rho}}{2\rho_{1,\nu}}$
when $\nu\geq N$.
Hence,
if $\nu\geq N$,  we have
\begin{equation*}
\begin{split}
&\int_{\{(t^2+|z|^4)^{\frac{1}{4}}<\rho\}}\left|L_{\theta_{\mathbb{H}^n}}\widetilde{V}_\nu(z,t)-r_\infty |\widetilde{V}_\nu(z,t)|^{\frac{2}{n}}\widetilde{V}_\nu(z,t)\right|^{\frac{2n+2}{n+2}}dV_{\theta_{\mathbb{H}^n}}\\
&=\int_{\{(t^2+|z|^4)^{\frac{1}{4}}<\rho\}}\left|L_{\theta_{\mathbb{H}^n}}
\Big((\rho_{1,\nu})^n\,\widetilde{v}_\nu\big(\rho_{1,\nu}z,(\rho_{1,\nu})^2t\big)\Big)\right.\\
&\hspace{20mm}\left.-r_\infty \Big|(\rho_{1,\nu})^n\,\widetilde{v}_\nu\big(\rho_{1,\nu}z,(\rho_{1,\nu})^2t\big)\Big|^{\frac{2}{n}}
(\rho_{1,\nu})^n\,\widetilde{v}_\nu\big(\rho_{1,\nu}z,(\rho_{1,\nu})^2t\big)\right|^{\frac{2n+2}{n+2}}dV_{\theta_{\mathbb{H}^n}}\\
&=\int_{\{(\tilde{t}^2+|\tilde{z}|^4)^{\frac{1}{4}}<\rho\rho_{1,\nu}\}}
\left|L_{\theta_{\mathbb{H}^n}}
\widetilde{v}_\nu(\tilde{z},\tilde{t})
-r_\infty |\widetilde{v}_\nu(\tilde{z},\tilde{t})|^{\frac{2}{n}}\widetilde{v}_\nu(\tilde{z},\tilde{t})\right|^{\frac{2n+2}{n+2}}dV_{\theta_{\mathbb{H}^n}}
=o(1),
\end{split}
\end{equation*}
where the first equality follows from (\ref{A44})
and the fact that $\displaystyle\rho\leq \frac{\widehat{\rho}}{2\rho_{1,\nu}}$,
the second equality follows from the change of variables
$(\tilde{z},\tilde{t})=\big(\rho_{1,\nu}z,(\rho_{1,\nu})^2t\big)$,
and the last equality follows from (\ref{A40}), (\ref{A41})
and the fact that $\displaystyle\rho\,\rho_{1,\nu}\leq \frac{\widehat{\rho}}{2}$.
This proves (i).
For (ii),
we have
\begin{equation*}
\begin{split}
&\int_{\mathbb{H}^n}\left((2+\frac{2}{n})|\nabla_{\theta_{\mathbb{H}^n}}\widetilde{V}_\nu|^2_{\theta_{\mathbb{H}^n}}+
r_\infty |\widetilde{V}_\nu|^{2+\frac{2}{n}}\right)dV_{\theta_{\mathbb{H}^n}}\\
&=\int_{\{(t^2+|z|^4)^{\frac{1}{4}}<\frac{\widehat{\rho}}{\rho_{1,\nu}}\}}\left((2+\frac{2}{n})|\nabla_{\theta_{\mathbb{H}^n}}\widetilde{V}_\nu|^2_{\theta_{\mathbb{H}^n}}+
r_\infty |\widetilde{V}_\nu|^{2+\frac{2}{n}}\right)dV_{\theta_{\mathbb{H}^n}}\\
&=\int_{\{(\tilde{t}^2+|\tilde{z}|^4)^{\frac{1}{4}}<\widehat{\rho}\}}\left((2+\frac{2}{n})
\left|\nabla_{\theta_{\mathbb{H}^n}}\Big(\widehat{\chi}\big((\tilde{t}^2+|\tilde{z}|^4)^{\frac{1}{4}}\big)
\widetilde{v}_\nu(\tilde{z},\tilde{t})\Big)\right|^2_{\theta_{\mathbb{H}^n}}\right.\\
&\hspace{20mm}\left.+
r_\infty \Big|\widehat{\chi}\big((\tilde{t}^2+|\tilde{z}|^4)^{\frac{1}{4}}\big)
\widetilde{v}_\nu(\tilde{z},\tilde{t})\Big|^{2+\frac{2}{n}}\right)dV_{\theta_{\mathbb{H}^n}}\\
&\leq
\int_{\{(\tilde{t}^2+|\tilde{z}|^4)^{\frac{1}{4}}<\widehat{\rho}\}}
(2+\frac{2}{n})
\widetilde{v}_\nu(\tilde{z},\tilde{t})^2
\left|\nabla_{\theta_{\mathbb{H}^n}}\widehat{\chi}\big((\tilde{t}^2+|\tilde{z}|^4)^{\frac{1}{4}}\big)
\right|^2_{\theta_{\mathbb{H}^n}}dV_{\theta_{\mathbb{H}^n}}\\
&\hspace{4mm}+
\int_{\{(\tilde{t}^2+|\tilde{z}|^4)^{\frac{1}{4}}<\widehat{\rho}\}}\left((2+\frac{2}{n})
\Big(\widehat{\chi}\big((\tilde{t}^2+|\tilde{z}|^4)^{\frac{1}{4}}\big)\Big)^2
|\nabla_{\theta_{\mathbb{H}^n}}
\widetilde{v}_\nu(\tilde{z},\tilde{t})|^2_{\theta_{\mathbb{H}^n}}\right.\\
&\hspace{20mm}\left.+
r_\infty \Big|\widehat{\chi}\big((\tilde{t}^2+|\tilde{z}|^4)^{\frac{1}{4}}\big)
\widetilde{v}_\nu(\tilde{z},\tilde{t})\Big|^{2+\frac{2}{n}}\right)dV_{\theta_{\mathbb{H}^n}}\\
&\leq \int_{\{(\tilde{t}^2+|\tilde{z}|^4)^{\frac{1}{4}}<\widehat{\rho}\}}
(2+\frac{2}{n})
\widetilde{v}_\nu(\tilde{z},\tilde{t})^2dV_{\theta_{\mathbb{H}^n}}\\
&\hspace{4mm}+
\int_{\{(\tilde{t}^2+|\tilde{z}|^4)^{\frac{1}{4}}<\widehat{\rho}\}}\left((2+\frac{2}{n})
|\nabla_{\theta_{\mathbb{H}^n}}
\widetilde{v}_\nu(\tilde{z},\tilde{t})|^2_{\theta_{\mathbb{H}^n}}+
r_\infty |\widetilde{v}_\nu(\tilde{z},\tilde{t})|^{2+\frac{2}{n}}\right)dV_{\theta_{\mathbb{H}^n}}\\
&\leq C,
\end{split}
\end{equation*}
where the first equality follows from (\ref{A44}),
the second equality follows from the change of variables
$(\tilde{z},\tilde{t})=\big(\rho_{1,\nu}z,(\rho_{1,\nu})^2t\big)$,
the third inequality follows from the property of the cut-off function $\widehat{\chi}$,
and the last inequality follows from the fact that $\widetilde{v}_\nu$ is uniformly bounded in $S_1^2(\mathbb{H}^n)$.
This proves (ii).
\end{proof}

It follows from Proposition \ref{propA.6}(ii)
that, by passing to subsequence if necessary, $\widetilde{V}_\nu$ converges to $\widetilde{V}$ weakly in $S_1^2(B)$ as $\nu\to\infty$
on each ball $B$ of $\mathbb{H}^n$.
Since the Folland-Stein embedding $S_1^2(B)\hookrightarrow L^s(B)$ is compact for $1<s<2+\displaystyle\frac{2}{n}$
on each ball $B$ of $\mathbb{H}^n$,
$\widetilde{V}_\nu$ converges to $\widetilde{V}$  in $L^s(B)$ for $1<s<2+\displaystyle\frac{2}{n}$.
On the other hand, it follows from Proposition \ref{propA.6}(i)
that $\widetilde{V}$ satisfies
\begin{equation}\label{A45}
(2+\frac{2}{n})\Delta_{\theta_{\mathbb{H}^n}}\widetilde{V}=r_\infty |\widetilde{V}|^{\frac{2}{n}}\widetilde{V}\hspace{2mm}\mbox{ in }\mathbb{H}^n.
\end{equation}

\begin{lem}\label{lemA.8}
(i)
For every ball $B$ in $\mathbb{H}^n$, there holds
$$\int_B\Big|L_{\theta_{\mathbb{H}^n}}(\widetilde{V}_\nu-\widetilde{V})-r_\infty |\widetilde{V}_\nu-\widetilde{V}|^{\frac{2}{n}}(\widetilde{V}_\nu-\widetilde{V})\Big|^{\frac{2n+2}{n+2}}dV_{\theta_{\mathbb{H}^n}}
\to 0\mbox{
as }\nu\to\infty.$$\\
(ii) There exists a constant $C$ such that
$$\int_{\mathbb{H}^n}\left(|\nabla_{\theta_{\mathbb{H}^n}}(\widetilde{V}_\nu-\widetilde{V})|^2_{\theta_{\mathbb{H}^n}}+ |\widetilde{V}_\nu-\widetilde{V}|^{2+\frac{2}{n}}\right)dV_{\theta_{\mathbb{H}^n}}\leq C\hspace{2mm}
\mbox{
for all }\nu.$$
\end{lem}
\begin{proof}
By the same proof of Proposition \ref{propB.2.5}, we have
\begin{equation*}
\widetilde{V}_\nu^{1+\frac{2}{n}}-\widetilde{V}^{1+\frac{2}{n}}-|\widetilde{V}_\nu-\widetilde{V}|^{\frac{2}{n}}(\widetilde{V}_\nu-\widetilde{V})
\to 0\hspace{2mm}\mbox{ in }L^{\frac{2n+2}{n+2}}(B)\mbox{ as }\nu\to\infty.
\end{equation*}
This together with
(\ref{A45}) and Proposition \ref{propA.6}(i) implies that
\begin{equation*}
\begin{split}
&\int_B\Big|L_{\theta_{\mathbb{H}^n}}(\widetilde{V}_\nu-\widetilde{V})-r_\infty |\widetilde{V}_\nu-\widetilde{V}|^{\frac{2}{n}}(\widetilde{V}_\nu-\widetilde{V})\Big|^{\frac{2n+2}{n+2}}dV_{\theta_{\mathbb{H}^n}}\\
&\leq C\int_B\Big|L_{\theta_{\mathbb{H}^n}}\widetilde{V}_\nu-r_\infty |\widetilde{V}_\nu|^{\frac{2}{n}}\widetilde{V}_\nu\Big|^{\frac{2n+2}{n+2}}dV_{\theta_{\mathbb{H}^n}}\\
&\hspace{4mm}+C\,
r_\infty\int_B\Big|
\widetilde{V}_\nu^{1+\frac{2}{n}}-\widetilde{V}^{1+\frac{2}{n}}-|\widetilde{V}_\nu-\widetilde{V}|^{\frac{2}{n}}(\widetilde{V}_\nu-\widetilde{V})\Big|^{\frac{2n+2}{n+2}}
dV_{\theta_{\mathbb{H}^n}}
\to 0
\end{split}
\end{equation*}
as $\nu\to \infty$.
This proves (i).
On the other hand, (ii) follows from (\ref{A45}) and Proposition \ref{propA.6}(ii).
This proves the assertion.
\end{proof}

\begin{lem}\label{lemA.9}
For every ball $B$ in $\mathbb{H}^n$,
$\widetilde{V}_\nu$ converges to $\widetilde{V}$ strongly in $S_1^2(B)$ as $\nu\to\infty$.
\end{lem}
\begin{proof}
By contradiction, we assume that
$\widetilde{V}_\nu$ does not converge to $\widetilde{V}$ strongly in $S_1^2(B)$ as $\nu\to\infty$
for some $B$ in $\mathbb{H}^n$.
Therefore, it follows from Lemma \ref{lemA.8} that
the sequence
$\{\widetilde{V}_\nu-\widetilde{V}\}$
satisfied the same properties of the sequence
$\{v_\nu\}$. In particular, it follows from Lemma \ref{lemA.4} that
there exists a sequence $\big(\widetilde{x}_\nu,\widetilde{\rho}_\nu)$ with
$\widetilde{x}_\nu\in B_{\rho_{1,\nu}}(x^*_{1,\nu})$ and $\widetilde{\rho}_\nu\to 0$ as $\nu\to\infty$ such that
\begin{equation}\label{A46}
(2+\frac{2}{n})\int_{\exp_{x^*_{1,\nu}}^{-1}\big(B_{\widetilde{\rho}_\nu}(\widetilde{x}_\nu)\big)}|\nabla_{\theta_{\mathbb{H}^n}}(\widetilde{V}_\nu-\widetilde{V})|^2_{\theta_{\mathbb{H}^n}}
dV_{\theta_{\mathbb{H}^n}}\geq \frac{1+a_0}{2} r_\infty^{-n}
Y(B,\theta_{\mathbb{H}^n})^{n+1}
\end{equation}
for $\nu$ sufficiently large.
It follows from Lemma \ref{lemA.8} that
\begin{equation}\label{A47}
\int_{\exp_{x^*_{1,\nu}}^{-1}\big(B_{\widetilde{\rho}_\nu}(\widetilde{x}_\nu)\big)}
\langle\nabla_{\theta_{\mathbb{H}^n}}\widetilde{V},\nabla_{\theta_{\mathbb{H}^n}}(\widetilde{V}_\nu-\widetilde{V})\rangle_{\theta_{\mathbb{H}^n}}
dV_{\theta_{\mathbb{H}^n}}\to 0
\end{equation}
as $\nu\to\infty$.
Combining (\ref{A46}) and (\ref{A47}), we get
\begin{equation}\label{A48}
\begin{split}
&(2+\frac{2}{n})\int_{\exp_{x^*_{1,\nu}}^{-1}\big(B_{\widetilde{\rho}_\nu}(\widetilde{x}_\nu)\big)}|\nabla_{\theta_{\mathbb{H}^n}}\widetilde{V}_\nu|^2_{\theta_{\mathbb{H}^n}}
dV_{\theta_{\mathbb{H}^n}}\\
&=(2+\frac{2}{n})\int_{\exp_{x^*_{1,\nu}}^{-1}\big(B_{\widetilde{\rho}_\nu}(\widetilde{x}_\nu)\big)}|\nabla_{\theta_{\mathbb{H}^n}}(\widetilde{V}_\nu-\widetilde{V})|^2_{\theta_{\mathbb{H}^n}}
dV_{\theta_{\mathbb{H}^n}}\\
&\hspace{4mm}
+(2+\frac{2}{n})\int_{\exp_{x^*_{1,\nu}}^{-1}\big(B_{\widetilde{\rho}_\nu}(\widetilde{x}_\nu)\big)}|\nabla_{\theta_{\mathbb{H}^n}}\widetilde{V}|^2_{\theta_{\mathbb{H}^n}}
dV_{\theta_{\mathbb{H}^n}}
\\
&\hspace{4mm}-2(2+\frac{2}{n})\int_{\exp_{x^*_{1,\nu}}^{-1}\big(B_{\widetilde{\rho}_\nu}(\widetilde{x}_\nu)\big)}
\langle\nabla_{\theta_{\mathbb{H}^n}}\widetilde{V},\nabla_{\theta_{\mathbb{H}^n}}(\widetilde{V}_\nu-\widetilde{V})\rangle_{\theta_{\mathbb{H}^n}}
dV_{\theta_{\mathbb{H}^n}}\\
&\geq \frac{1+a_0}{2} r_\infty^{-n}
Y(B,\theta_{\mathbb{H}^n})^{n+1}+o(1)\\
&\geq a_0 r_\infty^{-n}
Y(M,\theta_0)^{n+1}
\end{split}
\end{equation}
for $\nu$ sufficiently large,
where the last inequality follows from
$$Y(B,\theta_{\mathbb{H}^n})=Y(\mathbb{H}^n,\theta_{\mathbb{H}^n})=Y(S^{2n+1},\theta_{S^{2n+1}})\geq Y(M,\theta_0).$$
Since $\widetilde{V}_\nu$
has support in $\{(t^2+|z|^4)^{\frac{1}{4}}<\frac{\widehat{\rho}}{\rho_{1,\nu}}\}$,
it follows from (\ref{A48})
that there exists $\widetilde{\widetilde{x}}_\nu\in B_{\rho_{1,\nu}}(x^*_{1,\nu})$
such that
\begin{equation*}
\int_{B_{\widetilde{\rho}_\nu\rho_{1,\nu}}(\widetilde{\widetilde{x}}_\nu)}\left((2+\frac{2}{n})|\nabla_{\theta_0}v_\nu|_{\theta_0}^2+R_{\theta_0}v_\nu^2\right)dV_{\theta_0}
\geq a_0 r_\infty^{-n}
Y(M,\theta_0)^{n+1}
\end{equation*}
for $\nu$ sufficiently large.
This implies that
$$\rho_{\nu}(\widetilde{\widetilde{x}}_\nu)\leq \widetilde{\rho}_\nu\rho_{1,\nu}
<\rho_{1,\nu},$$
where we have used the fact that $\widetilde{\rho}_\nu\to 0$ as $\nu\to\infty$.
But this contradicts to (\ref{A36}). This proves Lemma \ref{lemA.9}.
\end{proof}

Since $\widetilde{V}$ satisfies
(\ref{A45}), it follows from the result of Jerison and Lee in \cite{Jerison&Lee1}
that
there exists $(z_0,t_0)\in \mathbb{H}^n$ and $\gamma_1>0$ such that
$$\widetilde{V}(z,t)=W\circ T_{(z_0,t_0)}(z,t),$$
%%%%%%%%%%%%%%%%%
%%%%%%%%%%%%%%%%
where
$$W(z,t)=\left(\frac{n(2n+2)}{r_\infty}\right)^{\frac{n}{2}}\left(\frac{\gamma_1^2}{\gamma_1^4t^2+(\gamma_1^2|z|^2+1)^2}\right)^{\frac{n}{2}}$$
and
$$T_{(z_0,t_0)}(z,t)=(z+z_0,t+t_0+2Im(z\cdot z_0))\mbox{ for }(z,t)\in\mathbb{H}^n$$
is the translation in $\mathbb{H}^n$.
By the optimality of $(x^*_{1,\nu}, \rho_{1,\nu})$, we can conclude that
$(z_0,t_0)=(0,0)$; for if $(z_0,t_0)\neq (0,0)$, we can
find $(\widetilde{x}^*_{1,\nu}, \widetilde{\rho}_{1,\nu})$ with
$\widetilde{\rho}_{1,\nu}<\rho_{1,\nu}$ such that
\begin{equation*}
\int_{B_{\widetilde{\rho}_{1,\nu}}(\widetilde{x}^*_{1,\nu})}\left((2+\frac{2}{n})|\nabla_{\theta_0}v_\nu|_{\theta_0}^2+R_{\theta_0}v_\nu^2\right)dV_{\theta_0}
\geq a_0 r_\infty^{-n}
Y(M,\theta_0)^{n+1}.
\end{equation*}
Therefore, we have
\begin{equation}\label{A56}
\widetilde{V}(z,t)=\left(\frac{n(2n+2)}{r_\infty}\right)^{\frac{n}{2}}\left(\frac{\gamma_1^2}{\gamma_1^4t^2+(\gamma_1^2|z|^2+1)^2}\right)^{\frac{n}{2}}.
\end{equation}
We remark that it was claimed in \cite{Gamara1} that
$\gamma_1=1$ (see the last line in P.146 in \cite{Gamara1}), which
does not seem to be true. In fact, we will show that $\gamma_1\geq C_1$. Here $C_1$ is a positive constant depending only on $a_0$, $r_{\infty}$, and $M$. We need the following:
%%%%%%%%%%%%%%%%%%%%%
%%%%%%%%%%%%%%%%%%%%
\begin{prop}\label{propA.10}
For any $x\in M$ and for any $r>0$, there holds
\begin{equation*}
\int_{B_r(x)}v_\nu \Delta_{\theta_0}v_\nu dV_{\theta_0}
=\int_{B_r(x)}|\nabla_{\theta_0}v_\nu|^2_{\theta_0}dV_{\theta_0}+o(1).
\end{equation*}
\end{prop}
\begin{proof}
We consider the following sequence of cut-off functions:
\begin{equation}\label{A60}
\chi_\nu(y)=\left\{
          \begin{array}{ll}
            1, & \hbox{ if $d(x,y)\leq r$;} \\
            0, & \hbox{ if $d(x,y)\geq r+r_\nu$,}
          \end{array}
        \right.
\end{equation}
such that
\begin{equation}\label{A61}
0\leq \chi_\nu\leq 1,\hspace{2mm} |\nabla_{\theta_0}\chi_\nu|_{\theta_0}\leq \frac{C}{r_\nu}\hspace{2mm}\mbox{
and }\hspace{2mm}|\Delta_{\theta_0}\chi_\nu|\leq \frac{C}{r_\nu^2},
\end{equation}
where $r_\nu$ will be chosen later.
Since the function $\chi_\nu v_\nu$ has support in
$B_{r+2r_\nu}(x)$, it follows from integration by parts that
\begin{equation}\label{A62}
\begin{split}
\int_{B_{r+2r_\nu}(x)}\chi_\nu v_\nu \Delta_{\theta_0}(\chi_\nu v_\nu) dV_{\theta_0}
=\int_{B_{r+2r_\nu}(x)}|\nabla_{\theta_0}(\chi_\nu v_\nu)|^2_{\theta_0} dV_{\theta_0}.
\end{split}
\end{equation}

We are going to expand the left and right hand side of (\ref{A62}).
By (\ref{A60}), the left hand side of (\ref{A62}) can be written as
\begin{equation*}
\begin{split}
&\int_{B_{r+2r_\nu}(x)}\chi_\nu v_\nu \Delta_{\theta_0}(\chi_\nu v_\nu) dV_{\theta_0}\\
&=\int_{B_{r}(x)}v_\nu \Delta_{\theta_0}v_\nu dV_{\theta_0}
+\int_{B_{r+2r_\nu}(x)-B_{r}(x)}\chi_\nu^2 v_\nu \Delta_{\theta_0} v_\nu dV_{\theta_0}\\
&\hspace{4mm}
+\int_{B_{r+2r_\nu}(x)-B_{r}(x)}\chi_\nu v_\nu^2 \Delta_{\theta_0} \chi_\nu dV_{\theta_0}
+2\int_{B_{r+2r_\nu}(x)-B_{r}(x)}\chi_\nu v_\nu \langle\nabla_{\theta_0} \chi_\nu,
\nabla_{\theta_0} v_\nu\rangle_{\theta_0} dV_{\theta_0}\\
&:=\int_{B_{r}(x)}v_\nu \Delta_{\theta_0}v_\nu dV_{\theta_0}
+I+II+III.
\end{split}
\end{equation*}
By (\ref{A60}) and (\ref{A61}), we have
\begin{equation*}
|II|\leq\int_{B_{r+2r_\nu}(x)-B_{r}(x)}\left|\chi_\nu v_\nu^2 \Delta_{\theta_0} \chi_\nu\right| dV_{\theta_0}\leq \frac{C}{r_\nu^2}\int_M v_\nu^2dV_{\theta_0}.
\end{equation*}
Since $v_\nu$ converges to $0$ in $L^2(M)$ as $\nu\to\infty$, if we choose
\begin{equation}\label{A64}
r_\nu=\left(\int_M v_\nu^2dV_{\theta_0}\right)^{\frac{1}{4}}\to 0\hspace{2mm}\mbox{ as }\nu\to\infty,
\end{equation}
then $|II|\to 0$ as $\nu\to\infty$.
Since $v_\nu$ converges to $0$ in $L^2(M)$ as $\nu\to\infty$, we have
\begin{equation}\label{A65}
\int_M \left|R_{\theta_0}v_\nu^2 \right|dV_{\theta_0}\to 0 \hspace{2mm}\mbox{ as }\nu\to\infty.
\end{equation}
Since $r_\nu\to 0$ as $\nu\to\infty$ by (\ref{A64})
and $v_\nu$ is uniformly bounded in $S_1^2(M)\hookrightarrow L^{2+\frac{2}{n}}(M)$,
we have
\begin{equation}\label{A66}
\int_{B_{r+2r_\nu}(x)-B_{r}(x)} |v_\nu|^{2+\frac{2}{n}} dV_{\theta_0}\to 0 \hspace{2mm}\mbox{ as }\nu\to\infty.
\end{equation}
Therefore, we have
\begin{equation*}
\begin{split}
\left|(2+\frac{2}{n})I\right|
&\leq \left(\int_{B_{r+2r_\nu}(x)-B_{r}(x)} v_\nu^{2+\frac{2}{n}}dV_{\theta_0}\right)^{\frac{n}{2n+2}}
\left(\int_M\left|L_{\theta_0} v_\nu -r_\infty |v_\nu|^{\frac{2}{n}}v_\nu\right|^{\frac{2n+2}{n+2}}dV_{\theta_0}\right)^{\frac{n+2}{2n+2}}\\
&\hspace{4mm}+\int_M \left|R_{\theta_0}v_\nu^2 \right|dV_{\theta_0}+r_\infty \int_{B_{r+2r_\nu}(x)-B_{r}(x)} |v_\nu|^{2+\frac{2}{n}} dV_{\theta_0}
\\
&=o(1),
\end{split}
\end{equation*}
where we have used (\ref{Assumption}), (\ref{A65}) and (\ref{A66}) in the last equality.
By (\ref{A64}), Cauchy-Schwarz inequality
and the fact that $v_\nu$ is bounded in $S_1^2(M)$, we have
\begin{equation}\label{A68}
\begin{split}
|III|&\leq
2\int_{B_{r+2r_\nu}(x)-B_{r}(x)}\left|\chi_\nu v_\nu \langle\nabla_{\theta_0} \chi_\nu,
\nabla_{\theta_0} v_\nu\rangle_{\theta_0}\right| dV_{\theta_0}\\
&\leq \frac{C}{r_\nu}\left(\int_M v_\nu^2dV_{\theta_0}\right)^{\frac{1}{2}}
\left(\int_M |\nabla_{\theta_0} v_\nu|^2_{\theta_0}dV_{\theta_0}\right)^{\frac{1}{2}}
\to 0\hspace{2mm}\mbox{ as }\nu\to\infty.
\end{split}
\end{equation}
Combining all these, we can see that
the left hand side of (\ref{A62}) is equal to
\begin{equation}\label{A67}
\int_{B_{r+2r_\nu}(x)}\chi_\nu v_\nu \Delta_{\theta_0}(\chi_\nu v_\nu) dV_{\theta_0}=\int_{B_{r}(x)}v_\nu \Delta_{\theta_0}v_\nu dV_{\theta_0}+o(1).
\end{equation}

On the other hand, the right hand side of (\ref{A62}) can be written as
\begin{equation*}
\begin{split}
&\int_{B_{r+2r_\nu}(x)}|\nabla_{\theta_0}(\chi_\nu v_\nu)|^2_{\theta_0} dV_{\theta_0}\\
&=\int_{B_{r}(x)}|\nabla_{\theta_0} v_\nu|^2_{\theta_0} dV_{\theta_0}
+\int_{B_{r+2r_\nu}(x)-B_r(x)}\chi_\nu^2|\nabla_{\theta_0}v_\nu|^2_{\theta_0} dV_{\theta_0}\\
&\hspace{4mm}+\int_{B_{r+2r_\nu}(x)-B_r(x)}v_\nu^2|\nabla_{\theta_0}\chi_\nu|^2_{\theta_0} dV_{\theta_0}
+\int_{B_{r+2r_\nu}(x)-B_r(x)}\chi_\nu v_\nu\langle\nabla_{\theta_0}\chi_\nu,\nabla_{\theta_0}v_\nu\rangle_{\theta_0} dV_{\theta_0}\\
&:=\int_{B_{r}(x)}|\nabla_{\theta_0} v_\nu|^2_{\theta_0} dV_{\theta_0}
+I'+II'+III'.
\end{split}
\end{equation*}
Since $r_\nu\to 0$ as $\nu\to\infty$ by (\ref{A64})
and $v_\nu$ is uniformly bounded in $S_1^2(M)$, we have
$$|I'|\leq
\int_{B_{r+2r_\nu}(x)-B_r(x)}|\nabla_{\theta_0}v_\nu|^2_{\theta_0} dV_{\theta_0}
\to 0\hspace{2mm}\mbox{ as }\nu\to\infty.$$
By (\ref{A61}) and (\ref{A64}), we have
\begin{equation*}
|II'|=\int_{B_{r+2r_\nu}(x)-B_r(x)}\chi_\nu^2|\nabla_{\theta_0}v_\nu|^2_{\theta_0} dV_{\theta_0}
\leq \frac{C}{r_\nu^2}\int_M v_\nu^2dV_{\theta_0}
\to 0\hspace{2mm}\mbox{ as }\nu\to\infty.
\end{equation*}
Since $III=III'$, it follows from
(\ref{A68}) that
$|III'|\to 0$ as $\nu\to\infty$.
Combining all these, we can see that
the right hand side of (\ref{A62}) is equal to
\begin{equation}\label{A69}
\int_{B_{r+2r_\nu}(x)}|\nabla_{\theta_0}(\chi_\nu v_\nu)|^2_{\theta_0} dV_{\theta_0}
=\int_{B_{r}(x)}|\nabla_{\theta_0} v_\nu|^2_{\theta_0} dV_{\theta_0}+o(1).
\end{equation}

The assertion follows from combining (\ref{A62}), (\ref{A67}) and (\ref{A69}).
\end{proof}

\begin{prop}\label{propA.11}
There exists a positive constant $C_1$ depending only on $a_0$, $r_\infty$ and $M$ such that
$$\gamma_1\geq C_1.$$
\end{prop}
\begin{proof}
Since $\rho_{1,\nu}\to 0$ as $\nu\to\infty$ by Lemma \ref{lemA.5},
we can find $N$ such that
\begin{equation}\label{A57}
C_0\leq \frac{\widehat{\rho}}{2\rho_{1,\nu}}\hspace{2mm}\mbox{ whenever }\nu\geq N,
\end{equation}
where $C_0$ is the uniform constant given in (\ref{B.5.5}).
Note that
\begin{equation}\label{A70}
\begin{split}
&\int_{\{(t^2+|z|^4)^{\frac{1}{4}}<C_0\}}|\widetilde{V}_\nu(z,t)|^{2+\frac{2}{n}} dV_{\theta_{\mathbb{H}^n}}\\
&=\int_{\{(t^2+|z|^4)^{\frac{1}{4}}<C_0\}}
\Big|(\rho_{1,\nu})^n\,\widetilde{v}_\nu\big(\rho_{1,\nu}z,(\rho_{1,\nu})^2t\big)\Big|^{2+\frac{2}{n}}
dV_{\theta_{\mathbb{H}^n}}\\
&=\int_{\{\rho_{x^*_{1,\nu}}(y)<C_0\rho_{1,\nu}\}}|\widetilde{v}_\nu(y)|^{2+\frac{2}{n}} dV_{\widehat{\theta}_{x^*_{1,\nu}}}+o(1)\\
&\geq\int_{\{d(x^*_{1,\nu},y)<\rho_{1,\nu}\}}|v_\nu(y)|^{2+\frac{2}{n}} dV_{\theta_0}+o(1),\\
\end{split}
\end{equation}
where the first equality follows from (\ref{A57}) and the definition of $\widetilde{V}_\nu$
in (\ref{A44}), and
the second equality follows
from (\ref{B.0}) and the change of variables
$(\tilde{z},\tilde{t})=\big(\rho_{1,\nu}z,(\rho_{1,\nu})^2t\big)$,
and the last inequality follows from (\ref{B.5.5}) and the definition of $\widetilde{v}_\nu$.
For sufficiently large $\nu$, we have
\begin{equation}\label{A71}
\begin{split}
&r_\infty\int_{B_{\rho_{1,\nu}}(x^*_{1,\nu})}|v_\nu|^{2+\frac{2}{n}} dV_{\theta_0}\\
&\geq \int_{B_{\rho_{1,\nu}}(x^*_{1,\nu})}v_\nu L_{\theta_0}v_\nu dV_{\theta_0}\\
&\hspace{4mm}-
\left(\int_{M}\left|L_{\theta_0}v_\nu-r_\infty |v_\nu|^{\frac{2}{n}}v_\nu\right|^{\frac{2n+2}{n+2}}dV_{\theta_0}\right)^{\frac{n+2}{2n+2}}
\left(\int_M|v_\nu|^{2+\frac{2}{n}} dV_{\theta_0}\right)^{\frac{n}{2n+2}}\\
&=\int_{B_{\rho_{1,\nu}}(x^*_{1,\nu})}\left((2+\frac{2}{n})|\nabla_{\theta_0}v_\nu|^2_{\theta_0}+R_{\theta_0}v_\nu^{2+\frac{2}{n}} \right)dV_{\theta_0}+o(1)\\
&\geq a_0 r_\infty^{-n}
Y(M,\theta_0)^{n+1} +o(1),
\end{split}
\end{equation}
where the first inequality follows from H\"{o}lder's inequality,
 the last equality follows from Proposition \ref{propA.10}, (\ref{Assumption})
and the fact that $v_\nu$ is bounded in $L^{2+\frac{2}{n}}(M)$,
and the last inequality follows from
the definition of $(x^*_{1,\nu},\rho_{1,\nu})$ in
(\ref{A36.5}).

Combining (\ref{A70}) and (\ref{A71}), we obtain
$$\int_{\{(t^2+|z|^4)^{\frac{1}{4}}<C_0\}}|\widetilde{V}_\nu(z,t)|^{2+\frac{2}{n}} dV_{\theta_{\mathbb{H}^n}}
\geq a_0 r_\infty^{-n}
Y(M,\theta_0)^{n+1} +o(1).$$
Combining this with Lemma \ref{lemA.9}, we get
\begin{equation}\label{A72}
\int_{\{(t^2+|z|^4)^{\frac{1}{4}}<C_0\}}\widetilde{V}(z,t)^{2+\frac{2}{n}} dV_{\theta_{\mathbb{H}^n}}
\geq a_0 r_\infty^{-n}
Y(M,\theta_0)^{n+1} +o(1).
\end{equation}
We compute
%%%%%%%%%%%%
%%%%%%%%%%%%
\begin{equation}\label{A73}
\begin{split}
&\int_{\{(t^2+|z|^4)^{\frac{1}{4}}<C_0\}}\widetilde{V}(z,t)^{2+\frac{2}{n}} dV_{\theta_{\mathbb{H}^n}}\\
&=
\left(\frac{n(2n+2)}{r_\infty}\right)^{n+1}
\int_{\{(t^2+|z|^4)^{\frac{1}{4}}<C_0\}}\left(\frac{\gamma_1^2}{\gamma_1^4t^2+(\gamma_1^2|z|^2+1)^2}\right)^{n+1}dV_{\theta_{\mathbb{H}^n}}\\
&=
\left(\frac{n(2n+2)}{r_\infty}\right)^{n+1}
\int_{\{(\tilde{t}^2+|\tilde{z}|^4)^{\frac{1}{4}}<\gamma_1 C_0\}}\left(\frac{1}{\tilde{t}^2+(|\tilde{z}|^2+1)^2}\right)^{n+1}dV_{\theta_{\mathbb{H}^n}}
\end{split}
\end{equation}
where the first equality follows from (\ref{A56})
and
the second equality follows from change of variables
$(\tilde{z},\tilde{t})=(\gamma_1z,\gamma_1^2t)$.
Note that the last term in (\ref{A73}) tends to zero as $\gamma_1\to 0^+$.
Hence, combining (\ref{A72}) and (\ref{A73}), we can conclude that $\gamma_1$ is bounded below by a positive constant
$C_1$
depending only on $a_0$, $r_\infty$, and $M$.
This proves the assertion.
\end{proof}

For any $(x,\lambda)\in M\times (0,\infty)$, we can find a unique solution $\widehat{\omega}(x,\lambda)$ of the following equation:
\begin{equation}\label{A74}
L_{\theta_0}\widehat{\omega}(x,\lambda)=r_\infty\omega'(x,\lambda)^{1+\frac{2}{n}} \hspace{2mm}\mbox{ in }M,
\end{equation}
where $\omega'(x,\lambda)$ is defined as
\begin{equation*}
\omega'(x,\lambda)(y)=
\left\{
  \begin{array}{ll}
    \chi_\delta(\rho_x(y))\varphi_x(y)\omega(x,\lambda)(y), & \hbox{ for $y\in B_{2\delta}(x)$;} \\
    0, & \hbox{ otherwise.}
  \end{array}
\right.
\end{equation*}
Here $\chi_\delta$ is the cut-off function defined in (\ref{testfun2}),
$\varphi_x$ is the conformal factor such that
$\widehat{\theta}_x=\varphi_x^{\frac{2}{n}}\theta_0$ in a neighborhood $B_{3\delta}(x)$ of $x$.
Moreover, $\omega(x,\lambda)(y)$ is given by
\begin{equation}\label{A41.4}
\omega(x,\lambda)(y)=\left(\frac{n(2n+2)}{r_\infty}\right)^{\frac{n}{2}}\left(\frac{\lambda^2}{\lambda^4t^2+(\lambda^2|z|^2+1)^2}\right)^{\frac{n}{2}},
\end{equation}
where $(z,t)$ is CR normal coordinates of $y$ centered at $x$.
It follows from the definition of $\widehat{\omega}(x,\lambda)$ that
\begin{equation}\label{A49}
\widehat{\omega}(x,\lambda)(y)=\omega'(x,\lambda)(y)=0\hspace{2mm}\hbox{ for }y\in M-B_{2\delta}(x).
\end{equation}
When $n=1$, there holds (see Proposition 1 in \cite{Gamara2})
\begin{equation}\label{A50}
|\widehat{\omega}(x,\lambda)(y)-\omega'(x,\lambda)(y)|\leq C\lambda^{-1}(1+|\log(\lambda^{-2}+\rho_x(y)^2)|)\hspace{2mm}\mbox{ for }y\in B_{2\delta}(x)
\end{equation}
and (see (3.6) in \cite{Gamara2})
\begin{equation}\label{A51}
\big|L_{\theta_0}\big(\widehat{\omega}(x,\lambda)(y)-\omega'(x,\lambda)(y)\big)\big|\leq
\inf\Big\{1,\frac{C}{\rho_x(y)^2+\lambda^{-2}}\Big\} \hspace{2mm}\mbox{ for }y\in B_{2\delta}(x).
\end{equation}
It follows from (\ref{A49})-(\ref{A51}) that
\begin{equation}\label{A41.3}
\|\widehat{\omega}(x,\lambda)-\omega'(x,\lambda)\|_{S_1^2(M)}\to 0\hspace{2mm}\mbox{ as }\lambda\to\infty.
\end{equation}
Similarly, when $M$ is spherical, it follows from Lemma 3 and Lemma 4 in \cite{Gamara1} that
\begin{equation}\label{A41.6}
|\widehat{\omega}(x,\lambda)(y)-\omega'(x,\lambda)(y)|\leq\frac{C}{\lambda^n}\hspace{2mm}\mbox{ for }y\in B_{2\delta}(x)
\end{equation}
and
\begin{equation}\label{A41.5}
\|\widehat{\omega}(x,\lambda)-\omega'(x,\lambda)\|_{S_1^2(M)}=O(\frac{1}{\lambda}).
\end{equation}

We have the following:

\begin{prop}\label{propA.12}
There holds
$$
\int_{B_{\rho_{1,\nu}}(x^*_{1,\nu})}
\Big|v_\nu-\omega'(x^*_{1,\nu},\frac{\gamma_1}{\rho_{1,\nu}})\Big|^{2+\frac{2}{n}}dV_{\theta_0}
\to 0\hspace{2mm}\mbox{ as }\nu\to\infty. $$
\end{prop}
\begin{proof}
By Lemma \ref{lemA.5}, we can find $N$ such that
$$C_0 \rho_{1,\nu}\leq \delta\hspace{2mm}\mbox{ for }\nu\geq N,$$
where $C_0$ is the constant in (\ref{B.5.5}).
Therefore, for $\nu\geq N$, we have
\begin{equation}\label{A78}
\begin{split}
&\int_{B_{\rho_{1,\nu}}(x^*_{1,\nu})}
\Big|v_\nu-\omega'(x^*_{1,\nu},\frac{\gamma_1}{\rho_{1,\nu}})\Big|^{2+\frac{2}{n}}dV_{\theta_0}\\
&=\int_{B_{\rho_{1,\nu}}(x^*_{1,\nu})}
\Big|\widetilde{v}_\nu(x)-
\chi_\delta(\rho_{x^*_{1,\nu}}(x))\omega(x^*_{1,\nu},\frac{\gamma_1}{\rho_{1,\nu}})(x)
\Big|^{2+\frac{2}{n}} dV_{\widehat{\theta}_{x^*_{1,\nu}}}\\
&\leq
\int_{\{(|z|^4+t^2)^{\frac{1}{4}}\leq C_0\rho_{1,\nu}\}}
\Big|\widetilde{v}_\nu(z,t)-
\omega(x^*_{1,\nu},\frac{\gamma_1}{\rho_{1,\nu}})(z,t)\Big|^{2+\frac{2}{n}}
dV_{\theta_{\mathbb{H}^n}}\\
&=\int_{\{(|z|^4+t^2)^{\frac{1}{4}}\leq C_0\}}
\Big|\widetilde{V}_\nu(\tilde{z},\tilde{t})-\widetilde{V}_\nu(\tilde{z},\tilde{t})\Big|^{2+\frac{2}{n}}
dV_{\theta_{\mathbb{H}^n}},
\end{split}
\end{equation}
where the first inequality follows from (\ref{B.5.5}),
the first equality from the definition of $\omega'(x^*_{1,\nu},\frac{\gamma_1}{\rho_{1,\nu}})$ in (\ref{A74}),
and the last equality follows from  (\ref{A56}), (\ref{A41.4}) and
 the change of variables $(\tilde{z},\tilde{t})=(\frac{z}{\rho_{1,\nu}},\frac{t}{(\rho_{1,\nu})^2})$.
Thanks to Lemma \ref{lemA.9}, the last expression in (\ref{A78}) tends to zero as $\nu\to\infty$.
This proves the assertion.
\end{proof}

We can now extract from $\{v_\nu\}$ the first bubble
and consider the following new sequence of functions:
$$v_\nu^1(x)=v_\nu(x)-\widehat{\omega}(x^*_{1,\nu},\frac{\gamma_1}{\rho_{1,\nu}})(x).$$

\begin{lem}\label{lemA.10}
(i)
The sequence $\{v_\nu^1\}$ satisfies
\begin{equation*}
\int_M\Big|L_{\theta_0}v_\nu^1-r_\infty |v_\nu^1|^{\frac{2}{n}}v_\nu^1\Big|^{\frac{2n+2}{n+2}}dV_{\theta_0}
\to 0\hspace{2mm}\mbox{ as }\nu\to\infty
\end{equation*}
(ii) There exists a constant $C$ such that
$$
\int_{M}\left(|\nabla_{\theta_0}v^1_\nu|_{\theta_0}^2+(v^1_\nu)^2\right)dV_{\theta_0}\leq C\hspace{2mm}
\mbox{ for all }\nu.$$
\end{lem}
\begin{proof}
We follow the proof of Lemma 14 in \cite{Gamara1}.
By Lemma \ref{lemA.5}, we can choose $N$ such that
$\rho_{1,\nu}\leq 2\delta$ for $\nu\geq N.$
For $\nu\geq N$, it follows from (\ref{A49}) that
\begin{equation}\label{A77}
\begin{split}
&\int_M\Big|L_{\theta_0}v_\nu^1-r_\infty |v_\nu^1|^{\frac{2}{n}}v_\nu^1\Big|^{\frac{2n+2}{n+2}}dV_{\theta_0}\\
&\leq C\int_{M}\Big|L_{\theta_0}v_\nu-r_\infty |v_\nu|^{\frac{2}{n}}v_\nu\Big|^{\frac{2n+2}{n+2}}dV_{\theta_0}\\
&\hspace{4mm}+C\int_{B_{\rho_{1,\nu}}(x^*_{1,\nu})}\Big|L_{\theta_0}\widehat{\omega}(x^*_{1,\nu},\frac{\gamma_1}{\rho_{1,\nu}})-r_\infty \widehat{\omega}(x^*_{1,\nu},\frac{\gamma_1}{\rho_{1,\nu}})^{1+\frac{2}{n}}\Big|^{\frac{2n+2}{n+2}}dV_{\theta_0}\\
&\hspace{4mm}+C\,r_\infty\int_{B_{\rho_{1,\nu}}(x^*_{1,\nu})}\left||v_\nu|^{\frac{2}{n}}v_\nu
-\widehat{\omega}(x^*_{1,\nu},\frac{\gamma_1}{\rho_{1,\nu}})^{1+\frac{2}{n}}\right.\\
&\hspace{20mm}\left.
-\Big|v_\nu-\widehat{\omega}(x^*_{1,\nu},\frac{\gamma_1}{\rho_{1,\nu}})\Big|^{\frac{2}{n}}
\Big(v_\nu-\widehat{\omega}(x^*_{1,\nu},\frac{\gamma_1}{\rho_{1,\nu}})\Big)\right|^{\frac{2n+2}{n+2}}dV_{\theta_0}.
\end{split}
\end{equation}
It follows from (\ref{A49}), (\ref{A50}), (\ref{A41.6}),
Lemma \ref{lemA.5} and
 the definition of $\widehat{\omega}(x,\lambda)$ that
\begin{equation}\label{A76}
\begin{split}
&\int_M\Big|L_{\theta_0}\widehat{\omega}(x^*_{1,\nu},\frac{\gamma_1}{\rho_{1,\nu}})-r_\infty \widehat{\omega}(x^*_{1,\nu},\frac{\gamma_1}{\rho_{1,\nu}})^{1+\frac{2}{n}}\Big|^{\frac{2n+2}{n+2}}dV_{\theta_0}\\
&=r_\infty\int_M\Big|\omega'(x^*_{1,\nu},\frac{\gamma_1}{\rho_{1,\nu}})^{1+\frac{2}{n}}- \widehat{\omega}(x^*_{1,\nu},\frac{\gamma_1}{\rho_{1,\nu}})^{1+\frac{2}{n}}\Big|^{\frac{2n+2}{n+2}}dV_{\theta_0}
=O(\rho_{1,\nu})=o(1).
\end{split}
\end{equation}
On the other hand,
\begin{equation*}
\begin{split}
&\left||v_\nu|^{\frac{2}{n}}v_\nu
-\widehat{\omega}(x^*_{1,\nu},\frac{\gamma_1}{\rho_{1,\nu}})^{1+\frac{2}{n}}-\Big|v_\nu-\widehat{\omega}(x^*_{1,\nu},\frac{\gamma_1}{\rho_{1,\nu}})\Big|^{\frac{2}{n}}
\Big(v_\nu-\widehat{\omega}(x^*_{1,\nu},\frac{\gamma_1}{\rho_{1,\nu}})\Big)\right|\\
&\leq
\left||v_\nu|^{\frac{2}{n}}v_\nu
-\omega'(x^*_{1,\nu},\frac{\gamma_1}{\rho_{1,\nu}})^{1+\frac{2}{n}}-\Big|v_\nu-\omega'(x^*_{1,\nu},\frac{\gamma_1}{\rho_{1,\nu}})\Big|^{\frac{2}{n}}
\Big(v_\nu-\omega'(x^*_{1,\nu},\frac{\gamma_1}{\rho_{1,\nu}})\Big)\right|\\
&\hspace{4mm}
+\left|\omega'(x^*_{1,\nu},\frac{\gamma_1}{\rho_{1,\nu}})^{1+\frac{2}{n}}
-\widehat{\omega}(x^*_{1,\nu},\frac{\gamma_1}{\rho_{1,\nu}})^{1+\frac{2}{n}}\right|+\left|
\omega'(x^*_{1,\nu},\frac{\gamma_1}{\rho_{1,\nu}})-\widetilde{\omega}(x^*_{1,\nu},\frac{\gamma_1}{\rho_{1,\nu}})\right|^{1+\frac{2}{n}}\\
&\hspace{4mm}+\left|
\Big|v_\nu-\omega'(x^*_{1,\nu},\frac{\gamma_1}{\rho_{1,\nu}})\Big|^{\frac{2}{n}}
\Big(\widetilde{\omega}(x^*_{1,\nu},\frac{\gamma_1}{\rho_{1,\nu}})-\omega'(x^*_{1,\nu},\frac{\gamma_1}{\rho_{1,\nu}})\Big)\right|\\
&\hspace{4mm}+\left|
\Big|\widetilde{\omega}(x^*_{1,\nu},\frac{\gamma_1}{\rho_{1,\nu}})-\omega'(x^*_{1,\nu},\frac{\gamma_1}{\rho_{1,\nu}})\Big|^{\frac{2}{n}}
\Big(v_\nu-\omega'(x^*_{1,\nu},\frac{\gamma_1}{\rho_{1,\nu}})\Big)\right|\\
&:=I+II+III+IV+V.
\end{split}
\end{equation*}
It follows from the proof of (\ref{A7}) that
$$I=O\left(|v_\nu|^{\frac{2}{n}}
\Big|v_\nu-\omega'(x^*_{1,\nu},\frac{\gamma_1}{\rho_{1,\nu}})\Big|
+\Big|v_\nu-\omega'(x^*_{1,\nu},\frac{\gamma_1}{\rho_{1,\nu}})\Big|^{\frac{2}{n}}
|v_\nu|\right).$$
This together with H\"{o}lder's inequality implies that
\begin{equation*}
\begin{split}
&\int_{B_{\rho_{1,\nu}}(x^*_{1,\nu})}
\left||v_\nu|^{\frac{2}{n}}v_\nu
-\omega'(x^*_{1,\nu},\frac{\gamma_1}{\rho_{1,\nu}})^{1+\frac{2}{n}}-\Big|v_\nu-\omega'(x^*_{1,\nu},\frac{\gamma_1}{\rho_{1,\nu}})\Big|^{\frac{2}{n}}
\Big(v_\nu-\omega'(x^*_{1,\nu},\frac{\gamma_1}{\rho_{1,\nu}})\Big)\right|^{\frac{2n+2}{n+2}}dV_{\theta_0}\\
&\leq C\left(\int_M|v_\nu|^{2+\frac{2}{n}}\right)^{\frac{2}{n+2}}\left(
\int_{B_{\rho_{1,\nu}}(x^*_{1,\nu})}\Big|v_\nu-\omega'(x^*_{1,\nu},\frac{\gamma_1}{\rho_{1,\nu}})\Big|^{2+\frac{2}{n}}dV_{\theta_0}\right)^{\frac{n}{n+2}}\\
&\hspace{4mm}
+C\left(
\int_{B_{\rho_{1,\nu}}(x^*_{1,\nu})}\Big|v_\nu-\omega'(x^*_{1,\nu},\frac{\gamma_1}{\rho_{1,\nu}})\Big|^{2+\frac{2}{n}}dV_{\theta_0}\right)^{\frac{2}{n+2}}
\left(\int_M|v_\nu|^{2+\frac{2}{n}}\right)^{\frac{n}{n+2}}\\
&=o(1),
\end{split}
\end{equation*}
where the last equality follows from Proposition \ref{propA.12}
and the fact that $v_\nu$ is bounded in $L^{2+\frac{2}{n}}(M)$.
Also, it follows from (\ref{A49}), (\ref{A50}) and (\ref{A41.6}) that
\begin{equation*}
II\leq C\frac{\rho_{1,\nu}}{\gamma_1}\hspace{2mm}\mbox{ and }\hspace{2mm}III\leq C\frac{\rho_{1,\nu}}{\gamma_1}.
\end{equation*}
It follows from (\ref{A49}), (\ref{A50}), (\ref{A41.6}) and H\"{o}lder's inequality that
\begin{equation*}
IV\leq C\frac{\rho_{1,\nu}}{\gamma_1}\Big|v_\nu-\omega'(x^*_{1,\nu},\frac{\gamma_1}{\varepsilon^*_{1,\nu}})\Big|^{\frac{2}{n}}\hspace{2mm}\mbox{ and }\hspace{2mm}V\leq C\Big(\frac{\varepsilon^*_{1,\nu}}{\gamma_1}\Big)^{\frac{2}{n}}
\Big|v_\nu-\omega'(x^*_{1,\nu},\frac{\gamma_1}{\rho_{1,\nu}})\Big|.
\end{equation*}
Combining all these with
Lemma \ref{lemA.5}, Proposition \ref{propA.11} and Proposition \ref{propA.12}, we can conclude that
\begin{equation}\label{A75}
\begin{split}
&\int_{B_{\rho_{1,\nu}}(x^*_{1,\nu})}\left||v_\nu|^{\frac{2}{n}}v_\nu
-\widehat{\omega}(x^*_{1,\nu},\frac{\gamma_1}{\rho_{1,\nu}})^{1+\frac{2}{n}}\right.\\
&\hspace{16mm}\left.
-\Big|v_\nu-\widehat{\omega}(x^*_{1,\nu},\frac{\gamma_1}{\rho_{1,\nu}})\Big|^{\frac{2}{n}}
\Big(v_\nu-\widehat{\omega}(x^*_{1,\nu},\frac{\gamma_1}{\rho_{1,\nu}})\Big)\right|^{\frac{2n+2}{n+2}}dV_{\theta_0}
=o(1).
\end{split}
\end{equation}
Now (i) follows from (\ref{Assumption}), (\ref{A77}), (\ref{A76}) and (\ref{A75}).

For (ii), note that
\begin{equation}\label{A52}
\begin{split}
&\int_M\left((2+\frac{2}{n})|\nabla_{\theta_0}\widehat{\omega}(x,\lambda)|_{\theta_0}^2+R_{\theta_0}\widehat{\omega}(x,\lambda)^2\right)dV_{\theta_0}\\
&=r_\infty\int_M\omega'(x,\lambda)^{1+\frac{2}{n}}\widehat{\omega}(x,\lambda)dV_{\theta_0}\\
&=r_\infty\int_M\omega'(x,\lambda)^{2+\frac{2}{n}}dV_{\theta_0}+O(\frac{1}{\lambda})
\int_M\omega'(x,\lambda)^{1+\frac{2}{n}}dV_{\theta_0}\\
&\leq r_\infty\int_M\omega'(x,\lambda)^{2+\frac{2}{n}}dV_{\theta_0}+O(\frac{1}{\lambda})
\left(\int_M\omega'(x,\lambda)^{2+\frac{2}{n}}dV_{\theta_0}\right)^{\frac{n+2}{2n+2}}
\end{split}
\end{equation}
where the first equality follows from the definition of $\widehat{\omega}(x,\lambda)$,
the second equality follows from (\ref{A49}), (\ref{A50}) and (\ref{A41.6}),
the last inequality follows from the H\"{o}lder's inequality.
Note also that
it follows from the definition of $\omega'(x,\lambda)$ that
\begin{equation}\label{A53}
\begin{split}
\int_M\omega'(x,\lambda)^{2+\frac{2}{n}}dV_{\theta_0}
\leq\int_{B_{2\delta}(x)}\omega(x,\lambda)^{2+\frac{2}{n}}dV_{\theta_0}
\leq \int_M\omega(x,1)^{2+\frac{2}{n}}dV_{\theta_0},
\end{split}
\end{equation}
where the last inequality  follows from the change of variables $(\tilde{z},\tilde{t})=(\lambda z,\lambda^2t)$
in the CR normal coordinates.
We derive from (\ref{A52}) and (\ref{A53})
that there exists a uniform constant $C$ such that
\begin{equation}\label{A54}
\int_M\left((2+\frac{2}{n})|\nabla_{\theta_0}\widehat{\omega}(x,\lambda)|_{\theta_0}^2+R_{\theta_0}\widehat{\omega}(x,\lambda)^2\right)dV_{\theta_0}\leq
C
\end{equation}
when $\lambda$ is sufficiently large.
By Cauchy-Schwarz inequality, we have
\begin{equation}\label{A55}
\begin{split}
&\int_{M}\left(|\nabla_{\theta_0}v^1_\nu|_{\theta_0}^2+(v^1_\nu)^2\right)dV_{\theta_0}\\
&\leq 2\int_{M}\left(|\nabla_{\theta_0}v_\nu|_{\theta_0}^2+v_\nu^2\right)dV_{\theta_0}
+2\int_{M}\left(|\nabla_{\theta_0}\widehat{\omega}(x^*_{1,\nu},\frac{\gamma_1}{\rho_{1,\nu}})|_{\theta_0}^2+
\widehat{\omega}(x^*_{1,\nu},\frac{\gamma_1}{\rho_{1,\nu}})^2\right)dV_{\theta_0}.
\end{split}
\end{equation}
Now (ii) follows from combining (\ref{A54}), (\ref{A55})
and the assumption that $\{v_\nu\}$ is uniformly bounded in $S_1^2(M)$.
This proves Lemma \ref{lemA.10}.
\end{proof}

Iterating the above procedure, either $v_\nu^1$ converges strongly to $0$  in $S_1^2(M)$ as $\nu\to\infty$,
or we can find a new sequence $(x^*_{2,\nu},\rho_{2,\nu})$
and extract another bubble by defining
$$v_\nu^2(x)=v_\nu^1(x)-\widehat{\omega}(x^*_{2,\nu},\frac{\gamma_2}{\rho_{2,\nu}})(x)$$
and show that
\begin{equation*}
\int_M\Big|L_{\theta_0}v_\nu^2-r_\infty |v_\nu^2|^{\frac{2}{n}}v_\nu^2\Big|^{\frac{2n+2}{n+2}}dV_{\theta_0}
\to 0\hspace{2mm}\mbox{ as }\nu\to\infty.
\end{equation*}
On the other hand, it can be shown that (see Lemma 15 and Lemma 16 in \cite{Gamara1})
\begin{equation}\label{A42}
\rho_{2,\nu}\geq\frac{1}{2}\rho_{1,\nu}\hspace{2mm}\mbox{ and }\hspace{2mm}
\frac{\rho_{2,\nu}}{\rho_{1,\nu}}+\frac{d(x^*_{1,\nu},x^*_{2,\nu})^2}{\rho_{1,\nu}\,\rho_{2,\nu}}
\to\infty
\end{equation}
as $\nu\to\infty$. Here
 $d$ is Carnot-Carath\'{e}odory distance on $M$
with respect to the contact form $\theta_0$.
This argument can be iterated as long as the new sequence $\{v_\nu^l\}$ does not coverage strongly to $0$
in $S_1^2(M)$.
And we claim that the iteration must terminate in finite steps. To see this, note that
\begin{equation*}
\begin{split}
&\int_M \left((2+\frac{2}{n})|\nabla_{\theta_0}v_\nu^{l}|^2_{\theta_0}
+R_{\theta_0}(v_\nu^{l})^2\right)dV_{\theta_0}\\
&=\int_M \left((2+\frac{2}{n})\Big|\nabla_{\theta_0}\Big(v_\nu^{l-1}-\widehat{\omega}(x^*_{l,\nu},\frac{\gamma_{l}}{\rho_{l,\nu}})\Big)\Big|^2_{\theta_0}
+R_{\theta_0}\Big(v_\nu^{l-1}-\widehat{\omega}(x^*_{l,\nu},\frac{\gamma_{l}}{\rho_{l,\nu}})\Big)^2\right)dV_{\theta_0}\\
&=\int_M \left((2+\frac{2}{n})|\nabla_{\theta_0}v_\nu^{l-1}|^2_{\theta_0}
+R_{\theta_0}(v_\nu^{l-1})^2\right)dV_{\theta_0}\\
&\hspace{4mm}
+r_\infty\int_M\widehat{\omega}(x^*_{l,\nu},\frac{\gamma_{l}}{\rho_{l,\nu}})
\omega'(x^*_{l,\nu},\frac{\gamma_{l}}{\rho_{l,\nu}})^{1+\frac{2}{n}}dV_{\theta_0}
+o(1),
\end{split}
\end{equation*}
where the first equality follows from the fact
that $v_\nu^{l-1}$ converges to $0$ weakly in $S_1^2(M)$ as $\nu\to\infty$,
and the  last equality follows from (\ref{A74}).
We compute
\begin{equation*}
\begin{split}
&\int_M\widehat{\omega}(x^*_{l,\nu},\frac{\gamma_{l}}{\rho_{l,\nu}})
\omega'(x^*_{l,\nu},\frac{\gamma_{l}}{\rho_{l,\nu}})^{1+\frac{2}{n}}dV_{\theta_0}\\
&\geq\int_{B_{\delta}(x^*_{l,\nu})}\Big(\varphi_{x^*_{l,\nu}}(y)\omega(x^*_{l,\nu},\frac{\gamma_{l}}{\rho_{l,\nu}})(y)\Big)^{2+\frac{2}{n}}dV_{\theta_0}+o(1)\\
&\geq \int_{\{(|z|^4+t^2)^{\frac{1}{4}}\leq \frac{\delta}{C_0}\}}\omega(x^*_{l,\nu},\frac{\gamma_{l}}{\rho_{l,\nu}})(z,t)^{2+\frac{2}{n}}dV_{\theta_{\mathbb{H}^n}}+o(1)\\
&=\left(\frac{n(2n+2)}{r_\infty}\right)^{n+1}
\int_{\{(|\tilde{z}|^4+\tilde{t}^2)^{\frac{1}{4}}\leq \frac{\delta}{C_0\rho_{l,\nu}}\}}
\left(\frac{\gamma_{l}^2}{\gamma_{l}^4\tilde{t}^2+(\gamma_{l}^2|\tilde{z}|^2+1)^2}\right)^{n+1}dV_{\theta_{\mathbb{H}^n}}+o(1)\\
&=\left(\frac{n(2n+2)}{r_\infty}\right)^{n+1}
\int_{\{(|\widehat{z}|^4+\widehat{t}^2)^{\frac{1}{4}}\leq \frac{\gamma_{l}\delta}{C_0\rho_{l,\nu}}\}}
\left(\frac{1}{\widehat{t}^2+(|\widehat{z}|^2+1)^2}\right)^{n+1}dV_{\theta_{\mathbb{H}^n}}+o(1)\\
&\geq\left(\frac{n(2n+2)}{r_\infty}\right)^{n+1}
\int_{\{(|\widehat{z}|^4+\widehat{t}^2)^{\frac{1}{4}}\leq \frac{C_1\delta}{C_0}\}}
\left(\frac{1}{\widehat{t}^2+(|\widehat{z}|^2+1)^2}\right)^{n+1}dV_{\theta_{\mathbb{H}^n}}+o(1),
\end{split}
\end{equation*}
where the first inequality follows from (\ref{A49}), (\ref{A50}), (\ref{A41.6}) and the definition of $\omega'(x^*_{l,\nu},\frac{\gamma_{l}}{\rho_{l,\nu}})$,
the second inequality follows from (\ref{B.0}) and (\ref{B.5.5}),
the first equality follows from
the change of variables
$(\tilde{z},\tilde{t})=(\frac{z}{\rho_{l,\nu}},\frac{t}{(\rho_{l,\nu})^2})$,
the  second equality follows from the change of variables
$(\widehat{z},\widehat{t})=(\gamma_l\tilde{z},\gamma_{l}^2\tilde{t})$,
and the last inequality follows from Proposition \ref{propA.11} and Lemma \ref{lemA.5}.
 Hence, if we let
$$C_2=r_\infty\left(\frac{n(2n+2)}{r_\infty}\right)^{n+1}
\int_{\{(|\widehat{z}|^4+\widehat{t}^2)^{\frac{1}{4}}\leq \frac{C_1\delta}{C_0}\}}
\left(\frac{1}{\widehat{t}^2+(|\widehat{z}|^2+1)^2}\right)^{n+1}dV_{\theta_{\mathbb{H}^n}},$$
then it follows from the above computation that
\begin{equation*}
\begin{split}
&\int_M \left((2+\frac{2}{n})|\nabla_{\theta_0}v_\nu^{l}|^2_{\theta_0}
+R_{\theta_0}(v_\nu^{l})^2\right)dV_{\theta_0}\\
&\geq \int_M \left((2+\frac{2}{n})|\nabla_{\theta_0}v_\nu^{l-1}|^2_{\theta_0}
+R_{\theta_0}(v_\nu^{l-1})^2\right)dV_{\theta_0}+C_2+o(1).
\end{split}
\end{equation*}
That is to say, the quantity
$\displaystyle\int_M \left((2+\frac{2}{n})|\nabla_{\theta_0}v_\nu^{l-1}|^2_{\theta_0}
+R_{\theta_0}(v_\nu^{l-1})^2\right)dV_{\theta_0}$
at the $l$-th step decreases by at least $C_2$ after extraction of a bubble.
Therefore, the iteration must stop after finite steps.

Therefore, there exists an integer $m$
and a sequence of $m$-tuples $(x^*_{k,\nu},\varepsilon^*_{k,\nu})_{1\leq k\leq m}$
where
$\varepsilon^*_{k,\nu}=\displaystyle\frac{\rho_{k,\nu}}{\gamma_k}$
such that
$$
\varepsilon^*_{k,\nu}\to 0\hspace{2mm}\mbox{ as }\nu\to\infty\mbox{
for all }1\leq k\leq m,$$
by Lemma \ref{lemA.5} and Proposition\ref{propA.11}.
Also, we have
\begin{equation}\label{A43}
\Big\|v_\nu-\sum_{k=1}^m\widehat{\omega}(x^*_{k,\nu},\frac{1}{\varepsilon^*_{k,\nu}})\Big\|_{S^2_1(M)}\rightarrow
0\hspace{2 mm}\mbox{ as }\hspace{2 mm}\nu\rightarrow\infty.
\end{equation}
Now (\ref{4.4}) follows from (\ref{A42}) and Proposition \ref{propA.11}.
On the other hand, we have
\begin{equation*}
\begin{split}
&\Big\|u_\nu-u_\infty-\sum_{k=1}^m\overline{u}_{(x^*_{k,\nu},\varepsilon^*_{k,\nu})}\Big\|_{S^2_1(M)}\\
&=\left\|v_\nu-\sum_{k=1}^m\widehat{\omega}(x^*_{k,\nu},\frac{1}{\varepsilon^*_{k,\nu}})
-\sum_{k=1}^m\left(\omega'(x^*_{k,\nu},\frac{1}{\varepsilon^*_{k,\nu}})-\widehat{\omega}(x^*_{k,\nu},\frac{1}{\varepsilon^*_{k,\nu}})\right)\right.\\
&\hspace{6mm}\left.
-\sum_{k=1}^m
\varphi_{x^*_{k,\nu}}(y)
\left(\frac{n(2n+2)}{r_\infty}\right)^{\frac{n}{2}}
(\varepsilon^*_{k,\nu})^n
\Big(1-\chi_\delta(\rho_{x^*_{k,\nu}}(y))\Big)G_{x^*_{k,\nu}}(y)\right\|_{S^2_1(M)}=o(1)
\end{split}
\end{equation*}
where the first equality follows from (\ref{testfcn}) and (\ref{testfun2}),
and the last equality follows from (\ref{A41.4}), (\ref{A41.3}), (\ref{A41.5}), (\ref{A43}),
Lemma \ref{lemA.5} and the fact that
the Green's function $G_{x^*_{k,\nu}}(y)$ is bounded in $S_1^2(K)$
for any compact set $K\subset M-\{x^*_{k,\nu}\}$ (see (\ref{B.3}) and (\ref{B.4})).
This proves (\ref{4.5}) and this completes the proof of Theorem \ref{Theorem4.1}.

\bibliographystyle{amsplain}

\begin{thebibliography}{30}

\bibitem{Aubin0}  T. Aubin, \'{E}quations diff\'{e}rentielles non lin\'{e}aires et probl\`{e}me de Yamabe concernant
la courbure scalaire. \textit{J. Math. Pures Appl. (9)} \textbf{55} (1976), 269--296.

\bibitem{Bahri&Coron}
A. Bahri and J. M. Coron, On a nonlinear elliptic equation involving the critical
Sobolev exponent: The effect of the topology of the domain. \textit{Comm. Pure Appl.
Math.} \textbf{41} (1988), 253--294.


\bibitem{Bramanti} M. Bramanti and L. Brandolini, Schauder estimates for parabolic
nondivergence operators of H\"{o}rmander type. \textit{J.
Differential Equations} \textbf{234} (2007), 177--245.

\bibitem{Bramanti1}
M. Bramanti and L. Brandolini,
$L^p$ estimate for nonvariational hypoelliptic operators with VMO coefficients.
 \textit{Trans. Amer. Math. Soc.} \textbf{352} (2000), 781--822.


\bibitem{Brendle4} S. Brendle, Convergence of the Yamabe flow for arbitrary initial
energy. \textit{J. Differential Geom.} \textbf{69} (2005), 217--278.

\bibitem{Brendle5} S. Brendle, Convergence of the Yamabe flow in dimension 6 and higher. \textit{Invent.
Math.} \textbf{170} (2007),  541--576.



 \bibitem{Chang&Cheng} S. C. Chang and J. H. Cheng, The Harnack estimate for the Yamabe flow on CR manifolds
of dimension 3. \textit{Ann. Global Anal. Geom.} \textbf{21} (2002),
111--121.


\bibitem{Chang&Chiu&Wu} S. C. Chang, H. L. Chiu, and C. T. Wu,
The Li-Yau-Hamilton inequality for Yamabe flow on a closed CR 3-manifold.
 \textit{Trans. Amer. Math. Soc.} \textbf{362} (2010), 1681--1698.

 \bibitem{CCY}
S. Chanillo and H. L. Chiu, ans P. Yang,
Embeddability for 3-dimensional Cauchy-Riemann manifolds and CR Yamabe invariants,
\textit{Duke Math. J.} \textbf{161} (2012), 2909--2921.


\bibitem{Cheng&Chiu&Yang} J. H. Cheng, H. L. Chiu,  and P. Yang,
Uniformization of spherical CR manifolds. \textit{Adv. Math.} \textbf{255} (2014), 182--216.

\bibitem{Cheng&Malchiodi&Yang} J. H. Cheng, A. Malchiodi and P. Yang,
A positive mass theorem in three dimensional Cauchy-Riemann geometry. \textit{Adv. Math.}  \textbf{308}  (2017),
276--347.





\bibitem{Chow} B. Chow, The Yamabe flow on locally conformally flat manifolds
with positive Ricci curvature. \textit{Comm. Pure Appl. Math.}
\textbf{45} (1992), 1003--1014.

\bibitem{Citti} G. Citti, Semilinear Dirichlet problem involving critical exponent for the Kohn Laplacian,
\textit{Ann. Mat. Pura Appl. (4)} \textbf{169} (1995), 375--392.

\bibitem{Citti&Uguzzoni} G. Citti and F. Uguzzoni, Critical semilinear equations on the Heisenberg group: The
effect of the topology of the domain. \textit{Nonlinear Anal.} \textbf{46} (2001), 399--417.

\bibitem{Dragomir} S. Dragomir and G. Tomassini, Differential Geometry and Analysis
on CR Manifolds, Progress in Mathematics, \textbf{246}, Birkh\"{a}user
Boston, Boston, MA, (2006).

\bibitem{Folland}  G. Folland, The tangential Cauchy-Riemann complex on spheres. \textit{Trans. Amer. Math. Soc.}
\textbf{171} (1972), 83--133.



\bibitem{Folland&Stein} G. Folland and E. Stein, Estimates for the $\bar\partial
\sb{b}$ complex and analysis on the Heisenberg group. \textit{Comm.
Pure Appl. Math.} \textbf{27} (1974), 429--522.




\bibitem{Gamara2} N. Gamara, The CR Yamabe conjecture---the case $n=1$. \textit{J.
Eur. Math. Soc.} \textbf{3} (2001), 105--137.





\bibitem{Gamara1} N. Gamara and R. Yacoub, CR Yamabe conjecture---the
conformally flat case. \textit{Pacific J. Math.} \textbf{201}
(2001), 121--175.

\bibitem{Habermann} L. Habermann, Conformal metrics of constant mass.
\textit{Calc. Var. Partial Differential Equations}  \textbf{12} (2001) 259--279.

\bibitem{Habermann&Jost}
L. Habermann and J. Jost, Green functions and conformal geometry. \textit{J. Differential Geom.} \textbf{53} (1999) 405--433.

\bibitem{Hamilton} R. S. Hamilton, The Ricci flow on surfaces.
Contemp. Math., \textbf{71},  Amer. Math. Soc., Providence, RI, (1988), 237--262.



\bibitem{Ho2} P. T. Ho,
Result related to prescribing pseudo-Hermitian scalar
curvature. \textit{Int. J. Math.} \textbf{24} (2013), 29pp.



\bibitem{Ho} P. T. Ho, The long time existence and convergence of the CR
Yamabe flow. \textit{Commun. Contemp. Math.} \textbf{14} (2012), 50 pp.

\bibitem{Ho1} P. T. Ho,
The Webster scalar curvature flow on CR sphere. Part I. \textit{Adv. Math.} \textbf{268} (2015), 758--835.


\bibitem{Jerison&Lee1} D. Jerison and J. M. Lee, Extremals for the Sobolev inequality on
the Heisenberg group and the CR Yamabe problem. \textit{J. Amer.
Math. Soc.} \textbf{1} (1988), 1--13.

\bibitem{Jerison&Lee2} D. Jerison and J. M. Lee,
Intrinsic CR normal coordinates and the CR Yamabe problem.
\textit{J. Differential Geom.} \textbf{29} (1989), 303--343.

\bibitem{Jerison&Lee3} D. Jerison and J. M. Lee, The Yamabe problem on CR manifolds. \textit{J.
Differential Geom.} \textbf{25} (1987), 167--197.

\bibitem{Lee} J. M. Lee,
The Fefferman metric and pseudo-Hermitian invariants. \textit{Trans. Amer. Math. Soc.}  \textbf{296}  (1986),  411--429.

\bibitem{Lee&Parker} J. M. Lee and T. H. Parker, The Yamabe problem.
\textit{Bull. Amer. Math. Soc. (N.S.)} \textbf{17} (1987), 37--91.


\bibitem{Schoen} R. Schoen, Conformal deformation of a Riemannian metric to constant scalar curvature.
\textit{J. Differential Geom.} \textbf{20} (1984), 479--495.


\bibitem{Schwetlick&Struwe} H. Schwetlick and M. Struwe, Convergence of the
Yamabe flow for ``large" energies.  \textit{J. Reine Angew. Math.}
\textbf{562} (2003), 59--100.

\bibitem{Simon} L. Simon, Asymptotics for a class of non-linear evolution equations with applications to
geometric problems. \textit{Ann. of Math. (2)} \textbf{118} (1983), 525--571.

\bibitem{Struwe} M. Struwe, A global compactness result for elliptic boundary value problems involving
limiting nonlinearities. \textit{Math. Z.} \textbf{187} (1984), 511--517.

\bibitem{T}
N. Tanaka,
A differential geometric study on strongly pseudo-convex manifolds, Kinokuniya, (1975) Tokyo.


\bibitem{Trudinger} N. S. Trudinger, Remarks concerning the conformal deformation of Riemannian structures
on compact manifolds. \textit{Ann. Scuola Norm. Sup. Pisa (3)} \textbf{22} (1968), 265--274.



\bibitem{Wang} W. Wang, Canonical contact forms on spherical CR manifolds.
\textit{J. Eur. Math. Soc.}  \textbf{5}  (2003),   245--273.

\bibitem{W}
S. M. Webster,
Pseudohermitian structures on a real hypersurface,
\textit{J. Differential Geom.} \textbf{13} (1978), 25--41.

\bibitem{Yamabe} H. Yamabe, On a deformation of Riemannian structures on compact manifolds. \textit{Osaka
Math. J.} \textbf{12} (1960), 21--37.

\bibitem{Ye} R. Ye, Global existence and convergence of Yamabe flow.
\textit{J. Differential Geom.} \textbf{39} (1994), 35--50.

\bibitem{Zhang} Y. Zhang, The contact Yamabe flow. Ph.D. thesis, University of Hanover (2006).

\end{thebibliography}

\end{document}